\documentclass[english,11pt,oneside]{amsart}
\usepackage[foot]{amsaddr}
\usepackage[utf8]{inputenc}
\usepackage[T1]{fontenc}
\usepackage{lmodern}
\usepackage[expansion=false]{microtype}
\usepackage{array,tabularx}
\usepackage[english]{babel}
\usepackage[margin=2.6cm]{geometry}
\usepackage{amsfonts, amssymb, amscd, amsmath}
\usepackage{amsthm}
\usepackage{latexsym}
\usepackage{graphicx}
\usepackage{xcolor}
\usepackage{booktabs}
\usepackage[section]{placeins}
\usepackage{enumitem}
\usepackage{tikz}
\usepackage{tikz-cd}
\usetikzlibrary{decorations.markings}
\usepackage{algorithm}
\usepackage{algorithmic}
\usepackage[numbers]{natbib}
\usepackage[colorlinks=true,linkcolor=blue,citecolor=blue,urlcolor=blue]{hyperref}
\usepackage[capitalise,noabbrev]{cleveref}
\usepackage{amsthm}
\usepackage{thmtools}

\makeatletter
\renewcommand{\email}[2][]{%
  \ifx\emails\@empty\relax\else{\g@addto@macro\emails{,\space}}\fi%
  \@ifnotempty{#1}{\g@addto@macro\emails{\textrm{(#1)}\space}}%
  \g@addto@macro\emails{#2}%
}
\makeatother

\crefname{thm}{Theorem}{Theorems}
\Crefname{thm}{Theorem}{Theorems}
\crefname{prop}{Proposition}{Propositions}
\Crefname{prop}{Proposition}{Propositions}
\crefname{lem}{Lemma}{Lemmas}
\Crefname{lem}{Lemma}{Lemmas}
\crefname{cor}{Corollary}{Corollaries}
\Crefname{cor}{Corollary}{Corollaries}
\crefname{defin}{Definition}{Definitions}
\Crefname{defin}{Definition}{Definitions}
\crefname{con}{Construction}{Constructions}
\Crefname{con}{Construction}{Constructions}
\crefname{ex}{Example}{Examples}
\Crefname{ex}{Example}{Examples}
\crefname{rem}{Remark}{Remarks}
\Crefname{rem}{Remark}{Remarks}
\crefname{appendix}{Appendix}{Appendices}
\Crefname{appendix}{Appendix}{Appendices}

\theoremstyle{plain}
\newtheorem{thm}{Theorem}[section]
\newtheorem{prop}[thm]{Proposition}
\newtheorem{lem}[thm]{Lemma}
\newtheorem{cor}[thm]{Corollary}
\theoremstyle{definition}
\newtheorem{defin}[thm]{Definition}
\newtheorem{con}[thm]{Construction}
\newtheorem{ex}[thm]{Example}
\theoremstyle{remark}
\newtheorem{rem}[thm]{Remark}

\DeclareMathOperator{\Imm}{Im}
\DeclareMathOperator{\id}{id}
\DeclareMathOperator{\sgn}{sgn}

\DeclareMathOperator{\rk}{rk}

\DeclareMathOperator{\Cl}{Cl}
\DeclareMathOperator{\sk}{sk}
\DeclareMathOperator{\lk}{lk}
\DeclareMathOperator{\st}{st}
\DeclareMathOperator{\relint}{relint}
\DeclareMathOperator{\Int}{int}
\newcommand{\Flag}{\mathbf{Flag}}
\newcommand{\Flagref}{\mathbf{Flag}_{\mathrm{ref}}}
\newcommand{\FFaceng}{\mathbf{QF}_{\mathrm{ng}}}
\newcommand{\FFace}{\mathbf{QF}}
\newcommand{\FFacesh}{\mathbf{QF}_{\mathrm{sh,ng}}}
\newcommand{\GACat}{\mathbf{GACat}}

\newcommand{\ssSet}{\mathfrak{ssSet}}
\newcommand{\Sets}{\mathfrak{Set}}
\newcommand{\Topcat}{\mathbf{Top}_{\mathrm{fin}/}}
\newcommand{\Ch}{\mathbf{Ch}_{\geq 0}}
\newcommand{\Dd}{\mathcal{D}}
\newcommand{\Gg}{\mathcal{G}}

\newcommand{\bbot}{\boldsymbol{\bot}}
\newcommand{\genface}[2]{d_{#1}^{#2}}
\newcommand{\Gudhi}{\textsc{Gudhi}}

\title[A data structure for quotient flag complexes]{A data structure for quotient flag complexes}
\author[K.S.]{Konstantin Sorokin}
\address[K.S.]{HSE University, Saint Petersburg, Russian Federation}
\thanks{Corresponding author: K.~Sorokin, \texttt{ksorokin@hse.ru, mopsless7@gmail.com}}
\author[A.L.]{Aleksandr Levin}
\address[A.L.]{HSE University, Moscow, Russian Federation}
\author[M.B.]{Maxim Beketov}
\address[M.B.]{HSE University, Moscow, Russian Federation}
\author[A.A.]{Anton Ayzenberg}
\address[A.A.]{Noeon Research, Tokyo, Japan}
\date{September 2026}

\hypersetup{
 pdftitle={A data structure for quotient flag complexes},
 pdfauthor={Konstantin Sorokin, Aleksandr Levin, Maxim Beketov, Anton Ayzenberg},
 pdfsubject={Quotient flag complexes, local operations, and computational evaluation}
}

\begin{document}
\raggedbottom

\begin{abstract}

Vietoris--Rips filtrations, which are standard in topological data analysis (TDA), consist of flag complexes, and a simplex tree stores a flag complex without any attaching data, since every simplex is determined by its vertices. In this paper we ask what survives of this economy when a flag complex $K$ is divided by a subcomplex $A$, each connected component of $A$ being crushed to a point. Such a quotient is a CW complex whose cells are the simplices of $K\setminus A$, but their attaching maps are no longer implicit. A surviving edge may become a loop, a surviving triangle may be attached along a single edge, and two quotients can have the same $1$-skeleton, or even the same $2$-skeleton, and different homotopy types.

We show that for flag $K$ the face order of the quotient is strictly graded exactly when $A$ is flag, that in this case every cell has at most two collapsed facets, and that the surviving labelled cells are determined by those of dimension at most three, sharply. For $m$-flag pairs the sharp threshold is $2m+1$, and it drops to $m+2$ when $K$ is flag. The prescribed cells form a regular CW decomposition exactly when every component of $A$ is full in $K$, which we prove with an explicit convex model. These results justify the proposed \emph{QF-tree}: a cell table that stores, for each surviving simplex, its ordered list of $d+1$ facets with collapsed facets flagged, indexed by a trie of quotient-vertex words. For bounded dimension its size is linear in the number of surviving simplices plus the retained provenance, and we derive and verify a simple formula for the collapsed fraction above which it is smaller than the homotopy-equivalent cone model. Because a collapse changes the attaching data only on the closed star of $A$, the QF-tree can also be applied locally inside a simplex tree. For a ball-shaped $A$ in the sampled Vietoris--Rips regime the closed star is a thin shell, and the median compact budget is below the cone model at every sampled radius.

An accompanying library, modelled on \Gudhi, implements the QF-tree, its local variant, an editable layer with local quotient updates, gluing, disc attachment, induced maps, cup products, fundamental-group presentations and zigzag persistence, and controlled experiments separate the cost of maintaining a quotient from the cost of the algebra computed on it.

\end{abstract}
\maketitle

\section{Introduction}\label{sec:intro}

The simplex tree of Boissonnat and Maria~\citep{boissonnat2014simplex} is the workhorse of \Gudhi~\citep{gudhi}: a simplicial complex is a trie of increasing vertex words, one node per simplex. This works because an abstract simplicial complex has canonical attaching maps, which need not be stored. Flag complexes push the economy one step further: a flag complex is the clique complex of its $1$-skeleton, so even the list of simplices is redundant once the graph is known. Vietoris--Rips complexes are flag, and this is why the persistence pipeline of topological data analysis scales to large point clouds~\citep{chazal2013interleaved}.

This paper is about what happens to that economy when we divide a flag complex by a subcomplex. Given $A\subseteq K$, we crush each connected component of $A$ to its own point and keep every simplex of $K\setminus A$ as an open cell of the quotient. The operation arises naturally in practice. Two Vietoris--Rips complexes $K$ and $K'$ built on overlapping samples share the flag subcomplex $K\cap K'$ (\cref{lem:intersection-flag}), and one may want to compare $K/(K\cap K')$ with $K'/(K\cap K')$. A subcomplex known to be contractible can be collapsed to shrink a complex before further computation, in the spirit of the reductions of discrete Morse theory~\citep{forman2002user,kozlov2021organized,fernandez2026strong} but keeping the original cells rather than a homotopy-equivalent model. Finally, the relative homology $H_*(K,A)$ is canonically the homology $H_*(Q,D)$ of the componentwise quotient relative to its distinguished component points. It agrees with the absolute $H_n(Q)$ for $n\ge2$, and degrees zero and one require the correction in \cref{thm:chain}.

Deleting the cofaces of $A$ would be a different, easier operation: it removes the relations we intend to keep. A quotient keeps the relations and changes their attachments. A surviving edge whose endpoints lie in the same component of $A$ becomes a loop, a surviving triangle with two edges in $A$ is attached to the quotient along a single edge, and collapsing the boundary of a disc creates a sphere. The resulting CW decomposition is in general not regular, even when the underlying space is a ball, and it is no longer determined by its low-dimensional skeleta (\cref{prop:skeleta,prop:2skel}). So the two questions that decide whether a quotient can be stored as cheaply as a flag complex are:
\begin{enumerate}[label=(Q\arabic*),leftmargin=3em]
\item Which part of the simplex-tree structure survives? Is the face order still graded, and can the attaching maps still be recorded by codimension-one data only?
\item Which skeleton determines the quotient, so that higher-dimensional records can be dropped?
\end{enumerate}

\subsection*{Results}
The answer to both questions is governed by two properties of the collapsed subcomplex, flagness and fullness, which Lemma~\ref{lem:flag-subcomplexes} shows are different: a flag subcomplex of a flag complex is the clique complex of \emph{any} subgraph, and it need not contain all the edges of $K$ between its own vertices. A missing edge with both endpoints in one component of $A$ survives as a loop; a missing edge joining two components survives as an ordinary edge.

\begin{itemize}[leftmargin=2em]
\item For flag $K$, the face order of $K/A$ is strictly graded if and only if $A$ is flag (\cref{thm:strict}), and then every surviving cell has at most two collapsed facets (\cref{lem:budget}). The ordered list of codimension-one facets, with collapsed ones flagged, determines the characteristic maps (\cref{cor:skeletal}), and the alternating sum of the genuine facets is the relative cellular differential (\cref{thm:chain}). A QF-tree stores these lists, at $d+1$ references per $d$-cell, plus one point per collapsed component. Despite the name, the structure is no longer a tree.
\item A flag collapse is determined by its cells of dimension at most three, and three cannot be lowered (\cref{prop:rigidity,prop:2skel}). For pairs whose minimal non-faces have at most $m+1$ vertices the threshold is $2m+1$, sharp (\cref{thm:rigidity-m,prop:sharp-m}), and it drops to $m+2$ when the ambient complex is flag (\cref{thm:threshold-unified}). No such threshold exists for gluings (\cref{prop:glue-noskel}).
\item The prescribed cells of $K/A$ form a regular CW decomposition if and only if every component of $A$ is full in $K$ (\cref{thm:fullness}). For flag $A$ this is visible on the edges (\cref{cor:regular-edges}). The proof of sufficiency is an explicit convex model of a simplex with disjoint faces crushed (\cref{lem:faces}), which we prefer to an appeal to ball recognition because homology and collapsibility tests cannot decide the weaker question of whether the quotient is a ball (\cref{sec:ball-question}).
\item The quotient representation is compared against a homotopy equivalent alternative: attach a cone to each component of $A$ and store the result in a simplex tree. In storage units, the quotient is the smaller representation once the collapsed fraction $\alpha=|A|/|K|$ exceeds $(s-1)/(s+1)$, where $s$ is the mean number of facets of a surviving cell (\cref{sec:crossover}). We measured this on Vietoris--Rips, Erd\H{o}s--R\'enyi and high-dimensional clique complexes with up to $855{,}000$ simplices, and predicted and observed crossovers of representation cost agree to within $0.01$.
\item The collapse changes the attaching data only on the closed star of $A$, so the QF-tree can be built on that closed star alone while the rest of $K$ stays in the simplex tree (\cref{sec:local}). Whether this is smaller than the cone model depends on the shape of $A$ as well as on its size. In the sampled Vietoris--Rips regime the median compact budget for a ball-shaped $A$ is below the cone model at every sampled radius, although not for every seed. For a scattered vertex subset it is above, and a four-point example shows that localization alone does not decide the comparison.
\end{itemize}

\subsection*{Software}
The library \texttt{qfcore}, written in C++17 with a Python interface and modelled on \Gudhi's API, implements the immutable QF-tree with its word trie, the closed-star local variant, and an editable cell layer supporting successive quotients with induced maps, point gluing, polygonal disc attachment, $\mathbb F_2$ cohomology with cup products, finite presentations of $\pi_1$, and insertion/deletion zigzag persistence driven by \Gudhi's streaming zigzag engine~\citep{qfe070gudhizz}. \Cref{sec:implementation,sec:experiments} describe the design and the controlled experiments. The library is available via \texttt{pip install qfcore} and at \url{https://github.com/qfcore/qfcore}.

In the experiments, maintaining the quotient locally is one to three orders of magnitude cheaper than rebuilding it, and the gap grows with the size of the complex (a paired rebuild/local ratio of $115$--$125$ on a triangulated torus with $3456$ simplices and $370$--$475$ on one with $13\,824$ simplices). Once the algebra computed on the result is included, the saving becomes small. Computing a cohomology ring from a retained quotient is several times faster than from the source triangulation, the only consumer for which we made a matched comparison. For one-shot relative homology, a direct relative boundary matrix as computed by \Gudhi\ remains faster. 

\subsection*{Related work}
Bounded-skeleton reconstruction has precedents for triangulated manifolds~\citep{dancis1984} and for manifold subcomplexes of cubes~\citep{rowlands2021}. Our object is the labelled upper set $K\setminus A$ inside a simplicial face poset, and the parameter $m$ is the usual bound on the size of minimal non-faces~\citep{nevo2008missing,goff2011balanced,adamaszek2013extremal}. Quotients and incidence categories occur in poset and group-action settings~\citep{hultman2002quotient,stanley1984quotients,quinn2013incidence,williams2024survey}, and we use only the standard semisimplicial constructions~\citep{rourke1971delta,ebert2019semisimplicial}. The signed cover category of \cref{app:covers} is an acyclic category in the sense of~\citep[Chapter~10]{kozlov2008combinatorial}. Tanaka~\citep{tanaka2019strong} develops the homotopy theory of finite acyclic categories and $\Delta$-complexes under point reductions, and Huang~\citep{huang2023abstract} studies incidence structures on finite bounded acyclic categories. Generalized triangulations in Regina~\citep{burton2013regina,ReginaHandbook}, the face-paired tetrahedra of SnapPy~\citep{culler2025snappy} and cell-tuple structures~\citep{lienhardt1994,brisson1993} also retain local attaching data, and \cref{sec:complexity} identifies the class of presentations shared with the facet-pairing formats. Compact simplex-tree variants and the skeleton-blocker representation~\citep{boissonnat2017compact,attali2012skeleton} are the natural baselines for the source complexes themselves. Dey, Fan and Wang~\citep{dey2014simplicial} compute persistence along simplicial maps, including vertex collapses. A vertex collapse merges simplices with equal images, whereas a componentwise quotient keeps them as distinct cells. The two already differ on the boundary of a triangle with one edge collapsed: the vertex collapse leaves an edge, the quotient a circle. Libraries for learning on cell and combinatorial complexes, such as TopoX~\citep{hajij}, are potential consumers of quotient presentations.

\subsection*{Organization}
\Cref{sec:quotients} fixes the objects and shows what is lost. \Cref{sec:decoration} builds the decoration $\Dd(K,A)$, its realization and its chains. \Cref{sec:graded} proves strict gradedness, \cref{sec:skeleta} the skeletal determinacy, and \cref{sec:regularity} the regularity criterion. \Cref{sec:encoding} describes the QF-tree, \cref{sec:crossover} the comparison with the cone model, and \cref{sec:local} the closed-star application. \Cref{sec:implementation,sec:experiments} describe the library and the experiments, and \cref{sec:discussion} concludes. The appendices contain the vertex-minimality of flag spheres, the signed cover category and what it forgets, the graded forcing lemmas, the convex model and the ball question, the order-preserving and permuted gluing regimes with the lens-space and Klein-bottle examples, and the complete experimental tables.

\section{Quotients of flag complexes}\label{sec:quotients}

\subsection{Flag complexes and their subcomplexes}\label{sec:prelim}
Throughout, simplicial complexes are finite and abstract, on a totally ordered vertex set; every simplex carries the induced increasing order, and homology is taken with $\mathbb{Z}$ coefficients unless stated otherwise. We write $\sk_1 K$ for the $1$-skeleton and $\Cl(G)$ for the clique complex of a graph $G$.

\begin{defin}\label{def:flag}
A simplicial complex $K$ is a \emph{flag complex} if $K=\Cl(\sk_1 K)$: every set of vertices that is pairwise connected by edges of $K$ spans a simplex of $K$. Equivalently, every minimal non-face of $K$ has exactly two vertices.
\end{defin}

A flag complex is determined by its $1$-skeleton. Vietoris--Rips complexes are flag because they are clique complexes of the neighbourhood graph at a given scale.

\begin{lem}\label{lem:intersection-flag}
If $K$ and $K'$ are flag complexes on a common vertex set, then $L=K\cap K'$ is a flag complex.
\end{lem}
\begin{proof}
Let $\sigma$ be a clique of $\sk_1 L=\sk_1 K\cap\sk_1 K'$. Then $\sigma$ is a clique of $\sk_1 K$ and of $\sk_1 K'$, hence $\sigma\in K$ and $\sigma\in K'$ by flagness of each, so $\sigma\in L$.
\end{proof}

\begin{lem}\label{lem:flag-subcomplexes}
Let $K$ be a flag complex. The flag subcomplexes of $K$ are exactly the complexes $\Cl(G)$ for arbitrary subgraphs $G\subseteq\sk_1 K$. In particular a flag subcomplex need not be \emph{full} (induced): it may omit edges of $K$ between its own vertices.
\end{lem}
\begin{proof}
If $G\subseteq\sk_1 K$ then every clique of $G$ is a clique of $\sk_1 K$, hence a simplex of $K$ by flagness; so $\Cl(G)\subseteq K$, and $\Cl(G)$ is flag by construction. Conversely a flag subcomplex $L$ equals $\Cl(\sk_1 L)$ with $\sk_1 L\subseteq\sk_1 K$.
\end{proof}

Inside the square with a diagonal, the hollow $4$-cycle is a flag subcomplex that is not full. The distinction between flagness and fullness runs through the whole paper: flagness of $A$ governs the order of the quotient (\cref{sec:graded}), fullness governs its regularity (\cref{sec:regularity}), and neither prevents the collapse from changing the homotopy type.

\begin{defin}[Componentwise collapse]\label{def:sc}
For a finite simplicial complex $K$ and a subcomplex $A$, let $A_\alpha$ be the nonempty connected components of $A$. On $|K|$ set $x\sim y$ if and only if $x=y$ or $x,y$ lie in one $|A_\alpha|$. Write $K/_{SC}A=|K|/{\sim}$ and $D=\pi_0(A)$. When $A=\varnothing$, $D=\varnothing$ and the quotient is $|K|$. We write $K/A$ only when $A$ is nonempty and connected, or when a one-point collapse is explicitly meant. The same operation makes sense for finite CW pairs.
\end{defin}

Intersections $K\cap K'$ of flag complexes are typically disconnected, which is why we collapse componentwise rather than to a single point. The difference matters already in degree one (\cref{ex:relative-interval}).

\begin{defin}[Notation]\label{def:notation}
The point to which a collapsed component is sent is written $\bbot$, or $\bbot_\alpha$ for the component indexed by $\alpha$. The structure recording, for every surviving cell, its ordered list of facets with collapse flags is written $\Dd(K,A)$ and called the \emph{decoration}, built in \cref{con:D}. We write
\[
F(K,A)\;=\;K\smallsetminus A
\]
for the \emph{surviving set}: the cells of the quotient, basepoints excluded. For an ordered simplex $\sigma=(v_0<\dots<v_d)$ of $K$ we write
\[
 \genface{i}{K}\sigma\;=\;\sigma\smallsetminus\{v_i\}
\]
for its $i$-th \emph{simplicial} facet, dropping the superscript when the ambient complex is clear. This is always a simplex of $K$, and it must be distinguished from the decorated face map $\partial_i$: the two agree when the facet survives, while $\partial_i\sigma=\bbot_\alpha$ when $\genface{i}{K}\sigma$ lies in the collapsed component $A_\alpha$ (\cref{con:D}).
\end{defin}

\subsection{What the low skeleta do not see}
A flag complex is determined by its $1$-skeleton. A quotient flag complex is not, and we illustrate this now, returning to it in \cref{sec:skeleta}.

\begin{prop}[the $1$-skeleton does not determine a quotient flag complex]\label{prop:skeleta}
Consider two flag complexes on $V=\{1,2,3,4,5\}$ (\cref{fig:counterex_for_sceleta}):
\begin{itemize}
\item $K$ has edges $51,52,12,24,43,31$ and the $2$-cell $125$ (the only triangle of $\sk_1 K$ is $125$, so $K$ is flag);
\item $K'$ has edges $53,52,12,24,43,31$ and no $2$-cells ($\sk_1 K'$ is triangle-free).
\end{itemize}
Collapse the closed edge $13$, a flag subcomplex by \cref{lem:flag-subcomplexes}. Both quotients have the same $1$-skeleton, with vertices $\{\bbot,2,4,5\}$ and edges $5\bbot,\,52,\,\bbot 2,\,24,\,4\bbot$, but $K/13$ contains a $2$-cell, the image of $125$, and $K'/13$ does not.
\end{prop}

After the collapse the three quotient vertices $\bbot,2,5$ are pairwise joined, but whether they span a cell depends on which of $1$ and $3$ each edge came from, and the quotient $1$-skeleton has thrown that away. A faithful encoding must therefore store attaching data beyond dimension one.

\begin{figure}[htbp]
\centering
\begin{tikzpicture}[line join=round, line cap=round, scale=1.4]
\begin{scope}
  \coordinate [label=above:$1$] (a) at (1,1);
  \coordinate [label=below:$2$] (b) at (2,0);
  \coordinate [label=left:$3$] (c) at (0,1);
  \coordinate [label=above:$4$] (d) at (2,2);
  \coordinate [label=below:$5$] (e) at (1,0);
  \draw (a) -- (e) -- (b) -- (a);
  \draw (b) -- (d) -- (c);
  \draw (a) -- (c);
  \draw[fill=blue,fill opacity=.4] (a) -- (b) -- (e) -- cycle;
  \node at (1,2.3) {$K$};
\end{scope}
\begin{scope}[xshift=5cm]
  \coordinate [label=above:$1$] (a) at (1,1);
  \coordinate [label=below:$2$] (b) at (2,0);
  \coordinate [label=left:$3$] (c) at (0,1);
  \coordinate [label=above:$4$] (d) at (2,2);
  \coordinate [label=below:$5$] (e) at (1,0);
  \draw (c) -- (e) -- (b) -- (a) -- (c);
  \draw (b) -- (d) -- (c);
  \node at (1,2.3) {$K'$};
\end{scope}
\end{tikzpicture}
\caption{Two flag complexes whose collapses along the edge $13$ have identical $1$-skeleta but non-homeomorphic quotients: $K/13$ carries a $2$-cell, $K'/13$ does not.}
\label{fig:counterex_for_sceleta}
\end{figure}
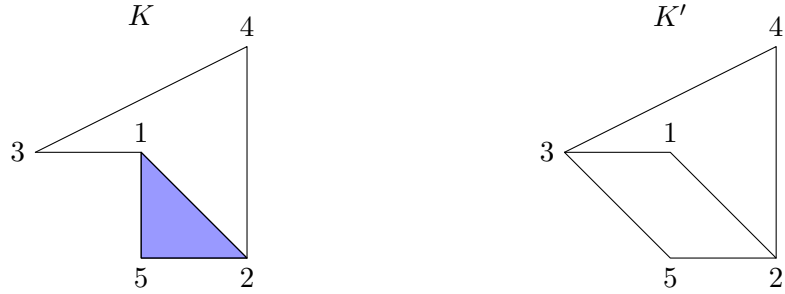

One more dimension does not help either. The next example is the sharpness witness for the threshold of \cref{prop:rigidity}.

\begin{prop}[the $2$-skeleton does not determine a quotient flag complex either]\label{prop:2skel}
On $V=\{1,2,3,4\}$ put
\begin{itemize}
    \item $K=\Delta^{3}$, $L=\Cl\bigl(\sk_1\Delta^3\smallsetminus\{34\}\bigr)=\overline{123}\cup\overline{124}$;
    \item $K'=\Cl\bigl(\sk_1\Delta^3\smallsetminus\{12\}\bigr)=\overline{134}\cup\overline{234}$, $L'=\Cl\bigl(\sk_1K'\smallsetminus\{34\}\bigr)$, the $4$-cycle $1\,3\,2\,4$.
\end{itemize}
Both are flag pairs. Their decorations agree in dimensions $\le2$, with the same cells, labels and face maps
\[
34\mapsto(\bbot,\bbot),\qquad 134\mapsto(34,\bbot,\bbot),\qquad 234\mapsto(34,\bbot,\bbot),
\]
and no genuine $0$-cells on either side. Yet $\Dd(K,L)$ carries the extra cell $1234\mapsto(234,134,\bbot,\bbot)$, and
\[
K/L\;\cong\;D^{3}/D^{2}\;\simeq\;*,\qquad K'/L'\;\cong\;D^{2}/\partial D^{2}\;\cong\;S^{2}.
\]
The two quotients are not even homotopy equivalent.
\end{prop}
\begin{proof}
$\Cl(\sk_1M\smallsetminus\{e\})$ consists of the simplices of $M$ not containing $e$, and is flag by \cref{lem:flag-subcomplexes}. Hence $F(K,L)$ is the set of simplices of $\Delta^3$ containing $34$, namely $\{34,134,234,1234\}$, while $F(K',L')$ is the set of simplices of $K'$ containing $34$, namely $\{34,134,234\}$; these agree in dimensions $\le2$. For the face maps, $\partial_i\sigma=\bbot$ iff $34\not\subseteq\genface{i}{}\sigma$, since a facet survives exactly when it still contains $34$. This reproduces the three lists displayed above, identically for both pairs: $\genface{0}{}(134)=34$ survives whereas $\genface{1}{}(134)=14$ and $\genface{2}{}(134)=13$ do not, and symmetrically for $234$; for the cell $1234$, present only on the left, $\genface{0}{}$ and $\genface{1}{}$ give $234$ and $134$ while $\genface{2}{}$ and $\genface{3}{}$ give $124$ and $123$. Topologically, $\lvert L\rvert$ is two triangles glued along $12$, a disc in $\partial\lvert K\rvert\cong S^2$, and collapsing a boundary disc of a ball yields a ball; whereas $\lvert K'\rvert$ is two triangles glued along $34$, a disc, whose boundary circle is $\lvert L'\rvert$.
\end{proof}

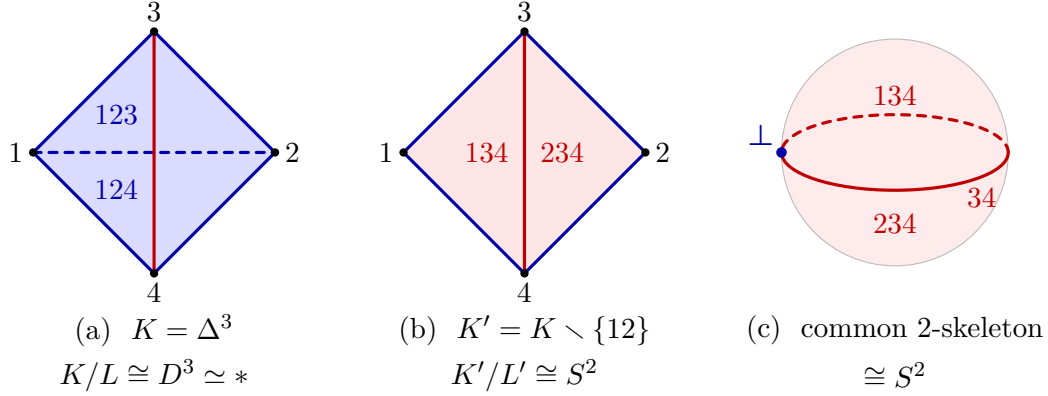
\begin{figure}[htbp]
\centering
\begin{tikzpicture}[line join=round, line cap=round,
  coll/.style={very thick, blue!65!black},
  surv/.style={very thick, red!75!black},
  vtx/.style={circle, fill=black, inner sep=1.1pt}]
\begin{scope}
  \coordinate (a1) at (-1.6,0);  \coordinate (a3) at (0,1.6);
  \coordinate (a2) at (1.6,0);   \coordinate (a4) at (0,-1.6);
  \fill[blue!14] (a1) -- (a3) -- (a2) -- (a4) -- cycle;
  \draw[coll] (a1) -- (a3) -- (a2) -- (a4) -- cycle;
  \draw[coll,dashed] (a1) -- (a2);
  \draw[surv] (a3) -- (a4);
  \foreach \p in {a1,a2,a3,a4} \node[vtx] at (\p) {};
  \node[left] at (a1) {$1$};  \node[above] at (a3) {$3$};
  \node[right] at (a2) {$2$}; \node[below] at (a4) {$4$};
  \node[blue!65!black] at (-0.5,0.5) {$123$};
  \node[blue!65!black] at (-0.5,-0.5) {$124$};
  \node at (0,-2.35) {(a)\ \ $K=\Delta^{3}$};
  \node at (0,-2.95) {$K/L\cong D^{3}\simeq *$};
\end{scope}
\begin{scope}[xshift=4.9cm]
  \coordinate (b1) at (-1.6,0);  \coordinate (b3) at (0,1.6);
  \coordinate (b2) at (1.6,0);   \coordinate (b4) at (0,-1.6);
  \fill[red!10] (b1) -- (b3) -- (b2) -- (b4) -- cycle;
  \draw[coll] (b1) -- (b3) -- (b2) -- (b4) -- cycle;
  \draw[surv] (b3) -- (b4);
  \foreach \p in {b1,b2,b3,b4} \node[vtx] at (\p) {};
  \node[left] at (b1) {$1$};  \node[above] at (b3) {$3$};
  \node[right] at (b2) {$2$}; \node[below] at (b4) {$4$};
  \node[red!75!black] at (-0.5,0) {$134$};
  \node[red!75!black] at (0.5,0) {$234$};
  \node at (0,-2.35) {(b)\ \ $K'=K\smallsetminus\{12\}$};
  \node at (0,-2.95) {$K'/L'\cong S^{2}$};
\end{scope}
\begin{scope}[xshift=9.8cm]
  \fill[red!8] (0,0) circle (1.5);
  \draw[gray!55] (0,0) circle (1.5);
  \draw[surv] (-1.5,0) arc (180:360:1.5 and 0.5);
  \draw[surv,dashed] (1.5,0) arc (0:180:1.5 and 0.5);
  \node[circle,fill=blue!65!black,inner sep=1.4pt] at (-1.5,0) {};
  \node[blue!65!black,above left=-1pt] at (-1.5,0) {$\bbot$};
  \node[red!75!black] at (0,0.75) {$134$};
  \node[red!75!black] at (0,-0.95) {$234$};
  \node[red!75!black] at (1.15,-0.6) {$34$};
  \node at (0,-2.35) {(c)\ \ common $2$-skeleton};
  \node at (0,-2.95) {$\cong S^{2}$};
\end{scope}
\end{tikzpicture}
\caption{\Cref{prop:2skel}. Blue: the collapsed subcomplex; red: the surviving cells. In (a) the tetrahedron is drawn with the edge $12$ dashed and the shading marks the collapsed disc $L=123\cup124$ hinged along it. In (b) the edge $12$ is absent, so $K'$ is the disc $134\cup234$ and $L'$ is its boundary $4$-cycle. Both quotients have the $2$-skeleton (c), one basepoint, the loop $34$ and two $2$-cells attached along it, but (a) fills it with a $3$-cell and (b) does not.}
\label{fig:2skeleta}
\end{figure}

It is natural to ask whether \emph{some} skeleton suffices. The answer, proved in \cref{sec:skeleta}, is yes for collapses, with threshold exactly three, and no for gluings.

Recall that the simplex tree rests on three properties: its nodes are in bijection with the simplices, each node stores the last vertex of its simplex, and the labels along a root-to-node path are the vertices of the simplex in increasing order. So the complex is built dimension by dimension, the face relation is a poset with no parallel attachments, and no attaching maps are stored. After a quotient we cannot keep all of this. The paper asks which of these properties survive, and under which hypotheses: strict gradedness together with facet-only attachment storage, and the skeleton dimensions that determine the structure.

\subsection{Homotopy type of a collapse}\label{sec:collapse}

\begin{prop}\label{prop:cone}
Let $A\subseteq X$ be a finite simplicial pair. There is a homotopy equivalence
\[ X/_{SC}A \;\simeq\; X\cup\bigsqcup_{\alpha}\operatorname{cone}(A_\alpha), \] where $\operatorname{cone}(A_\alpha)=(A_\alpha\times[0,1])/(A_\alpha\times\{1\})$ is attached to $X$ along $A_\alpha\times\{0\}=A_\alpha$.
\end{prop}
\begin{proof}
Write $Y=X\cup\bigsqcup_\alpha\operatorname{cone}(A_\alpha)$. Each cone is a contractible CW subcomplex, and there are finitely many disjoint cones; collapsing them one at a time is a homotopy equivalence~\cite[Proposition~0.17]{hatcher2002algebraic}. The quotient of $Y$ collapsing each cone to its apex is exactly $X/_{SC}A$.
\end{proof}

\begin{cor}\label{cor:pi1}
For a finite pair $A\subseteq X$ with $X$ connected,
\[ \pi_1(X/_{SC}A)\;\cong\;\pi_1(X)\Big/\big\langle\!\big\langle \bigcup_\alpha \Imm\big(\pi_1(A_\alpha)\to\pi_1(X)\big) \big\rangle\!\big\rangle , \] the quotient by the normal subgroup generated by the images of the components. (The images depend on basepaths only up to conjugacy, which the normal closure absorbs.)
\end{cor}
\begin{proof}
Apply van Kampen to $Y$ of \cref{prop:cone}: attaching the cone over the connected $A_\alpha$ kills the image of $\pi_1(A_\alpha)$.
\end{proof}

\begin{cor}[absolute and relative quotient homology]\label{cor:mv}
For a finite pair $A\subseteq K$ and $D=\pi_0(A)$, there is a natural long exact sequence
\[
 \cdots\to\bigoplus_\alpha H_n(A_\alpha)\to H_n(K) \to H_n(K/_{SC}A)\to\bigoplus_\alpha H_{n-1}(A_\alpha)\to\cdots \quad(n\ge2),
\]
whose low-degree end is
\[
\begin{split}
 \bigoplus_\alpha H_1(A_\alpha)\to H_1(K)\to H_1(K/_{SC}A)
 \to\bigoplus_\alpha H_0(A_\alpha)\\
\longrightarrow H_0(K)\oplus\mathbb Z[D]
 \longrightarrow H_0(K/_{SC}A)\to0.
\end{split}
\]
Moreover $H_*(K,A)\cong H_*(K/_{SC}A,D)$ in all degrees, and $H_n(K/_{SC}A)\cong H_n(K,A)$ for $n\ge2$.
\end{cor}
\begin{proof}
Apply Mayer--Vietoris to the cone model of \cref{prop:cone}, using CW neighbourhood thickenings of $K$ and of the union of the cones. Their intersection retracts onto $A$, and the cones have only $H_0=\mathbb Z[D]$. Relative cellular chains identify the pairs $(K,A)$ and $(K/_{SC}A,D)$, and the long exact sequence of the latter pair gives the degree restriction.
\end{proof}

A componentwise collapse can create homology in every dimension $n\ge2$, even from a contractible flag source. (It never creates connected components, and the map on $H_1$ is surjective by \cref{cor:mv}.)

We now consider quotients of flag spheres, giving two ways of obtaining minimal spherical CW flag quotient complexes as an illustration of extreme cases.  

\begin{thm}[classes created by a flag collapse]\label{thm:homotopy-change}
For each $n\ge2$ there is a contractible finite flag complex $X$ on $2n$ vertices with a connected flag subcomplex $A$ such that $X/_{SC}A\cong S^n$, and $2n$ is the smallest number of vertices for which this is possible.
\end{thm}
\begin{proof}
Let $B=(S^0)^{*(n-1)}$ be the boundary of the $(n-1)$-dimensional cross-polytope on $2n-2$ vertices, and put
\[
 X=\Delta^1*B,\qquad A=\partial\Delta^1*B=(S^0)^{*n}.
\]
By \cref{lem:join} and \cref{thm:cross-polytope}, $B$ and $A$ are flag spheres, $|B|\cong S^{n-2}$, $|A|\cong S^{n-1}$, and $X$ is flag as the join of two flag complexes. The join of an interval with $S^{n-2}$ is an $n$-ball with boundary $\partial\Delta^1*B$, so $|X|\cong B^n$ is contractible and $|A|=\partial|X|$, connected for $n\ge2$. Hence $X/_{SC}A=X/A\cong B^n/S^{n-1}\cong S^n$, on $2+2(n-1)=2n$ vertices.

For the lower bound let $X$ be contractible, $A\subseteq X$ connected and flag, and $X/A\cong S^n$. By \cref{cor:mv} and the long exact sequence of the pair, $\widetilde H_{n-1}(A)\cong H_n(X,A)\cong\widetilde H_n(X/A)\cong\mathbb Z$, so \cref{thm:2k} gives $|V(A)|\ge2n$ and hence $|V(X)|\ge 2n$.
\end{proof}

For $n=2$ the construction is the square with one diagonal modulo its boundary four-cycle, whose quotient has cell counts $(1,1,2)$ and Betti numbers $(1,0,1)$. The cone over the four-cycle, with five vertices, gives the same sphere with cells $(2,4,4)$. The vertex bound of \cref{thm:2k} is classical, and we recall it with a short proof in \cref{app:flag-spheres}, together with the extremal collapse family (octahedron modulo the disc complementary to an edge) that serves as a running example.

\subsection{Gluings}\label{sec:gluing-short}
A second way of identifying points of a flag complex is to glue equal-dimensional simplices along face-compatible affine maps. When the maps preserve the vertex orders this is a congruence of semisimplicial sets and the same cell records as for collapses represent it, with parallel facet occurrences and incidence numbers of absolute value greater than one now allowed (the torus, the Klein bottle and the M\"obius band are worked out in \cref{app:gluing}). Order-reversing gluings need the full vertex bijection per facet and are not semisimplicial quotients with the same cell inventory. The lens spaces $L(p,q)$ show that a sign per facet cannot replace the bijection. We keep the gluing regime in \cref{app:gluing} because the library implements only point gluing, and because, unlike collapses, gluings are determined by no finite skeleton (\cref{prop:glue-noskel}).

\section{The decoration, its realization, and its chains}\label{sec:decoration}

This section makes precise what a QF-tree stores and why this suffices. The construction is standard semisimplicial bookkeeping. Its only unusual feature is that the collapsed components are represented by explicit absorbing objects instead of by discarding degeneracies (\cref{rem:why-not-sset}).

\subsection{Semisimplicial sets with component points}\label{sec:machinery}
Let $\Delta_{inj}$ have the finite nonempty ordered sets $[n]$ as objects and increasing injections as maps. A semisimplicial set is a functor $X:\Delta_{inj}^{op}\to\Sets$, equivalently a graded set with face operators satisfying
\[
 \partial_i\partial_j=\partial_{j-1}\partial_i\qquad(i<j).
\]
The category $\ssSet$ is a presheaf category with levelwise colimits~\citep{maclane,rourke1971delta}, and realization $|X|=\int^{[n]}X_n\times\Delta_n^{top}$ preserves colimits; its CW cells are indexed by all semisimplicial simplices~\citep{rourke1971delta,ebert2019semisimplicial}. An ordered simplicial complex gives $K^\Delta$ by using its increasing vertex lists.

\begin{rem}[why the basepoint is represented explicitly]\label{rem:why-not-sset}
One could form the quotient in simplicial sets and discard degeneracies, but nondegenerate simplices of a simplicial set are not closed under faces, so this does not immediately give a semisimplicial object. We instead adjoin an explicit absorbing subobject for each component, as a matter of presentation.
\end{rem}

\begin{defin}[basepoint chains]\label{def:star}
Let $\star$ be the terminal semisimplicial set, with one simplex in every degree, and for a finite set $D$ put $\star_D=\coprod_{\alpha\in D}\star$. The representable $\Delta^0$ has no positive-dimensional simplices, so a positive-dimensional semisimplicial object cannot in general map to it; a map to $\star_D$ records which component absorbs each face.
\end{defin}

The realization $|\star|$ is not a point, but it is contractible. Its $1$-skeleton is one vertex with one loop $a$, and its unique $2$-cell is attached along $aaa^{-1}$, so that the $2$-skeleton is the dunce hat and the fundamental group is $\langle a\mid a\rangle=1$. Cells of higher dimension do not change the fundamental group. In positive degrees the cellular differential is multiplication by $\sum_{i=0}^{n}(-1)^i$, which is zero in odd degrees and the identity in even ones, so the reduced integral homology vanishes and the homological Whitehead theorem shows that $|\star|$ is contractible~\citep{hatcher2002algebraic}. The data structure stores only symbols for these chains, never their cells.

\begin{defin}[pointed face presentations]\label{def:fface}
For fixed finite $D$, let $\mathbf{QF}_D$ be the subcategory of $\star_D\downarrow\ssSet$ consisting of levelwise injective structure maps $\star_D\hookrightarrow\widetilde S$ with finitely many simplices outside the image. Write $S_n=\widetilde S_n\setminus(\star_D)_n$ for the \emph{genuine} cells. Let $\FFace$ be the total category with objects $(D,\widetilde S)$ and morphisms $(\rho,f)$, where $\rho:D\to D'$ and the semisimplicial map $f$ commute with the structure maps. Such an $f$ may send a genuine cell to a basepoint chain. The wide subcategory $\FFaceng$ consists of the morphisms sending every genuine cell to a genuine cell.
\end{defin}

\begin{lem}[unravelling the presentation]\label{lem:fface-unravel}
An object is equivalently a finite genuine graded set $S_n$, a finite set $D$, and face maps
\[
 \partial_i:S_n\longrightarrow S_{n-1}\sqcup D
\]
satisfying the semisimplicial identities after extending every $\bbot_\alpha\in D$ absorbingly in all degrees. A morphism is a degree-preserving map into the extended graded sets, with a map of component labels, that commutes with faces.
\end{lem}
\begin{proof}
The distinguished chains give $(\star_D)_n=D$ for every $n$, and all their face maps preserve the component label. Adjoining these formal chains to the displayed finite data, or deleting them from the structure-map image, gives the mutually inverse descriptions.
\end{proof}

This is the form in which the data structure stores an object: $d+1$ slots per $d$-cell, each holding either a genuine $(d-1)$-cell or a component label.

\begin{defin}[pointed realization]\label{def:real-pointed}
Set $|\widetilde S|_\bullet=|\widetilde S|\amalg_{|\star_D|}D$: every basepoint chain is collapsed to its own point. The target category $\Topcat$ has objects $(Y,D\to Y)$ with finite discrete $D$ and morphisms commuting with the maps of point sets.
\end{defin}

\begin{lem}[cells of the pointed realization]\label{lem:cells}
The space $|\widetilde S|_\bullet$ is a finite CW complex with $D\sqcup S_0$ as zero-cells and one $n$-cell for each $\sigma\in S_n$, $n>0$. The restriction of its characteristic map to local facet $i$ is the characteristic map of $\partial_i\sigma$ if genuine, and the constant map at $\bbot_\alpha$ otherwise.
\end{lem}
\begin{proof}
The distinguished chains realize to a CW subcomplex. Quotienting its components to distinct points leaves exactly the stated cells, and the skeletal adjunctions are the source semisimplicial adjunctions followed by this quotient.
\end{proof}

\begin{defin}[component-resolved shallowness]\label{def:shallow}
On $D\sqcup\coprod_n S_n$, let $\preceq$ be the reflexive transitive closure of all genuine facet relations and $\bbot_\alpha\prec\sigma$ when a facet equals $\bbot_\alpha$. Give every basepoint rank zero and each genuine cell its dimension. Put $B_\alpha(S)=\{\sigma:\bbot_\alpha\preceq\sigma\}$. The object is \emph{shallow} if each $\sigma\in B_\alpha(S)$ of dimension at least two has a genuine facet in $B_\alpha(S)$, for every $\alpha$. Write $\FFacesh$ for the shallow objects with genuine-preserving morphisms.
\end{defin}

\begin{lem}[where a cover can skip rank]\label{lem:skips}
Covers between genuine cells raise dimension by one. A basepoint $\bbot_\alpha$ covers $\sigma$ precisely when $\sigma\in B_\alpha(S)$ and no genuine facet of $\sigma$ lies in $B_\alpha(S)$. Therefore all covers raise rank by one exactly when $S$ is shallow.
\end{lem}
\begin{proof}
An iterated genuine face chain passes through each intermediate dimension; once it enters a basepoint it never leaves that label. If a genuine proper face of $\sigma$ lies above $\bbot_\alpha$, a genuine facet containing that face does too. Conversely such a facet is an intermediate object, excluding a cover from the basepoint.
\end{proof}

\subsection{Decoration, realization, and normalized chains}\label{sec:model}
\begin{defin}[categories of pairs]\label{def:flagcat}
Let $\Flag$ consist of finite ordered flag pairs $(K,A)$, with $A$ flag on its own vertex set. A morphism is a simplicial map $f:K\to K'$ that is increasing and injective on every simplex and satisfies $f(A)\subseteq A'$. Its wide subcategory $\Flagref$ requires additionally $f(\sigma)\in A'\iff\sigma\in A$. The construction below applies verbatim to nonflag pairs.
\end{defin}

\begin{con}[decoration]\label{con:D}
For $D=\pi_0(A)$, the component map defines $c:A^\Delta\to\star_D$. Set
\[
 \Dd(K,A)=K^\Delta\amalg_{A^\Delta}\star_D.
\]
Because $A^\Delta\hookrightarrow K^\Delta$ is levelwise injective, the distinguished chains remain distinct. The genuine cells are $K\setminus A$, and
\[
 \partial_i\sigma=
 \begin{cases}
  \genface{i}{K}\sigma,&\genface{i}{K}\sigma\notin A,\\
  \bbot_\alpha,&\genface{i}{K}\sigma\in A_\alpha.
 \end{cases}
\]
A pair map induces a map of the pushouts; genuine cells absorbed into $A'$ map into the corresponding basepoint chain. Thus $\Dd$ is functorial on $\Flag$, and genuine-preserving on $\Flagref$. For an order-preserving gluing datum (\cref{app:gluing}), $\Dd(K,\sim)$ is instead the levelwise coequalizer of the relation, with $D=\varnothing$.
\end{con}

\begin{thm}[realization]\label{thm:realization}
For every finite simplicial pair $A\subseteq K$, with $D=\pi_0(A)$,
\[
 |\Dd(K,A)|_\bullet\cong |K|/_{SC}|A|,
\]
naturally under the nondegenerate pair maps above. For order-preserving gluing data, $|\Dd(K,\sim)|\cong|K|/{\sim}$. Neither assertion requires flagness.
\end{thm}
\begin{proof}
Realization of the defining pushout is $|K|\amalg_{|A|}|\star_D|$; pasting with the pushout to $D$ gives $|K|\amalg_{|A|}D$, precisely the componentwise quotient. Realization preserves the defining coequalizer for an order-preserving congruence. Naturality follows from the universal properties and the induced map on components.
\end{proof}

\begin{rem}[skeletal attachments]\label{rem:skeletal-filtration}
Concretely: start at $D\sqcup S_0$ and attach one copy of $\Delta^n$ for every genuine $n$-cell, using on its $i$-th boundary facet the characteristic map of $\partial_i\sigma$ or the constant map at the recorded component point. The semisimplicial identities give agreement on codimension-two overlaps and hence on all iterated overlaps. Nothing here needs shallowness: a local facet may be mapped directly to a zero-cell.
\end{rem}

\begin{cor}[codimension-one constructibility]\label{cor:skeletal}
For a simplicial collapse or an order-preserving congruence, the ordered codimension-one targets, including component flags, determine the quotient's characteristic maps.
\end{cor}
\begin{proof}
Apply \cref{lem:cells} and induct on cell dimension.
\end{proof}

\begin{con}[normalized relative chains]\label{con:normalization}
Define $\mathcal N:\FFace\to\Ch$ by $\mathcal N_n(S)=\mathbb Z[S_n]$ and
\[
 \partial\sigma=\sum_{\substack{0\le i\le n\\\partial_i\sigma\text{ genuine}}} (-1)^i\partial_i\sigma
\]
for $n\ge1$, with $\partial=0$ in degree zero. On morphisms, a genuine generator maps to its image if genuine and to zero if absorbed. This is the quotient of the ordinary semisimplicial chain complex of $\widetilde S$ by the chain subcomplex of $\star_D$, so it is a functor and $\partial^2=0$.
\end{con}

\begin{thm}[relative and absolute cellular chains]\label{thm:chain}
There are natural identifications
\[
 \mathcal N\Dd(K,A)\cong C_*(K,A;\mathbb Z),\qquad H_*(\mathcal N\Dd(K,A))\cong H_*(Q,D;\mathbb Z), \quad Q=|K|/_{SC}|A|,
\]
and for every finite presentation a natural short exact sequence
\[
 0\longrightarrow\mathbb Z[D][0] \longrightarrow C_*^{CW}(|S|_\bullet) \longrightarrow\mathcal N(S)\longrightarrow0.
\]
In particular $H_n(Q)\cong H_n(K,A)$ for $n\ge2$, and if $D$ is a single point then $H_*(Q,D)\cong\widetilde H_*(Q)$.
\end{thm}
\begin{proof}
The genuine generators are the simplices outside $A$, and their differential is the source simplicial differential modulo $C_*(A)$. In the quotient CW structure the only discarded actual cells are $D$ in degree zero; positive-dimensional formal basepoint-chain cells do not remain. This proves the short exact sequence and identifies normalization with relative cellular chains of $(Q,D)$; the long exact sequence of that pair gives the degree range.
\end{proof}

\begin{ex}[the low-degree correction]\label{ex:relative-interval}
Let $K$ be the subdivided interval $1-2-3$ and $A=\{1\}\sqcup\{3\}$. The componentwise quotient is still an interval, so $H_1(Q)=0$, but $H_1(K,A)=\mathbb Z$, accounted for by the kernel of $H_0(D)\to H_0(Q)$. Keeping distinct component points therefore changes the topology: crushing all of $A$ to one point would give a circle here.
\end{ex}

\begin{prop}[quotient towers]\label{prop:quotient-tower}
For nested finite pairs $K_t\subseteq K_u$, $A_t\subseteq A_u$, there are canonical maps $K_t/_{SC}A_t\to K_u/_{SC}A_u$, composing in the parameter order. Their normalized relative chain maps send a surviving source simplex to itself if it still survives and to zero if it has entered $A_u$.
\end{prop}
\begin{proof}
Each component of $A_t$ lies in one component of $A_u$, so the composite $|K_t|\to |K_u|\to |K_u|/_{SC}|A_u|$ is constant on every class of the source quotient and factors uniquely; uniqueness proves the composition law. \Cref{con:normalization} gives the chain statement.
\end{proof}

These maps need not be genuine-preserving. This is the reason we did not build the theory on the face poset with signed covers: the signed cover quiver forgets which local facet an occurrence came from, and two order-preserving gluings of the same source with isomorphic signed quivers and identical chain complexes can have different fundamental groups (\cref{prop:signed-quiver-counterexample}). The nerve of the free cover category does not model the quotient even for a disc (\cref{prop:nerve}). \Cref{app:covers} records these facts and the restricted morphism class on which the cover functor does give a chain functor. The data structure stores the ordered facet lists, which is strictly more.

With all preparation done, we can move to addressing the questions we raised in the introductory \cref{sec:intro}. 

\section{Strict gradedness}\label{sec:graded}

The first question of the introduction is answered here. Under the face order of \cref{def:shallow}, a rank skip means a cell attached entirely to a basepoint.

\begin{lem}[skips are hollow simplices]\label{lem:skip-hollow}
For a finite simplicial pair and $\sigma\in K\setminus A$ of dimension at least two, $\bbot_\alpha\lessdot\sigma$ holds exactly when $\partial\sigma\subseteq A_\alpha$. Thus skips correspond to minimal nonfaces of $A$ of dimension at least two that belong to $K$.
\end{lem}
\begin{proof}
If the entire boundary lies in $A_\alpha$, no genuine proper face intervenes. Conversely suppose $\bbot_\alpha\lessdot\sigma$ and choose $v\in V(\sigma)\cap V(A_\alpha)$ from an iterated collapsed face. If some vertex $w$ is outside $V(A_\alpha)$, the edge $vw$ cannot belong to $A$; it extends to a genuine facet containing $v$, an intermediate object above $\bbot_\alpha$, a contradiction. All vertices of $\sigma$ therefore belong to $A_\alpha$, any genuine facet would again be intermediate, and the whole boundary lies in $A_\alpha$.
\end{proof}

\begin{thm}[strict gradedness for flag ambient complexes]\label{thm:strict}
Let $K$ be finite and flag, and $A\subseteq K$. The following are equivalent:
\begin{enumerate}[label=(\arabic*),leftmargin=2em]
\item $\Dd(K,A)$ is component-resolved shallow; equivalently its quotient face order is strictly graded;
\item $A$ is flag;
\item no surviving simplex of dimension at least two has its whole boundary in $A$.
\end{enumerate}
\end{thm}
\begin{proof}
If $A$ is flag, a surviving simplex cannot have every boundary edge in $A$, so \cref{lem:skip-hollow} rules out skips. For a direct component-resolved argument, start from $\bbot_\alpha\prec\sigma$ and choose $v\in V(A_\alpha)\cap V(\sigma)$. If an edge $vw$ is not in $A$, extend it to a genuine facet containing $v$. Otherwise every vertex of $\sigma$ lies in $A_\alpha$; flagness and $\sigma\notin A$ then provide a missing edge $ab$ of $A$ inside $\sigma$, and any facet containing it is genuine and lies above $\bbot_\alpha$, including the case where that facet is the edge $ab$ itself.

If $A$ is not flag, choose a minimal nonface $\tau$ on $V(A)$ with at least three vertices. Its edges lie in $A\subseteq K$, and flagness of $K$ forces $\tau\in K$. Its boundary is connected and lies in one component of $A$, and \cref{lem:skip-hollow} gives a skip.
\end{proof}

\begin{rem}\label{rem:archetype}
The archetype is $(\Delta^2,\partial\Delta^2)$: one basepoint and a two-cell attached entirely to it, quotient $S^2$, a cover skipping rank one. The filled triangle and its boundary have the same graph, so the obstruction is not visible from $\sk_1 A$. The graph produces clique candidates, and membership in $A$ decides which are missing. The data structure can also represent nonflag collapses, since its cell table is indexed by dimension, but the traversal by covers that strict gradedness allows is then unavailable, and the facet budget of next \cref{lem:budget} can fail.
\end{rem}

\section{Skeletal determinacy}\label{sec:skeleta}

We now answer the second question: how much of a quotient must be stored so that the rest can be recovered. Write $\Dd(K,L)_{\le n}$ for the truncation of the decoration to cells of dimension $\le n$, retaining the full basepoint set and all face maps between the retained cells. Two opposing forces meet in the next two lemmas, and the threshold is where they cross.

\begin{lem}[forcing]\label{lem:forcing}
Let $K$ be a flag complex, $L\subseteq K$ a subcomplex, and $\sigma\subseteq V(K)$ a vertex set at least three of whose facets $\sigma\smallsetminus\{v\}$ lie in surviving set $F(K,L)$. Then $\sigma\in F(K,L)$.
\end{lem}
\begin{proof}
Let $\tau_i=\sigma\smallsetminus\{v_i\}\in K\smallsetminus L$ for three distinct $v_1,v_2,v_3\in\sigma$. Every edge $\{a,b\}\subseteq\sigma$ omits at least one of $v_1,v_2,v_3$, hence lies in the corresponding $\tau_i\subseteq K$. So $\sigma$ is a clique of $\sk_1K$ and $\sigma\in K$ by flagness. If $\sigma\in L$ then $\tau_1\in L$ by face-closedness, a contradiction.
\end{proof}

\begin{lem}[facet budget]\label{lem:budget}
Let $L\subseteq K$ be a \emph{flag} subcomplex ($K$ arbitrary) and $\sigma\in F(K,L)$ of dimension $d\ge1$, with vertices $v_0<\dots<v_d$. Put
\[
N_1(\sigma)=\{\rho\subseteq\sigma\;:\;1\le|\rho|\le2,\ \rho\notin L\},
\]
the vertices and edges of $\sigma$ that survive the collapse. Then
\begin{enumerate}[label=(\roman*)]
\item $N_1(\sigma)\neq\varnothing$;
\item $\partial_i\sigma=\bbot$ (equivalently, by \cref{con:D}, the simplicial facet $\genface{i}{}\sigma=\sigma\smallsetminus\{v_i\}$ lies in $L$) if and only if $v_i\in\rho$ for every $\rho\in N_1(\sigma)$;
\item hence $\sigma$ has exactly $\bigl|\bigcap N_1(\sigma)\bigr|\le2$ collapsed facets and at least $d-1$ genuine ones.
\end{enumerate}
The bound $2$ is attained in every dimension.
\end{lem}
\begin{proof}
(i) If some vertex $v$ of $\sigma$ lies outside $V(L)$ then $\{v\}\in N_1(\sigma)$. Otherwise $\sigma\subseteq V(L)$ and $\sigma\notin L$, so $\sigma$ contains a minimal non-face of $L$, which has at most two vertices by flagness.

(ii) If $\genface{i}{}\sigma\in L$ then every $\rho\subseteq\sigma\smallsetminus\{v_i\}$ with $|\rho|\le2$ lies in $L$, so every member of $N_1(\sigma)$ contains $v_i$. Conversely, if every member of $N_1(\sigma)$ contains $v_i$, then every vertex of $\sigma\smallsetminus\{v_i\}$ lies in $V(L)$ and every edge in $\sk_1L$, so $\sigma\smallsetminus\{v_i\}$ is a clique of $\sk_1L$ and flagness gives $\genface{i}{}\sigma\in L$.

(iii) The collapsed facets are indexed by $\bigcap N_1(\sigma)$, the intersection of a nonempty family of sets of cardinality at most two. For attainment take $K=\Delta^{d}$ and $L=\Cl\bigl(\sk_1\Delta^{d}\smallsetminus\{v_{d-1}v_d\}\bigr)$: then $N_1(\sigma)=\{v_{d-1}v_d\}$ and the collapsed facets of $\Delta^d$ are $\genface{d-1}{}\sigma$ and $\genface{d}{}\sigma$.
\end{proof}

Forcing needs three surviving facets, and the budget guarantees $d-1$ of them. A $d$-cell can therefore be hidden from the cells below it only if $d-1\le2$, and from dimension four upwards every cell is implied by its own facets. This argument proves the next proposition and locates the sharp case at $d=3$, which is \cref{prop:2skel}.

\begin{prop}[flag collapses are labelled $3$-determined]\label{prop:rigidity}
Let $(K,A)$ and $(K',A')$ be flag pairs on a common ordered vertex set. If their surviving sets agree through dimension three, then $K\setminus A=K'\setminus A'$ in every dimension. If in addition a bijection of their component-point sets identifies the decorated truncations through dimension three, it extends uniquely to equality of the full decorations and gives a homeomorphism of the quotient CW presentations.
\end{prop}
\begin{proof}
Induct on $d\ge4$. A surviving $d$-simplex has at least $d-1\ge3$ genuine facets by \cref{lem:budget}; the induction hypothesis and \cref{lem:forcing} force the same simplex in the other pair, and symmetry gives equality of the surviving sets.

For a collapsed facet $\tau\subset\sigma$ in dimension $\ge3$, choose $u\in V(\tau)$. Since $\sigma\notin A$ and $A$ is flag, $\sigma$ contains a nonface $\rho$ of $A$ with one or two vertices. The face $\eta=\rho\cup\{u\}$ survives, has dimension at most two, and its iterated face at $u$ is recorded in the truncation as the component containing $u$, hence containing $\tau$. Genuine targets are already fixed by the surviving set.
\end{proof}

The component assignment is additional data: on vertices $0,1,2,3$, the pairs $(01\cup23\cup03,\ 01\cup23)$ and $(\text{path }0123\cup03,\ \text{path }0123)$ are flag, have the single surviving edge $03$, and their quotients are an interval and a loop respectively.

\subsection{Larger minimal non-faces}
The argument only used that the minimal non-faces of $L$ are small. Let $m\ge1$. A complex is \emph{$m$-flag} if every minimal non-face has at most $m+1$ vertices; it is then determined by its $m$-skeleton, and $1$-flag means flag. This is the class $\mathcal F_{m+1}$ of the face-number literature~\citep{nevo2008missing,goff2011balanced,adamaszek2013extremal}. The facet budget becomes $d-m$ genuine facets (\cref{lem:budget-m}) and forcing needs $m+2$ of them when $K$ is only $m$-flag (\cref{lem:forcing-m}); both proofs are in \cref{app:graded}.

\begin{thm}[$m$-flag collapses are $(2m{+}1)$-determined]\label{thm:rigidity-m}
Let $(K,L)$ and $(K',L')$ be pairs of $m$-flag complexes with $m$-flag subcomplexes, on a common ordered vertex set, whose surviving sets agree in dimensions $\le 2m+1$. Then $F(K,L)=F(K',L')$ in every dimension, and with matching component-point data the decorations agree. The bound is sharp, and $m=1$ is \cref{prop:rigidity}.
\end{thm}
\begin{proof}
Induct on $d\ge 2m+2$. A surviving $d$-simplex has at least $d-m\ge m+2$ genuine facets by \cref{lem:budget-m}, of dimension $d-1\ge 2m+1$, hence in $F(K',L')$ by induction or by hypothesis; \cref{lem:forcing-m} inside $(K',L')$ gives $\sigma\in F(K',L')$. Exchange the pairs for the reverse inclusion. A collapsed facet is read off, as in \cref{prop:rigidity}, from a surviving face $\rho\cup\{u\}$ of dimension at most $m+1\le 2m+1$.
\end{proof}

\begin{thm}[unified threshold for flag ambient complexes]\label{thm:threshold-unified}
Let $K,K'$ be flag on a common vertex set, $L\subseteq K$ and $L'\subseteq K'$ subcomplexes all of whose missing faces have at most $m+1$ vertices, with $\dim L,\dim L'\le\ell$. Put $t=\min(m+2,\ \ell+1)$. Then agreement of $F(K,L)$ and $F(K',L')$ in dimensions $\le t$ implies agreement in all dimensions.
\end{thm}
\begin{proof}
Both routes are the induction of \cref{prop:rigidity}, started one dimension above $t$, and both end in the flag forcing \cref{lem:forcing}, which needs three facets because $K,K'$ are flag. For $d\ge m+3$, \cref{lem:budget-m} supplies $d-m\ge3$ genuine facets of dimension $d-1\ge m+2$. For $d\ge\ell+2$ every facet has dimension $>\ell$ and is genuine, so all $d+1\ge3$ are available.
\end{proof}

The two thresholds do not contradict each other: $m$ bounds how much of $L$ is invisible from below, $\ell$ bounds how much room a cell has to hide in, and flagness of $K$ controls how cheaply a cell is forced. One family witnesses the sharpness of both.

\begin{prop}[sharpness]\label{prop:sharp-m}
Fix $p,q\ge2$, let $V=P\sqcup Q$ with $|P|=p$, $|Q|=q$, and set
\[
K=2^{V},\quad L=\{\tau\subseteq V: Q\not\subseteq\tau\},\qquad K'=\{\tau\subseteq V: P\not\subseteq\tau\},\quad L'=\{\tau\in K': Q\not\subseteq\tau\}.
\]
Then $K$ is flag, $K'$ is $(p-1)$-flag and flag precisely when $p=2$, $L$ is $(q-1)$-flag, $L'$ is $(\max\{p,q\}-1)$-flag, $\dim L=p+q-2$ and $\dim L'=p+q-3$. The decorations $\Dd(K,L)$ and $\Dd(K',L')$ agree in all dimensions $\le p+q-2$ and differ in dimension $p+q-1$, where the first has one further cell; $K/L\simeq*$ while $K'/L'\cong S^{\,p+q-2}$. Taking $p=q=m+1$ shows that $2m+1$ in \cref{thm:rigidity-m} is sharp; taking $p=2$, $q=m+1$ shows that $m+2$ in \cref{thm:threshold-unified} is sharp; and $m=1$ is \cref{prop:2skel}.
\end{prop}
\begin{proof}
$K=2^V$ has no missing face. The missing faces are $\{P\}$ for $K'$, $\{Q\}$ for $L$, and $\{P,Q\}$ for $L'$; equivalently $L=\Delta(P)*\partial\Delta(Q)$ and $L'=\partial\Delta(P)*\partial\Delta(Q)$, which gives the dimensions. Since $\tau\in K'$ iff $P\not\subseteq\tau$, we get $F(K,L)=\{Q\cup S: S\subseteq P\}$ and $F(K',L')=\{Q\cup S: S\subsetneq P\}$, with $\dim(Q\cup S)=q-1+|S|$; the two agree except at $S=P$. The face maps agree: for $S\subsetneq P$ the facet $(Q\cup S)\smallsetminus\{v\}$ omits $Q$ exactly when $v\in Q$, and lies in $K'$ in every case, so both decorations send $\partial_i\mapsto\bbot$ precisely at the positions indexed by $Q$. Both $L$ and $L'$ are connected for $q\ge2$. Topologically, $L$ is a cone with apex any $v\in P$, so $K/L\simeq*$; and $S\mapsto Q\cup S$ identifies $C_*(K',L')$ with the augmented chain complex of $\partial\Delta^{P}$ shifted by $q$, so $\widetilde H_*(K'/L')=\mathbb Z$ in degree $p+q-2$, while $|K'|$ is the union of $p$ facets of $\partial\Delta^V\cong S^{p+q-2}$, a ball with boundary complex $L'$.
\end{proof}

\begin{cor}[truncated storage]\label{cor:store}
For an $m$-flag pair whose cells record their vertex supports, as the QF-tree does, the cells of dimension $>2m+1$ need not be stored: they and their facet lists are recovered by iterating \cref{lem:forcing-m} over candidate vertex sets. For flag pairs the retained range is $\dim\le3$.
\end{cor}

Two remarks on scope. First, the theorems reconstruct the \emph{labelled} decoration on a fixed vertex set, and the source pair is not recovered. We do not know in general whether the label-free truncation $\Dd(K,L)_{\le3}$ determines $\Dd(K,L)$ up to isomorphism in $\FFace$: two genuine facets meeting in a ridge that lies in $L$ both report $\bbot$, and the truncation cannot see that they meet (\cref{lem:forcing-intrinsic} gives a local label-free test when the intersections are known). The negative results \cref{prop:skeleta,prop:2skel,prop:sharp-m} are label-free as stated, so sharpness holds in either reading. Second, the truncation saves storage only when cells exceed the bound: on Vietoris--Rips complexes capped at dimension three it saves nothing, while on uncapped clique complexes of dimension ten and above it discards most of the records.

\subsection{Gluings have no threshold}
\begin{prop}[gluings are determined by no finite skeleton]\label{prop:glue-noskel}
For every $n\ge1$ there is a flag complex $K$ and order-preserving gluing data $\sim_1,\sim_2$ on $K$ with $\Dd(K,\sim_1)_{\le n-1}=\Dd(K,\sim_2)_{\le n-1}$ while $\lvert K\rvert/{\sim_1}$ and $\lvert K\rvert/{\sim_2}$ are not homotopy equivalent.
\end{prop}
\begin{proof}
Let $K$ be two disjoint $n$-simplices $\sigma=(v_0<\dots<v_n)$ and $\sigma'=(w_0<\dots<w_n)$, a flag complex whose minimal non-faces are the cross pairs. Let $\sim_1$ identify each proper face $\{v_{i_0},\dots,v_{i_k}\}$ with $\{w_{i_0},\dots,w_{i_k}\}$ by $v_i\mapsto w_i$, and let $\sim_2$ be $\sim_1$ together with $\sigma\sim\sigma'$. Both satisfy \cref{def:gluing}; condition (ii) there is a one-way implication and does not force $\sigma\sim\sigma'$ from its facets. The quotients agree in dimensions $<n$, whereas $\lvert K\rvert/{\sim_1}\cong S^{n}$ and $\lvert K\rvert/{\sim_2}\cong D^{n}$.
\end{proof}

The reason is structural. A collapse is specified by a subcomplex $L$, which for flag $L$ is one-dimensional data, so only bounded information can hide above the visible range. A gluing datum is an equivalence relation on cells of every dimension, constrained only downwards, and identifications in the top dimension are free choices invisible below. The summary is:
\begin{center}
\begin{tabular}{@{}lll@{}}
\toprule
regime & determined by & sharpness\\
\midrule
collapse, $L$ flag & cells of $\dim\le3$ & \cref{prop:2skel}\\
collapse, $L$ $m$-flag & cells of $\dim\le 2m+1$ & \cref{prop:sharp-m}\\
collapse, $K$ flag, $\dim L\le\ell$ & cells of $\dim\le\ell+1$ & $\ell\le2$: \cref{prop:skeleta,prop:2skel}\\
gluing datum & no finite skeleton & \cref{prop:glue-noskel}\\
\bottomrule
\end{tabular}
\end{center}
The third line is \cref{prop:rigidity-dim} of the appendix; its common-vertex-set hypothesis cannot be dropped (\cref{rem:dichotomy-skel}).

\section{Regularity, and the role of fullness}\label{sec:regularity}

Flagness of $L$ governs the \emph{order} of the quotient. A different condition, independent of flagness, governs whether the prescribed cells form a regular CW decomposition, and it is the one \cref{lem:flag-subcomplexes} warned about: fullness. Throughout, $q\colon\lvert K\rvert\to K/_{SC}A$ is the quotient map, each $\lvert A_\alpha\rvert$ crushed to its own point $p_\alpha$, and $\mathcal E(K,A)$ is the prescribed open-cell decomposition of \cref{lem:cells}: the points $p_\alpha$ together with $e_\sigma=q(\relint\lvert\sigma\rvert)$ for $\sigma\in K\smallsetminus A$.

\begin{defin}\label{def:regular}
$\mathcal E(K,A)$ is \emph{regular} if its open cells admit characteristic maps that are homeomorphisms from closed balls onto their closures. The characteristic maps are \emph{not} required to be the inherited restrictions $q|_{\lvert\sigma\rvert}$.
\end{defin}

Crushing one edge of a triangle gives a regular bigon even though the inherited $2$-cell map is not injective, whereas crushing two adjacent edges leaves a $2$-ball whose prescribed cells are irregular, because the surviving edge becomes a loop. For a vertex set $W$ put $K[W]=\{\sigma\in K:V(\sigma)\subseteq W\}$. A subcomplex $C$ is \emph{full} in $K$ if $C=K[V(C)]$.

\begin{lem}[closed-cell reduction]\label{lem:closed-cell}
A finite CW decomposition of a space is regular if and only if every open $r$-cell $e$ satisfies $(\overline e,\overline e\smallsetminus e)\cong(B^r,S^{r-1})$. For the quotient decomposition and every surviving $\sigma$,
\begin{equation}
 (\overline{e_\sigma},\ \overline{e_\sigma}\smallsetminus e_\sigma)
 =\bigl(q(\lvert\sigma\rvert),\ q(\lvert\partial\sigma\rvert)\bigr).
 \label{eq:closed-cell-pair}
\end{equation}
\end{lem}
\begin{proof}
A regular characteristic map is a homeomorphism of the displayed ball pairs; conversely such a homeomorphism is a characteristic map for the same open cell, whose boundary already lies in the lower skeleton. For~\eqref{eq:closed-cell-pair}: $q$ is injective on $\relint\lvert\sigma\rvert$ and identifies no point of it with a boundary point, $\lvert\sigma\rvert$ is compact and the quotient Hausdorff, so $q(\lvert\sigma\rvert)$ is closed and equals $\overline{e_\sigma}$ by density of the interior.
\end{proof}

The sufficiency half of the criterion rests on an explicit convex model, which we state here and prove in \cref{app:convex}. Hersh~\citep{Hersh} gives general criteria for recognizing a regular CW complex from its characteristic maps. For componentwise collapses the convex model suffices.

\begin{lem}[disjoint faces can be crushed without changing the ball pair]\label{lem:faces}
Let $\Delta$ be a $d$-simplex, $d\ge1$, and $F_1,\dots,F_s$ a possibly empty family of pairwise disjoint nonempty proper faces. Let $q$ crush each $F_i$ to its own point. Then $\bigl(\Delta/\{F_i\},\ q(\partial\Delta)\bigr)\cong(B^d,S^{d-1})$, with $q(\relint\Delta)$ corresponding to the interior of the ball.
\end{lem}

\begin{thm}[full-component criterion]\label{thm:fullness}
For any finite simplicial pair $A\subseteq K$, the decomposition $\mathcal E(K,A)$ is regular if and only if every component of the collapsed subcomplex is full in $K$:
\[
 A_\alpha=K[V(A_\alpha)]\qquad\text{for every component }A_\alpha .
\]
No flag hypothesis is required, on $K$ or on $A$.
\end{thm}
\begin{proof}
Suppose some $A_\alpha$ is not full and choose $\sigma\in K[V(A_\alpha)]\smallsetminus A_\alpha$. If $\sigma$ lay in $A$ it would lie in one component, and having a vertex in $V(A_\alpha)$ that component would be $A_\alpha$; so $\sigma$ survives. Choose an inclusion-minimal surviving $\tau$ supported on $V(A_\alpha)$. Its dimension is $r\ge1$, every proper face of $\tau$ lies in $A_\alpha$ by minimality and the component argument, and so the only identifications inside $\lvert\tau\rvert$ crush its entire boundary to one point: $\overline{e_\tau}\cong\lvert\tau\rvert/\lvert\partial\tau\rvert\cong S^r$, which is not a closed ball. By \cref{lem:closed-cell} no reparametrization of the same open cell can repair this.

Conversely let every component be full. For a surviving $\sigma$ of positive dimension and each $\alpha$ with $\lvert\sigma\rvert\cap\lvert A_\alpha\rvert\neq\varnothing$, every subset of $V(\sigma)\cap V(A_\alpha)$ spans a face of $K$ with vertices in $V(A_\alpha)$, hence a face of $A_\alpha$ by fullness; conversely a point of the intersection lies in the relative interior of a unique face of $\sigma$, which must lie in $A_\alpha$. So $\lvert\sigma\rvert\cap\lvert A_\alpha\rvert$ is a proper face of $\sigma$, and faces for different components are disjoint because the $V(A_\alpha)$ are. \Cref{lem:faces} applied to $\lvert\sigma\rvert$ gives a homeomorphism $(q(\lvert\sigma\rvert),q(\lvert\partial\sigma\rvert))\cong(B^{\dim\sigma},S^{\dim\sigma-1})$ carrying the open ball onto $e_\sigma$, which by~\eqref{eq:closed-cell-pair} is a regular characteristic map.
\end{proof}

The necessity argument locates the \emph{minimal} missing simplex because that is where the first offending closure appears: a top-dimensional closure may be a ball while a lower-dimensional one is a sphere.

\begin{cor}[flag case: regularity is visible on edges]\label{cor:regular-edges}
Assume $A$ is flag. The following are equivalent:
\begin{enumerate}[label=(\roman*)]
\item $\mathcal E(K,A)$ is regular;
\item every component of $A$ is full in $K$;
\item every edge $uv\in K$ with $u,v\in V(A_\alpha)$ lies in $A$;
\item no surviving edge of the quotient has both endpoints at the same collapsed point.
\end{enumerate}
\end{cor}
\begin{proof}
(i)$\Leftrightarrow$(ii) is \cref{thm:fullness}, and fullness gives the edge condition. Conversely, if every edge of $\sigma\in K$ with vertices in $V(A_\alpha)$ lies in $A$, those edges lie in $A_\alpha$, so $\sigma$ is a clique of $\sk_1A_\alpha$ and flagness of $A$ puts $\sigma$ in $A_\alpha$. (iii)$\Leftrightarrow$(iv) is the definition of a surviving loop edge.
\end{proof}

This is the practical form: compute the components of $\sk_1A$ and scan $E(K)\smallsetminus E(A)$ for an edge with both endpoints in one component, in time $O(|V|+|E(A)|+|E(K)|)$ by a breadth-first search on $\sk_1A$ (union--find gives the same in near-linear time). The test is not monotone in a filtration: a new path can turn a surviving edge into a loop, and a later scale can absorb that edge into the collapsed core. The flag hypothesis cannot be dropped from the edge test: for $A=\partial\Delta^2\subset\Delta^2$ no edge survives, yet the $2$-cell has closure $S^2$. \Cref{thm:fullness} detects it through the minimal missing $2$-simplex.

\begin{prop}[regularity implies strict gradedness]\label{prop:reg-implies-shallow}
For a finite simplicial pair $A\subseteq K$, regularity of the prescribed quotient cells implies shallowness of $\Dd(K,A)$. The converse fails.
\end{prop}
\begin{proof}
By \cref{thm:fullness} every component is full. If $\bbot_\alpha\prec\sigma$ and $\dim\sigma\ge2$, choose $v\in V(\sigma)\cap V(A_\alpha)$; some vertex $w$ of $\sigma$ lies outside $V(A_\alpha)$, else fullness puts $\sigma$ in $A$. The edge $vw$ survives; extend it to a facet of $\sigma$, a genuine facet above $\bbot_\alpha$. The boundary arc $ab\cup bc$ in the triangle $abc$ gives a shallow but nonregular quotient.
\end{proof}

So for flag $K$ the two hypotheses divide the labour cleanly:
\[
\begin{gathered}
 L\text{ flag}\ \Longleftrightarrow\ \text{strict gradedness};\qquad
 \text{component fullness}\ \Longleftrightarrow\ \text{regular prescribed cells},
\end{gathered}
\]
with regular $\Rightarrow$ strictly graded and no implication back. A regular prescribed decomposition can be reconstructed from its face poset by Bj\"orner's theorem~\citep{bjorner1995topological}. When fullness fails that theorem no longer applies, which is another reason to keep the ordered facet lists. \Cref{app:convex} also explains why the weaker question, whether the quotient of a ball is a ball, cannot be settled by homology, acyclicity or collapsibility: it concerns the cellularity of the crushed sets in the boundary sphere, which is why \cref{thm:fullness} is proved with an explicit model.

Fullness has one more consequence that we will need in \cref{sec:crossover}: it is exactly the condition under which the cone model of \cref{prop:cone} remains a flag complex (\cref{prop:cone-flag}).

\section{The QF-tree}\label{sec:encoding}

We now describe the data structure. The cell table is the primary record of the topology, and the word trie is an index over it, as the trie of a simplex tree indexes its simplices, except that here one word can point to several cells.

\subsection{The cell record}\label{sec:qfcell}
A $d$-cell stores its identifier, its dimension, a provenance reference (the source simplex identifier and its vertex support), and exactly $d+1$ local facet slots. A zero-cell has an empty facet list. A slot contains either a collapsed-component identifier or a genuine target of dimension $d-1$. In the collapse and order-preserving regimes the face coordinate map is prescribed by the increasing source lists, so a slot is one reference, and a relative orientation sign may be cached but is derived data. These are the data required by \cref{cor:skeletal}. A class representative is not determined by the sorted quotient vertices: different source cells can have the same vertex image and different attachments, and the record keeps them apart.

The facet budget of \cref{lem:budget} counts local facet slots that target a basepoint: a flag collapse has at most two per cell. Maintaining the count makes this necessary condition testable in constant time, although passing the test does not prove that $A$ is flag.

\subsection{Container and quotient operations}\label{sec:qfcomplex}
A container has a cell dictionary, per-dimension identifiers, component-point records and optional reverse-incidence indices. Source cells are built dimension first with lexicographic tie-breaking, so every referenced facet exists before a cell is installed (\cref{alg:build}).

\begin{algorithm}[t]
\caption{Promote a finite ordered simplicial complex to a cell table}
\label{alg:build}
\begin{algorithmic}[1]
\REQUIRE the nonempty simplices of $K$, each an increasing source tuple
\FORALL{$\sigma=(v_0,\ldots,v_d)$ in increasing dimension, then lexicographic order}
 \STATE Create a cell with identifier $\sigma$, dimension $d$ and provenance $\sigma$.
 \IF{$d>0$}
  \STATE Store local facet $i$ as $\sigma\setminus\{v_i\}$ for $0\le i\le d$.
 \ELSE
  \STATE Store an empty facet list.
 \ENDIF
\ENDFOR
\STATE Return the cell table and construct a word index only if requested.
\end{algorithmic}
\end{algorithm}

\paragraph{Collapse a cell subcomplex.} Verify closure under all attaching images, find the components of its one-skeleton, replace every component by one point, delete its old cells, and redirect any surviving local facet whose target has been removed. The operation applies to every finite CW presentation considered here, including nonflag pairs, and the flag theorems apply when their hypotheses hold. Collapsing the boundary of a triangle, for instance, creates a rank-two skip, and the table represents it regardless.

\paragraph{Merge zero-cells.} Replace specified groups of zero-cells by their representatives and redirect the incident zero-cell targets. No other record changes. Equal quotient vertices do not imply equal local positions, and a vertex merge can produce loop edges without violating the prescribed dimension of any cell.

\paragraph{Dynamic locality.} A cell record needs rewriting only if one of its recorded facet targets is absorbed or identified. Reverse-incidence indices give direct access to those references, and indirect higher-face effects are already reflected by recursion through the modified target records. This is the observation that \cref{sec:local} turns into a representation and \cref{sec:implementation} into an editable layer. The rewrite is local, but global homology still requires the rest of $K$.

\subsection{Boundary reconstruction and validation}\label{sec:boundary-reconstruction}
For a genuine target $\tau_i$ with affine coordinate map $\theta_i$ the oriented cellular contribution is
\[
 \partial c=\sum_{\substack{0\le i\le d\\\dim\tau_i=d-1}} (-1)^i\sgn(\theta_i)\tau_i ,
\]
with $\sgn\theta_i=1$ for implicit increasing maps; for normalized relative chains all component-point terms are omitted (\cref{con:normalization}), and in dimension one the endpoint points are retained in absolute chains. Higher local faces are recovered by composing the facet maps and reindexing at each step; a collapsed value stays at its component point. The recursion does not require strict grading. The structural checks are:
\begin{enumerate}[label=(V\arabic*),leftmargin=2em]
\item valid target dimensions and component identifiers, and at each codimension-two overlap agreement of the two induced maps, a collapse being absorbing with its particular component label;
\item $\partial^2=0$ for the assembled differential, a necessary consistency check rather than a proof of correctness of the attaching maps;
\item in an asserted flag-collapse regime, the facet budget and component-resolved shallowness.
\end{enumerate}
With implicit increasing maps these cost $O(\sum_c d_c^2)$ elementary operations. Homology is a separate cost: Betti numbers over a field come from boundary ranks, and integral homology needs Smith normal forms with compatible changes of basis between adjacent differentials.

\subsection{Quotient words and the trie}\label{sec:qftree}
For an original vertex $v$ let $[v]$ be its current zero-cell identifier, and order the zero-cells, say by minimum source label. From a source $d$-simplex form the multiset of its $d+1$ vertex images, sort it, remove repeated labels and pad to length $d+1$ by repeating the final label. This is its \emph{index word}. Insert each word in a prefix trie and keep all cell identifiers in its terminal bucket. In the cone over a square with its boundary collapsed, all four triangles have word $(Q,5,5)$ and all four spokes have word $(Q,5)$, and they remain different cells. A word is only a lookup key: it identifies neither a cell nor its cofaces, and a masked lookup must examine the bucket's provenance and, when required, take the facet closure.

\subsection{Storage, and comparison with existing formats}\label{sec:complexity}
Let $N$ be the number of quotient cells and $M_{prov}$ the retained provenance. Facet storage for collapse and order-preserving presentations uses
\[
 O\!\left(N+\sum_{\dim c>0}(\dim c+1)+M_{prov}\right)
\]
words. For bounded dimension the cell and facet records are $O(N)$, so the whole is linear in the cell inventory precisely when $M_{prov}=O(N)$ as well. That condition can fail even for a one-dimensional collapse: collapsing an entire connected path leaves one quotient cell while the component-membership list retains every one of its vertices. This is a separate matter from retaining an archive of the original pair. Explicit permutations (\cref{app:gluing}) cost a further factor $\dim c$. An uncompressed trie adds at most $1+\sum_w|w|$ nodes plus terminal buckets of total size $N$. If every source member of a large gluing class is retained explicitly, provenance is not bounded by a constant per cell and must be counted separately.

Compared with the simplex tree, the facet table pays for the changed attachments and the component flags. Compared with storing every local face of every cell, it avoids an exponential expansion. Regina's generalized triangulations~\citep{burton2013regina,ReginaHandbook} take top-dimensional simplices with pairwise facet bijections and allow repeated vertices and self-identifications, so nonregularity and permutations are not new in themselves. The exact subclass shared with such facet-pairing formats consists of the quotients whose entire identification relation is generated by pairwise top-facet gluings. On this subclass the two presentations are interconvertible, with different storage, since QF retains lower-dimensional cells and provenance. What QF adds is non-pure sources (a Vietoris--Rips complex can be non-pure), unbounded facet incidences, lower-face identifications not generated by top-facet pairings, and source-labelled subcomplex collapse. Generalized maps and cell-tuple structures~\citep{lienhardt1994,brisson1993} and the compact simplicial representations~\citep{boissonnat2017compact,attali2012skeleton} are further baselines, and a fair comparison must match the requested operations, input class and dimension.

\section{The cone model and the storage crossover}\label{sec:crossover}

\Cref{prop:cone} offers an alternative to storing the quotient: attach a simplicial cone to each component of $A$ and keep the result, a bona fide simplicial complex, in a simplex tree. When only the homotopy type is wanted this is the natural choice, so we determine when the quotient representation is the smaller one.

\begin{con}[the cone model]\label{con:cone-model}
For $A\subseteq K$ with components $\{A_\alpha\}_{\alpha\in D}$, choose one new vertex $w_\alpha$ per component and set $\check K(K,A)=K\cup\bigcup_{\alpha\in D}\bigl(w_\alpha * A_\alpha\bigr)$. No quotient is taken and no $\bbot$ occurs. By \cref{prop:cone} it is homotopy equivalent to $K/_{SC}A$, and it can be stored in a simplex tree at one node per simplex.
\end{con}

\begin{prop}[when the cone model stays flag]\label{prop:cone-flag}
Let $K$ be flag and $A\subseteq K$. Then $\check K(K,A)$ is flag if and only if every component of $A$ is full in $K$, if and only if $K/_{SC}A$ has regular prescribed cells.
\end{prop}
\begin{proof}
The second equivalence is \cref{thm:fullness}. For the first, $\sk_1\check K$ is $\sk_1K$ plus the edges $w_\alpha v$ for $v\in V(A_\alpha)$, and no two apices are adjacent, so a clique of $\sk_1\check K$ either lies in $V(K)$, and is a simplex of $K$ by flagness, or has the form $\{w_\alpha\}\cup S$ with $S$ a clique of $\sk_1K$ inside $V(A_\alpha)$. Then $S\in K[V(A_\alpha)]$, while $\{w_\alpha\}\cup S\in\check K$ holds precisely when $S\in A_\alpha$. So every such clique spans a simplex iff $K[V(A_\alpha)]\subseteq A_\alpha$.
\end{proof}

The three conditions of this paper therefore coincide: regular collapse, full components and a flag cone model. The cone model can be stored in a simplex tree even when it is not flag, and fullness decides only whether it remains flag.

\begin{prop}[exact cell counts]\label{prop:cell-identity}
Let $c=\lvert\pi_0(A)\rvert$. Counting nonempty simplices and quotient cells,
\[
 N_Q=\lvert K\rvert-\lvert A\rvert+c,\qquad N_{\check K}=\lvert K\rvert+\lvert A\rvert+c,\qquad N_{\check K}-N_Q=2\lvert A\rvert .
\]
\end{prop}
\begin{proof}
The quotient deletes the simplices of $A$ and adds one basepoint per component; the cone adds one apex and one simplex $w_\alpha*\sigma$ per $\sigma\in A_\alpha$.
\end{proof}

The quotient always has fewer cells, but the QF-tree keeps $d+1$ references per $d$-cell while the simplex tree keeps one node per simplex. Writing $\alpha=|A|/|K|$ for the collapsed fraction and
\[
 S_Q\;=\;c+\!\!\sum_{\sigma\in K\smallsetminus A}\!\!(\dim\sigma+1), \qquad s\;=\;\frac{S_Q}{N_Q}
\]
for the storage units and the mean arity of a surviving cell, one has $S_Q\approx s(1-\alpha)|K|$ against $N_{\check K}\approx(1+\alpha)|K|$, so
\begin{equation}
 \frac{S_Q}{N_{\check K}}\;\approx\;\frac{s(1-\alpha)}{1+\alpha},
 \qquad\text{giving}\qquad
 \alpha^{\ast}\;\approx\;\frac{s-1}{s+1}
 \label{eq:crossover}
\end{equation}
for the fraction above which the quotient stores fewer units than the cone model. The number $\alpha^\ast$ depends on the input only through its dimension distribution, via $s$. Small examples show both directions: the octahedron modulo the disc complementary to an edge has $4$ quotient cells against $50$ simplices in the cone model, while $\Delta^4$ modulo a triangle has $25$ against $39$. On the boundary of the $5$-dimensional cross-polytope ($242$ simplices, $s\approx3.3$) the quotient representation becomes the smaller one at $|A|\approx0.54\,|K|$, as~\eqref{eq:crossover} predicts.

\subsection{Measurements}
We tested~\eqref{eq:crossover} on three families of flag complexes chosen to span a wide range of $s$, taking the collapsed subcomplex to be an induced subcomplex and recording the achieved simplex fraction rather than a vertex fraction, since an induced subcomplex on a fraction $f$ of the vertices contains roughly $f^{k+1}$ of the $k$-simplices. Across nine conditions, on sources of up to $854\,953$ simplices and with twenty seeds each, the predicted and measured crossovers differ by at most $0.0088$, with median error $0.0033$, while $s$ ranges from $1.86$ to $5.11$ and $\alpha^\ast$ from $0.30$ to $0.68$. In these experiments the crossover is governed by the dimension distribution alone: at fixed vertex count and mean degree, changing the clique structure alone moves $\alpha^\ast$ by a factor of two. The measurements and the plot are in \cref{sec:crossover-measured}.

\subsection{Homology alone needs no quotient object}\label{sec:homology-direct}
One might expect the quotient to pay for itself downstream, since by \cref{thm:chain} its cellular chain complex is $C_*(K,A)$ and is smaller than the cone model's. The relevant baseline, however, is the relative chain complex assembled directly from $K$: index the simplices of $K\smallsetminus A$ and drop the boundary terms that land in $A$. No quotient object, no cone and no basepoints are needed, and the matrix differs from the quotient's only by the $|\pi_0(A)|$ basepoint generators in degree zero. We measured the three routes on $190$ pairs ($3$-skeleta of Vietoris--Rips complexes on $100$ to $1000$ points, $\alpha$ from $0.06$ to $0.76$), reducing all three with one shared $\mathbb F_2$ column-reduction kernel. The median paired total-time ratio was $1.65$ for quotient versus direct and $2.52$ for cone versus direct in a reference Python implementation, and $1.33$ for quotient versus direct with a compiled kernel. The chain complexes were checked against a brute-force oracle on all $729$ simplicial pairs on four labelled vertices and $200$ larger random pairs, and across $2400$ measured pairs on sources of up to $854\,953$ simplices the three routes agreed on every $\beta_k$, $k\ge2$. The full-component criterion of \cref{thm:fullness} held in all $2400$ induced cores and failed in all $2560$ edge-thinned ones, as predicted.

For one-shot relative homology the direct route is therefore preferable, as \cref{sec:experiments} confirms against \Gudhi. Beyond masking, the quotient provides the cells of $K/_{SC}A$ as addressable objects, its homeomorphism type, and closure under further operations that need the attaching data.

\section{Local application: the closed star}\label{sec:local}

The remark on dynamic locality in \cref{sec:qfcomplex} can be turned into a representation. A collapse changes the attaching data of a simplex only when the simplex meets $A$, so the QF-tree need only be built there, and the rest of $K$ can stay in the simplex tree that already holds it. This is the setting the library is designed for: a large Vietoris--Rips complex in a \Gudhi\ simplex tree, and a comparatively small subcomplex to be collapsed.

Let $A\subseteq K$ and partition the simplices of $K$ into four classes:
\begin{itemize}[leftmargin=2em]
\item \emph{collapsed}: $\sigma\in A$;
\item \emph{star}: $\sigma\notin A$ with $V(\sigma)\cap V(A)\neq\varnothing$;
\item \emph{frontier}: $\sigma$ disjoint from $V(A)$ that is a face of a star simplex;
\item \emph{untouched}: everything else.
\end{itemize}
A simplex meets $|A|$ if and only if it has a vertex in $V(A)$, so collapsed and star simplices together form the star $\st(A)$ of $A$ in $K$, and the first three classes form the closed star $\overline{\st}(A)$, a subcomplex of $K$ containing $A$. Write $U$ for the set of simplices disjoint from $V(A)$, so that $K\setminus A$ is the disjoint union of the star simplices and $U$, and the frontier is $U\cap\overline{\st}(A)$.

\begin{prop}[locality of the collapse]\label{prop:locality}
Let $A\subseteq K$ be a finite simplicial pair.
\begin{enumerate}[label=(\roman*),leftmargin=2em]
\item Every facet of a simplex in $U$ lies in $U$, and its record in $\Dd(K,A)$ is the simplicial one: no facet is collapsed and all targets are in $U$.
\item Every facet of a star simplex is collapsed, star, or frontier.
\item Consequently $\Dd(K,A)$ is the pushout, along the frontier, of $\Dd(\overline{\st}(A),A)$ and the semisimplicial set $U^\Delta$: the decoration of the closed-star pair together with the simplex tree of $K$ determines the quotient, its characteristic maps, and its chains.
\end{enumerate}
\end{prop}
\begin{proof}
(i) A face of a simplex disjoint from $V(A)$ is disjoint from $V(A)$, hence not in $A$ and in $U$. (ii) A facet of a star simplex either meets $V(A)$, and is then collapsed or star, or is disjoint from $V(A)$, and is then by definition frontier. (iii) By (i) and (ii) the genuine cells of $\Dd(K,A)$ are the disjoint union of the star and frontier cells, which form $\Dd(\overline{\st}(A),A)$ with the same face maps, and the cells of $U\setminus\text{frontier}$, whose face maps are simplicial and land in $U$; the two parts overlap exactly in the frontier, on which both prescribe the simplicial face maps. The claims about characteristic maps and chains follow from \cref{cor:skeletal,thm:chain}.
\end{proof}

The local QF-tree stores $d+1$ references for each star and frontier cell, with source identifiers into $K$, and nothing for the untouched cells. How much this saves depends on what happens to the source records, and we distinguish three accounting conventions, all in the units of \cref{sec:crossover}: one unit per simplex held in a simplex tree, $d+1$ units per QF record, one unit per component. Write $T$, $F$, $S$ for the untouched, frontier and star simplices, $H=S\cup F$ for the cells with local records, $R_H=\sum_{\sigma\in H}(\dim\sigma+1)$, and $c=|\pi_0(A)|$.
\begin{itemize}[leftmargin=2em]
\item \emph{Retained source.} The object \texttt{K.local\_quotient(A)} of the library keeps the whole of $K$ next to the local records: $S_{\mathrm{ret}}=|K|+R_H+c$.
\item \emph{Compact.} \texttt{compact()} replaces $K$ by the subcomplex $T\cup F$ (a subcomplex, since faces of untouched or frontier simplices are untouched or frontier) and keeps the local records: $S_{\mathrm{cmp}}=|T|+|F|+R_H+c$. The frontier is present in both parts, which is what allows them to be joined without the source pair.
\item \emph{Ideal.} A hybrid that shares every record, so that a frontier simplex is stored once and a star record can point at it directly, would use $S_{\mathrm{id}}=|T|+R_H+c$. This is a reference budget for the stated inventory of explicit closed-star records without duplication, and other hybrid encodings can use less. For example, one may retain $U=T\cup F$ with its implicit simplicial attachments and store explicit facet lists only for $S$, with tagged references into $U$, $S$ or the component points. In the same logical unit convention this gives
\[
 S_{\mathrm{imp}}=|T|+|F|+R_S+c =S_{\mathrm{id}}-\sum_{\sigma\in F}\dim\sigma.
\]
Frontier faces need no explicit lists because $F$ is a subcomplex and its face maps are unchanged. The library does not implement this alternative.

\end{itemize}
Against the cone model, $S_{\mathrm{cone}}=|K|+|A|+c$, the three criteria are
\begin{align*}
 S_{\mathrm{id}}<S_{\mathrm{cone}}&\iff\textstyle\sum_{\sigma\in H}\dim\sigma<2|A|,\\
 S_{\mathrm{cmp}}<S_{\mathrm{cone}}&\iff\textstyle\sum_{\sigma\in H}\dim\sigma+|F|<2|A|,\\
 S_{\mathrm{ret}}<S_{\mathrm{cone}}&\iff R_H<|A|.
\end{align*}
These budgets count logical units. They are not byte counts and omit masks, labels, provenance and transient indexes.

Whether the closed star is small compared with $A$ depends on the shape of $A$ rather than on its size, and localization alone does not decide it. Take the Vietoris--Rips complex at radius $1$ on the four collinear points $(0,0),(0.1,0),(0.8,0),(0.9,0)$, which is the full tetrahedron, and let $A$ be the edge on the first two points, the induced subcomplex on a ball of radius $0.2$ about the origin. Both are flag and $A$ is connected, but every simplex of $K$ belongs to the closed star of $A$: $T=\varnothing$, $|F|=3$, $|S|=9$, $R_H=28$, and $S_{\mathrm{id}}=29$, $S_{\mathrm{cmp}}=32$, $S_{\mathrm{ret}}=44$ against $S_{\mathrm{cone}}=19$. A statement such as ``localized collapses have thin shells'' therefore needs geometric hypotheses on the sampling, and we report only one sampled regime. In our experiments, at every sampled ball radius, the median compact budget is below the cone budget, although two of the five seeds at the smallest radius lie slightly above it (Appendix \ref{sec:local-measured}). For these medians, the compact closed-star presentation thus removes the crossover of \cref{sec:crossover} for collapses that arise from a region of a point cloud. A record-sharing hybrid on a live simplex tree would reach the ideal budget. The library's compact object exceeds it by at most about $3\%$ on the tested inputs, which is less than one percent of the cone budget. The retained-source object exceeds it by a factor that grows with the radius, to about $3.6$ at $\rho=0.50$.

\section{The library}\label{sec:implementation}

The library \texttt{qfcore} is written in C++17 with a Python interface, and its API follows \Gudhi's: a \texttt{FlagComplex} is built from a graph or a point cloud as a \Gudhi\ simplex tree is, a subcomplex is a mask or an induced or flag subcomplex, and \texttt{K.quotient(A)} returns an immutable \texttt{QFTree}. \Gudhi's own simplex tree, its persistence-matrix module and its streaming zigzag engine are used as consumers and as baselines throughout. This section describes the four layers, and \cref{sec:experiments} reports the measurements.

\subsection{The immutable QF-tree}
\texttt{QFTree} is the cell table of \cref{sec:encoding}: per-dimension arrays of source identifiers, vertex supports, and $d+1$ facet slots per cell, the first $c$ zero-cells being the component points. The word trie is optional. Construction from a pair $(K,A)$ is \cref{alg:build} followed by redirection of the facets in $A$, and the validator runs the checks (1)--(3) of \cref{sec:boundary-reconstruction}: facet supports, codimension-two identities, and, on request, the facet budget, component-resolved shallowness (\cref{thm:strict}) and component fullness (\cref{cor:regular-edges}). A further collapse of a closed cell subcomplex returns the new tree together with the cell map, which is checked to be a morphism of pointed presentations. Trees are serialized without their source pair in a versioned binary format with a checksum, and reloaded without it, so the quotient is a standalone object.

\subsection{The closed-star variant}
\texttt{K.local\_quotient(A)} implements \cref{sec:local}. It labels every simplex of $K$ as untouched, frontier, star or collapsed in one pass over $K$ with a face walk from the star simplices, builds the QF-tree of the closed-star pair with source identifiers into $K$, and keeps the snapshot of $K$ taken at construction, which later insertions into $K$ do not affect. This is the retained-source object of \cref{sec:local}. \texttt{compact()} replaces $K$ by the subcomplex of untouched and frontier simplices and returns an object that computes the quotient and relative chain complexes from the two parts alone, joined along the frontier by vertex support. It realizes the compact budget of \cref{sec:local}. The source complexes themselves are the library's flat \texttt{FlagComplex} arrays, into which a \Gudhi\ simplex tree is copied. An in-place overlay on a live simplex tree, which would reach the ideal budget, is not implemented. The relative and quotient chain complexes of the retained-source object are assembled in the same generator order as the direct construction, so that the two can be compared column for column. The method \texttt{verify()} performs this comparison together with a record-by-record comparison with the full QF-tree. The measurements of \cref{sec:local-measured} ran it on every pair and also checked that the compact object returns the same Betti numbers.

\subsection{The editable layer}\label{sec:qfe070records}
\texttt{EditableQF} is a mutable cell table with stable identifiers and reverse incidences. A simplex-type $d$-cell keeps its ordered list of $d+1$ local facets, each a codimension-one cell or a constant point map, and repeated occurrences remain separate: cells are never merged because their vertices, words or boundary columns agree. Reverse incidences are kept before cancellation over the coefficient field, so an edge appearing twice with opposite signs in the boundary of a disc is still a dependency and cannot be deleted as a maximal cell.

The editable layer also accepts a polygonal $2$-cell, stored as a basepoint and a finite closed signed edge word $w=e_1^{\epsilon_1}\cdots e_m^{\epsilon_m}$, and the empty word is a constant attaching map. Polygons are a cellular extension and cannot serve as ordinary facets of higher simplex cells. The editor supports cell insertion, deletion of a maximal cell or of a cell with all its cofaces, point identification, and componentwise quotienting by a closed subcomplex. During a quotient each selected component is represented by a retained vertex, surviving cells keep their identifiers, and only records that directly reference a removed cell are rewritten (\cref{prop:locality} in its dynamic form). The operation returns a sparse cell map, the identity being understood on all other cells. Cloning, serialization, whole-object validation and chain assembly remain global, and branching uses full copies rather than a persistent structure, which \cref{sec:qfe070collapse} shows in the memory figures.

\subsection{Consumers: maps, cohomology, $\pi_1$, zigzag}\label{sec:qfe070consumers}
All algebra is over $\mathbb F_2$ and every homology or ring result is an explicit snapshot rebuilt for the changed space, so the local topology updates do not maintain homology bases incrementally. For a represented map $f:Q\to R$ the consumer constructs the chain map, homology representatives, $H_k(f)$, kernel and cokernel bases, and compositions. The cell-map validator checks compatibility with the ordered attaching data slot by slot, in addition to $\partial F=F\partial$. The latter is necessary for a chain map but does not certify that the map respects the stored parametrizations.

For simplex-type cells the cup product is Alexander--Whitney evaluated through the iterated local face maps,
\begin{equation}
 (a\smile b)(\sigma)
 =a(\sigma[0,\ldots,p])\,b(\sigma[p,\ldots,p+q]),
 \label{eq:qfe070aw}
\end{equation}
a positive-degree cochain evaluating to zero on a constant face; SageMath documents a comparable interface for finite simplicial sets~\citep{qfe070sage}. For polygonal $2$-cells and degree-one cocycles, representatives are first changed by coboundaries so as to vanish on a deterministic maximal forest, and with $A_i=a(e_i)$, $B_i=b(e_i)$ the based word diagonal gives
\begin{equation}
 \langle a\smile b,P\rangle
 =\sum_{i<j}A_iB_j+\sum_{\epsilon_i=-1}A_iB_i
 \pmod 2 .
 \label{eq:qfe070polygon}
\end{equation}
The inverse-letter term is essential: for $\mathbb{RP}^2$ with relator $aa$ the square of the generator is nonzero, for the word $aa^{-1}$ it is zero. The formula is checked against an independent triangulated-fan model, in which each polygon is replaced by ordered triangles and the whole product table is compared under the induced pullback using~\eqref{eq:qfe070aw} only.

Fundamental groups are extracted as edge-path presentations: nonforest edges of a maximal forest generate, an ordered triangle contributes the boundary path $(d_2,d_0,d_1^{-1})$ with constant and forest steps omitted, and a polygon contributes its word. A bounded Tietze simplifier is available. A useful control is the torus and $(S^1\vee S^1)\vee S^2$ as one-vertex models with two loops $a,b$ and one disc attached along $aba^{-1}b^{-1}$ or constantly. Their integral cellular differentials are identical and zero, both have $\beta=(1,2,1)$, but the stored words give $\langle a,b\mid aba^{-1}b^{-1}\rangle$ against $\langle a,b\mid\ \rangle$, and the degree-one products $x\smile y=y\smile x=z$ against zero. The proposed identity cell map between them is rejected despite the identical chain matrices.

Point gluing forms the disjoint union of saved presentations and identifies specified points, returning the inclusion maps. Subsequent quotients and disc attachments operate on the result without the original pairs. For insertion and deletion streams the consumer is \Gudhi's streaming zigzag with stored intervals~\citep{qfe070gudhizz,carlsson2010zigzag}. The editor checks the legality of each elementary step from the full topological incidences, and all representations compared in the edit experiment feed the same engine. Quotient maps are not deletions, so for small zigzags of arbitrary homology maps a separate reference consumer computes interval multiplicities from the ranks of the canonical maps from inverse to direct limits on subintervals. This consumer is exact but global.

\section{Experiments}\label{sec:experiments}

We report three sets of measurements: a static comparison of the QF-tree with \Gudhi\ for relative homology, in which \Gudhi's direct route is faster, a controlled comparison of local quotient updates with rebuilding and with a masked simplex tree, and the reuse of a retained quotient for cohomology rings, zigzag persistence and a source-free composition of operations. Every timed computation runs in a fresh process. The routes, the timing endpoints, the design of the operation experiments and the complete tables and figures are in \cref{app:tables}, together with the timing endpoints on which the conclusions depend. This section gives only the findings.

\subsection{Static comparison with \Gudhi}\label{sec:qf-gudhi-controlled-050}
On $3$-skeleta of Vietoris--Rips complexes on $5000$ points, cold computation of a relative Betti vector (source import, quotient or cone construction, boundary assembly and the solver) is $1.26$ to $1.55$ times slower through a full QF-tree than through direct relative \Gudhi, at every one of the five sampled fractions. The QF route overtakes the component-cone route between the fractions $0.20$ and $0.50$ (at $\alpha=0.95$ the median paired cone/QF ratio is $2.6$). When both representations are already materialized and the same \Gudhi\ Matrix reducer is used, feeding it from a QF-tree is $0.98$ to $1.16$ times as fast as from the simplex-tree pair ($1.3$ to $1.4$ for two factors by an exact common intersection), and on an uncapped $50\,000$-vertex clique complex with $1\,170\,307$ simplices and dimension $11$ the QF-fed solve took $0.011$\,s against $0.044$\,s. Serialized size strongly favours the quotient (at $\alpha=0.95$, $0.23$\,MiB against $1.28$ for the marked simplex tree and $2.49$ for the cone model, and $3.6$ against $17.9$ and $34.8$ on the large example), while process peak memory does not: $56$\,MiB against $45$ on the $5000$-vertex inputs and $402$ against $201$ on the large one. All Betti vectors agreed across routes in every one of the $135$ pairs of the replication and the $60$ of the pilot, and sixty stages of nested quotients agreed byte for byte with the corresponding direct quotients. As anticipated in \cref{sec:homology-direct}, a direct relative matrix is the faster tool for a one-shot relative Betti vector, and a QF-tree is worth its construction only when the quotient is subsequently used as a space.

\subsection{Operations on a retained quotient}\label{sec:qfe070collapse}
The operation experiments use periodic triangulations $K_s$ of the $2$-torus on an $s\times s$ grid of vertices, with $\lvert K_s\rvert=6s^2$ (here $s$ is the grid side, unrelated to the mean arity of \cref{sec:crossover}). Contracting a maximal spanning tree gives an initial quotient $Q_s$ with $4s^2+2$ cells, which at $s=48$ means $13\,824$ source simplices and $9218$ quotient cells. We compare three quotient implementations: the editable local QF, the immutable QF-tree rebuilt after each change, and a persistent \Gudhi\ simplex tree with an updated subcomplex mask. All three feed the same consumer, so the comparison isolates the cost of maintaining the representation. All compared outputs agree: boundary matrices, cycle representatives, induced maps, kernels and cokernels for quotients; complete zigzag barcodes and final cell sets for edits; basis-independent invariants and the fan-model pullback for rings.

The results fall into four parts. First, local updates are cheap, but they are a small part of the total cost. Each local edge quotient visits four cell records and examines six facet occurrences, independently of the quotient size, which ranges from $66$ to $9218$ cells. At $s=48$ the paired rebuild/local ratio for the topology update alone is $369$ (range $307$--$431$). The complete-operation ratio is only $1.17$, and the masked-pair/local ratio is $1.0$, because homology-basis construction accounts for a median $94\%$ of local-QF operation time.

Second, editing behaves in the same way. For thirty-two edits at $s=48$ the paired reconstruction/local ratio is $23$ (range $18$--$24$), while the \Gudhi/local ratio is $0.75$ (range $0.65$--$0.96$): local editing avoids the repeated full-table copies, but \Gudhi's own insertion and deletion route is faster on this family. Barcodes agree across the three representations, including the pairing of births and deaths, which a sequence of Betti vectors would not determine.

Third, retaining the quotient is worthwhile for cohomology rings. From a saved quotient $Q_s$ and the saved triangulation $K_s$ the same consumer returns the torus ring, with dimensions $(1,2,1)$, zero squares and a nonzero cross product. At $s=24$ basis construction and the positive-degree product table take $4.6$\,ms from the quotient against $23.8$\,ms from the triangulation, a paired ratio of $5.2$ (range $4.4$--$6.3$). Including loading, the ratio is $2.6$, and including the entire process it is $1.1$. The gain therefore applies to a new ring computation on an already retained factor, with the construction of the quotient not charged.

Fourth, the operations compose without the source. A composition task loads two copies of the saved factor, identifies a marked point of each, builds the inclusion maps, quotients by an edge from each module, saves and reloads, extracts a fundamental-group presentation, attaches a disc along a commutator, computes the cohomology ring, saves and reloads again, and finally deletes and reattaches the disc in a zigzag session. The pointed union has $\beta=(1,4,2)$, the maps and attaching records survive throughout, and a control repeats the whole computation after the source files have been deleted. At $s=24$ the final object has $4610$ cells and the median time including loads, saves and reloads is $186$\,ms. This task has no comparator and demonstrates only that the operations compose on the retained object.

\section{Discussion and limitations}\label{sec:discussion}

The two structural statements about collapses are complementary. Flagness of the collapsed subcomplex controls the order of the quotient when the ambient complex is flag, and the size of its minimal non-faces controls how many low-dimensional cells must be kept to reconstruct the rest. Componentwise fullness controls regularity without any flag hypothesis, and it is also the condition under which the cone model stays flag. Together they determine what a simplex-tree-like structure can keep after a quotient: strict gradedness, codimension-one attachment storage at $d+1$ references per cell, and a finite reconstruction threshold, all under the hypothesis that the collapsed part is flag, which is automatic for intersections of Vietoris--Rips complexes.

The experiments show where retaining the quotient is worthwhile. Maintaining it locally is orders of magnitude cheaper than rebuilding it, and computing a cohomology ring from a retained factor is several times cheaper than from the source triangulation. Fundamental-group presentations are extracted from the retained factor without the source pair. When only a vector of Betti numbers is wanted, the relative boundary matrix that \Gudhi\ already computes is the better tool. In the sampled regime, the closed-star representation of \cref{sec:local} changes the storage comparison in the situation that motivated the paper, a large complex whose collapsed part occupies a region of the point cloud. It does not help when the collapsed part is scattered, and the four-point example shows that localization alone guarantees no saving.

The results have several limitations. The reconstruction theorems are for labelled pairs on a fixed vertex set, and we do not know whether the label-free truncation determines the quotient in general. The gluing theory of \cref{app:gluing} is only partially implemented: the library supports point gluing and polygonal disc attachment but not arbitrary affine identifications, and a coherent permuted gluing needs the full vertex bijection per facet rather than a sign. Source labels, class membership and the word trie can dominate storage when the quotient has few cells. The storage bound of \cref{sec:complexity} is linear in the number of cells only for bounded dimension (the full $d$-simplex has $\Theta(N\log N)$ facet references) and for retained provenance linear in the number of cells (collapsing a whole path leaves one cell but keeps the labels of all its vertices). None of the storage-unit comparisons is a byte or RSS measurement. The dimension-capped Vietoris--Rips inputs of \cref{sec:crossover,sec:local,sec:experiments} are $3$-skeleta of flag complexes and need not be flag themselves, so they exercise the constructions, which apply to arbitrary pairs, rather than the flag-specific theorems. The operation experiments use a two-dimensional torus family with three seeds per condition. They do not cover heterogeneous modules, large incidence stars or high-dimensional attaching maps, and the local topology updates do not maintain homology bases incrementally, although basis construction dominates the running time. Finally, the regression tests check the implementation, whereas the theorems rest on their proofs.

\section*{Software and data availability}
The source of \texttt{qfcore} (version 0.7.4), with its tests, examples and executable controls, is freely available under the MIT License via \texttt{pip install qfcore} and at \url{https://github.com/qfcore/qfcore}. The numerical data and plotting scripts of the operation experiments form a separate reproducibility companion, provided in the \texttt{reproducibility} directory of the same repository. The closed-star measurements of \cref{sec:local} and \cref{fig:star-locality} are produced by \texttt{tools/star\_locality.py} and \texttt{tools/plot\_star\_locality.py} of the release.

\section*{Acknowledgements} 

The work of K.S. was supported by HSE University (HSE-BR-2025-077). This research was supported in part with computational resources of HPC facilities at HSE University \cite{charisma}. K.S. thanks Fedor Vylegzhanin and Ilyas Bayramov for fruitful discussions on the original construction of quotient flags. Large language model assistants (Claude by Anthropic and ChatGPT by OpenAI) were used to review the manuscript and the software, to draft and revise passages of the text and to write additional software tests, and the authors take full responsibility for the content. 

\renewcommand{\bibliofont}{\fontsize{10}{11.4}\selectfont}

\appendix

\section{Flag spheres and the extremal collapse family}\label{app:flag-spheres}

\begin{lem}[joins]\label{lem:join}
For graphs $G_1,G_2$ on disjoint vertex sets let $G_1*G_2$ denote their join (all edges between the parts added). Then $\Cl(G_1*G_2)=\Cl(G_1)*\Cl(G_2)$, the simplicial join.
\end{lem}
\begin{proof}
A clique of $G_1*G_2$ is exactly a union of a clique of $G_1$ and a clique of $G_2$, since all cross pairs are adjacent; and the simplices of a simplicial join are exactly such unions.
\end{proof}

\begin{thm}\label{thm:cross-polytope}
Let $G$ be the graph on $2k$ vertices $\{x_1,y_1,\dots,x_k,y_k\}$ with all edges present except the $k$ partner edges $x_iy_i$. Then $\Cl(G)$ is the boundary of the $k$-dimensional cross-polytope; in particular $|\Cl(G)|\cong S^{k-1}$.
\end{thm}
\begin{proof}
$G$ is the join of the $k$ two-vertex edgeless graphs $\{x_i,y_i\}$. By Lemma~\ref{lem:join}, $\Cl(G)=\Cl(\{x_1,y_1\})*\cdots*\Cl(\{x_k,y_k\})=(S^0)^{*k}$, the $k$-fold join of $0$-spheres. Since $S^p*S^q\cong S^{p+q+1}$, induction gives $(S^0)^{*k}\cong S^{k-1}$, and combinatorially the join of $k$ copies of $S^0$ is the boundary complex of the cross-polytope. (Concretely: a maximal clique picks one vertex from each pair, since partners are never adjacent, so there are $2^k$ facets of dimension $k-1$, the facets of the cross-polytope.)
\end{proof}

The cross-polytopes are the vertex-minimal flag spheres. The bound below is classical (see Meshulam~\citep{domination} for the sharp statement in terms of face numbers). We include the short link induction because the uniqueness statement is used in the examples.

\begin{thm}[vertex-minimality]\label{thm:2k}
Let $X$ be a flag complex with $\widetilde H_{k-1}(X;\mathbb{Z})\neq 0$ for some $k\ge 1$. Then $X$ has at least $2k$ vertices. In particular every flag triangulation of $S^{k-1}$ has at least $2k$ vertices, and the cross-polytope boundary attains the bound.
\end{thm}
\begin{proof}
Two facts about flag complexes: (i) for a vertex $v$, the link $\lk(v)$ is the clique complex of the induced subgraph on the neighbours of $v$ (since $\sigma\in\lk(v)\iff\sigma\cup\{v\}\in X\iff\sigma\cup\{v\}$ is a clique $\iff\sigma$ is a clique inside $N(v)$), hence flag; (ii) the induced subcomplex $X\smallsetminus v$ on $V\smallsetminus\{v\}$ is flag.

We induct on $|V|$; the statement is: for all $k\ge1$, a flag complex on $n$ vertices with $\widetilde H_{k-1}\neq0$ has $n\ge 2k$. If $k=1$: $\widetilde H_0\neq0$ means $X$ is disconnected, so $n\ge2$. Let $k\ge2$ and pick any vertex $v$. Decompose $X=\st(v)\cup(X\smallsetminus v)$ with $\st(v)\cap(X\smallsetminus v)=\lk(v)$, where $\st(v)$ denotes the closed star, a cone, hence acyclic. Mayer--Vietoris gives exactness of
\[
\widetilde H_{k-1}(X\smallsetminus v)\longrightarrow \widetilde H_{k-1}(X) \longrightarrow \widetilde H_{k-2}(\lk(v)).
\]
Since the middle group is non-zero, one flank is non-zero.

\emph{Case 1: $\widetilde H_{k-1}(X\smallsetminus v)\neq0$.} By induction (fewer vertices, same $k$), $n-1\ge 2k$, so $n\ge 2k$.

\emph{Case 2: $\widetilde H_{k-2}(\lk(v))\neq0$.} First, $v$ has a non-neighbour: otherwise $X=v*\lk(v)$ is a cone, so $\widetilde H_{k-1}(X)=0$, a contradiction. Hence $\lk(v)$ is a flag complex on at most $n-2$ vertices (missing $v$ and a non-neighbour). By induction (fewer vertices, parameter $k-1$), $n-2\ge 2(k-1)$, so $n\ge 2k$.
\end{proof}

\begin{rem}[uniqueness]\label{rem:2k-unique}
For $k=1$, equality means two isolated vertices. For $k\ge2$, equality forces the cross-polytope: if $n=2k$, Case~1 is impossible at every vertex (it would give $n\ge 2k+1$), so every vertex has \emph{exactly} one non-neighbour. Case~2 must then hold at every vertex $v$, so $\lk(v)$ is a flag complex with $\widetilde H_{k-2}\neq0$ on at most $n-2=2k-2$ vertices, and the bound forces $\lk(v)$ to have exactly $2(k-1)$ vertices, so $v$ is adjacent to all but one vertex, and it has at least one non-neighbour because $X$ is not a cone. Non-adjacency is therefore a perfect matching on $V$, and Theorem~\ref{thm:cross-polytope} identifies $X=\Cl(\sk_1X)$ as the cross-polytope boundary.
\end{rem}

\begin{ex}[the extremal collapse family]\label{ex:octa-family}
Let $k\ge2$, let $X$ be the cross-polytope sphere of Theorem~\ref{thm:cross-polytope}, and let $e=x_ix_j$ be any edge. The closed star $\overline{\st}(e)=e*\lk(e)$, where $\lk(e)$ is the clique complex induced on the common neighbours of $x_i$ and $x_j$ (all vertices except $x_i,x_j$ and their two partners), which is a cross-polytope sphere $S^{k-3}$, with the formal empty complex $S^{-1}$ at $k=2$; so $\overline{\st}(e)\cong D^{k-1}$. The deletion of the open star of $e$ is the inclusion-maximal proper flag subcomplex $L$ obtained by deleting that edge: it is $\Cl$ of $\sk_1X$ with the single edge $e$ removed, hence flag by Lemma~\ref{lem:flag-subcomplexes}, and $|L|\cong D^{k-1}$ with $\partial L\cong S^{k-2}$. The quotient $X/L\cong S^{k-1}$ has exactly the cells
\[
0'=[L],\quad e,\quad \{\,e*\tau\;:\;\tau\in\lk(e)\,\},
\]
one basepoint, one $1$-cell (a loop, both endpoints collapsed) and one cell in every dimension $2,\dots,k-1$ for each face of the link sphere; all attachments are along codimension-one faces, collapsed facets appearing as basepoint arcs. For $k=2$: the square modulo a $3$-chain gives a loop on one vertex. For $k=3$: the octahedron ($6$ vertices, $12$ edges, $8$ triangles) modulo the disc complementary to one edge gives $S^2$ with cells $(1,1,2)$; the library confirms Betti numbers $(1,0,1)$ and $\chi=2$. This family is related to, but different from, the contractible-source construction of Theorem~\ref{thm:homotopy-change}, since its source is not contractible.
\end{ex}

\section{Signed cover categories, and what they forget}\label{app:covers}

The face poset of a regular CW complex determines it up to homeomorphism~\citep{bjorner1995topological}. For the nonregular quotients of this paper one might hope that a face poset enriched by parallel covers and signs still suffices. It does not, and this appendix records exactly what such a signed cover category retains: enough for the cellular chain complex, on a restricted class of morphisms, and not enough for the homotopy type.

\begin{defin}[Signed cover categories and their morphisms]\label{def:gacat}
An object of $\GACat$ is a finite acyclic category $\mathcal C$ (in the sense of~\citep[Chapter~10]{kozlov2008combinatorial}) with ranks, a distinguished set $D$ of rank-zero objects, and signs on covers, such that every cover raises rank by one and, for genuine $x,y$ with $\rk y=\rk x+2$,
\[
 \sum_{\substack{z\notin D\\x\xrightarrow{f}z\xrightarrow{g}y\text{ covers}}} \varepsilon(f)\varepsilon(g)=0.
\]
Put $[x:y]=\sum_{f\in\operatorname{Cov}(x,y)}\varepsilon(f)$. A morphism is a rank-preserving functor preserving covers and their signs, mapping basepoints to basepoints and genuine objects to genuine objects, and satisfying the additional \emph{incidence pushforward condition}
\begin{equation}
 [x':F(y)]=\sum_{\substack{x\notin D\\F(x)=x'}}[x:y]
 \label{eq:incidence-pushforward}
\end{equation}
for each genuine $y$ and target genuine $x'$ one rank below it. Identities satisfy this condition, and composition preserves it by summing first over the intermediate object fibres.
\end{defin}

\begin{rem}[Why the extra morphism condition is needed]\label{rem:chain-map-counterexample}
A sign-preserving functor can send one positive arrow $x\to y$ to one of two positive arrows $x'\rightrightarrows y'$. Rank and cover preservation hold, while the induced chain differentials are $1$ and $2$. The object maps are not a chain map. Equation~\eqref{eq:incidence-pushforward} rules out exactly this discrepancy. Bounds $[x:y]\in\{0,\pm1\}$ alone do not characterize regularity, since nonregular collapse presentations can also have only these coefficients.
\end{rem}

\begin{con}[Cover functor]\label{con:U}
For $S\in\FFacesh$, take the quiver with objects $D\sqcup\coprod_nS_n$, one arrow $\partial_i\sigma\to\sigma$ for every genuine facet occurrence, and an arrow $\bbot_\alpha\to\sigma$ for every basepoint facet of an edge. Give each such arrow sign $(-1)^i$. Let $U(S)$ be its free category. Genuine-preserving face maps send the occurrence indexed by $i$ to the occurrence indexed by $i$, defining $U$ on morphisms.
\end{con}

\begin{prop}[Properties of the cover functor]\label{prop:U}
The functor $U:\FFacesh\to\GACat$ is well defined. Its covers are the quiver arrows. For each component label, $\bbot_\alpha\preceq\sigma$ in the original presentation if and only if there is a morphism $\bbot_\alpha\to\sigma$ in $U(S)$.
\end{prop}
\begin{proof}
Every arrow raises rank by one, so freeness gives acyclicity and identifies the covers. Shallowness provides, by induction, a genuine facet above each relevant component point; this proves reachability. The incidence axiom is the cancellation of the two orders of deleting two vertices in the semisimplicial identities. If the final face is genuine, both intermediate facets are genuine, so none of the cancelling terms is discarded. For a genuine-preserving morphism, the target facet at each index is the image of the source facet at that same index. Thus the sum of incoming signs over the source objects mapping to $x'$ equals the target incidence sum, proving~\eqref{eq:incidence-pushforward}. General pointed morphisms that absorb genuine cells are excluded from this rank-preserving construction, but remain valid for $\mathcal N$.
\end{proof}

\begin{con}[Linearization of covers]\label{con:GC}
Define $C_n(\mathcal C)=\mathbb Z\{x\notin D:\rk x=n\}$ and
\[
 \partial y=\sum_{x\notin D}[x:y]x.
\]
The incidence axiom gives $\partial^2=0$. On morphisms use $x\mapsto F(x)$; equation~\eqref{eq:incidence-pushforward} is exactly the chain-map identity. Hence $C:\GACat\to\Ch$ is a functor on the stated morphism class. On $\FFacesh$ there is a natural equality $CU=\mathcal N$. Write $\Gg=U\Dd$ on $\Flagref$. For an order-preserving gluing datum, $\Gg(K,\sim)$ denotes the object $U\Dd(K,\sim)$; this notation does not specify any additional gluing morphism category.
\end{con}

\begin{rem}[What the quiver retains]\label{rem:quiver}
The free category can be represented by its finite signed generating quiver, without storing composite paths. Its arrows retain incidence occurrences and signs but forget their local facet indices. The complete facet table of $\Dd$ additionally retains those indices and all collapsed higher facets. Thus the quiver and the table are different objects. Acyclic categories in the standard sense~\citep[Chapter~10]{kozlov2008combinatorial} may have relations between paths, whereas our free category imposes none.
\end{rem}

\begin{rem}[The two regimes]\label{rem:dichotomy}
For an order-preserving gluing $D=\varnothing$ and $\Dd$ is an ordinary semisimplicial set. For a collapse the distinguished component chains need not be inactive; flagness supplies shallowness, but collapsed facet values remain. Permuted gluings use the affine-CW construction of Proposition~\ref{prop:permuted-cw}. Their signed cellular chains satisfy $\partial^2=0$, but they are not objects of $\FFace$ with the same cells.
\end{rem}

\begin{figure}[htbp]\centering
\begin{tikzcd}[column sep=large,row sep=large]
 \Flagref \arrow[r,"\Dd"] \arrow[d,hook]
 & \FFacesh \arrow[r,"U"] \arrow[d,hook]
 & \GACat \arrow[d,"C"]\\
 \Flag \arrow[r,"\Dd"']
 & \FFace \arrow[r,"\mathcal N"'] \arrow[d,"|-|_\bullet"']
 & \Ch\\
 & \Topcat &
\end{tikzcd}
\caption{The reflecting restriction is needed only for the rank-preserving cover path. The direct normalized-chain and realization paths allow genuine cells to be absorbed, as happens in a quotient tower. Permuted CW presentations have a separate realization rule.}
\label{fig:architecture}
\end{figure}

\subsection{What signed covers and their nerves do not recover}
\begin{prop}[A free-cover nerve is not a realization]\label{prop:nerve}
For $K=\Delta(123)$ and $A=12\cup13$, the quotient is contractible, whereas $B\Gg(K,A)\simeq S^1$.
\end{prop}
\begin{proof}
The genuine cells are $a=23$ and $b=123$, with $\partial a=(\bbot,\bbot)$ and $\partial b=(a,\bbot,\bbot)$. The quiver has two arrows from the component point to $a$ and one from $a$ to $b$. Its free-category nerve is homotopy equivalent to that graph by Proposition~\ref{prop:nerve-graph}, hence to $S^1$. The source $A$ is a contractible boundary arc, so its collapse in the disc is a homotopy equivalence and the quotient is contractible. This refutes reconstruction from the nerve, although it does not exclude other reconstructions from the retained data.
\end{proof}

\begin{prop}[The free-cover nerve is a graph up to homotopy]\label{prop:nerve-graph}
For every finite shallow presentation, the classifying space of its free cover category is homotopy equivalent to the underlying generating graph. Each connected component is a wedge of circles, possibly a point. In particular its reduced homology vanishes above degree one.
\end{prop}
\begin{proof}
A nondegenerate simplex of the nerve is a directed path $P$ in the generating quiver together with cuts separating its nonempty factors. Outside the subcomplex of vertices and single-arrow paths, toggle the cut immediately after the first arrow of $P$. This pairs each cell with a codimension-one face. In an alternating matching path, changing an interior cut other than this distinguished cut cannot continue the matching path; a continuing step must remove the first factor and strictly shorten the underlying path. Hence the matching is acyclic. The quiver is finite and acyclic, so all path lengths are bounded, and the matching collapses the nerve to the generating graph~\citep{forman2002user,kozlov2021organized}.
\end{proof}

Even an uncollapsed regular triangle has seven face objects and nine cover occurrences. Its free-cover nerve has first Betti number three, although the triangle is contractible. Bj\"orner's regular-CW reconstruction uses the face poset with its transitive relations~\citep{bjorner1995topological}, whereas this nerve is that of the free category on the covers. Nonregularity is therefore not the cause of this particular nerve failure.

\begin{prop}[Signed cover data lose attachment information]\label{prop:signed-quiver-counterexample}
Two order-preserving gluings of a common flag source can have isomorphic signed cover quivers and different fundamental groups, even with equal integral cellular chain complexes.
\end{prop}
\begin{proof}
Start with two disjoint ordered triangles $012$ and $345$, a flag source. Identify all six vertices, and give both quotients the edge classes $a,b,c$. On $012$ the facet list is $(b,c,a)$. In the first quotient the second facet list is $(a,c,b)$; in the second it is $(b,c,a)$. These are generated by increasing identifications of corresponding edge representatives (and the specified vertex identifications), so both are order-preserving congruences.

In each quotient, every edge has two opposite-signed covers from the unique vertex, and every two-cell has positive occurrences of $a,b$ and a negative occurrence of $c$. The signed quivers and the cellular matrices are therefore identical after forgetting the indices of the two positive facets. Reading the oriented boundary words gives, respectively,
\[
 \langle a,b,c\mid ab c^{-1},\ ba c^{-1}\rangle\cong\mathbb Z^2, \qquad \langle a,b,c\mid ab c^{-1},\ ab c^{-1}\rangle\cong F_2.
\]
The first is the standard two-triangle torus presentation. In the second, the second disc is attached along a loop already filled by the first, giving homotopy type $(S^1\vee S^1)\vee S^2$. The two fundamental groups differ, proving the assertion. Unlike the nerve example, this gives isomorphic retained signed quivers with different spaces. The assertion is about abstract signed quivers: an additional archive of the source members of each gluing class is extra information, and is not required to be preserved by this quiver isomorphism.
\end{proof}

\subsection*{Worked traces}
Tables~\ref{tab:disc} and~\ref{tab:torus} give the full cell tables for two small examples.
\begin{table}[htbp]
\centering\small
\begin{tabular}{@{}llccc@{}}
\toprule
cell & $\dim$ & $\Dd$: $(\partial_0,\dots,\partial_n)$ & $\Gg$: covers into cell &
$C\Gg$: $\partial$ \\
\midrule
$a=(2,3)$ & $1$ & $(\bbot,\ \bbot)$ & $*\lessdot a$ (two parallel) & $0$ \\
$b=(1,2,3)$ & $2$ & $\bigl((2,3),\ \bbot,\ \bbot\bigr)$ & $a\lessdot b$,
$[a{:}b]=+1$ & $+a$ \\
\bottomrule
\end{tabular}
\caption{Disc witness of Proposition~\ref{prop:nerve} ($K=\Delta^2$, $L$ the two edges through vertex $1$): shallow and strictly graded even though $b$ carries two $\bbot$-facets; $\widetilde H_*=0$, matching $K/L\simeq*$; cautionary $B\Gg\simeq S^1$.}
\label{tab:disc}
\end{table}

\begin{table}[htbp]
\centering\small
\begin{tabular}{@{}llccc@{}}
\toprule
cell & $\dim$ & $\Dd$: $(\partial_0,\partial_1,\partial_2)$ &
$\Gg$: covers into cell & $C\Gg$: $\partial$ \\
\midrule
$Q$ & $0$ & --- & --- & --- \\
$a$ & $1$ & $(Q,\ Q)$ & $Q\lessdot a$ (two parallel) & $0$ \\
$b$ & $1$ & $(Q,\ Q)$ & $Q\lessdot b$ (two parallel) & $0$ \\
$c$ & $1$ & $(Q,\ Q)$ & $Q\lessdot c$ (two parallel) & $0$ \\
$F_1=(1,2,4)$ & $2$ & $\bigl((2,4)\!\to\!b,\ (1,4)\!=\!c,\ (1,2)\!\to\!a\bigr)$ &
$a,b,c\lessdot F_1$ & $a+b-c$ \\
$F_2=(1,3,4)$ & $2$ & $\bigl((3,4)\!\to\!a,\ (1,4)\!=\!c,\ (1,3)\!\to\!b\bigr)$ &
$a,b,c\lessdot F_2$ & $a+b-c$ \\
\bottomrule
\end{tabular}
\caption{Torus of Example~\ref{ex:torus}: $\bbot$ inactive, an ordinary $\Delta$-complex, strictly graded with no flag hypothesis used. Here $\partial_1=0$ (each edge is a loop at $Q$), $\operatorname{im}\partial_2=\langle a+b-c\rangle$ is primitive, and $\ker\partial_2=\langle F_1-F_2\rangle$, giving $H_*=(\mathbb Z,\mathbb Z^2,\mathbb Z)$.}
\label{tab:torus}
\end{table}

\section{Graded forcing and budget lemmas}\label{app:graded}

\begin{defin}[$m$-flag]\label{def:mflag}
Let $m\ge1$. A simplicial complex $M$ is \emph{$m$-flag} if every missing face (minimal non-face) of $M$ has at most $m+1$ vertices; equivalently, if $\sigma\subseteq V(M)$ lies in $M$ as soon as every subset of $\sigma$ of cardinality $\le m+1$ does. Thus an $m$-flag complex is determined by its $m$-skeleton, and $1$-flag $=$ flag. In the notation of the face-number literature this is the class $\mathcal F_{m+1}$ \citep{nevo2008missing,goff2011balanced,adamaszek2013extremal}. Write $\Cl_m(\,\cdot\,)$ for the $m$-clique operator: $\Cl_m(X)$ is the largest $m$-flag complex with $m$-skeleton $X$.
\end{defin}

\begin{defin}[$\Flag^{(m)}$]\label{def:mflagcat}
$\Flag^{(m)}$ has objects the pairs $(K,L)$ with $K$ a finite ordered $m$-flag complex and $L\subseteq K$ an $m$-flag subcomplex; nondegenerate pair maps as in Definition~\ref{def:flagcat}, with the reflecting subcategory specified separately. Thus $\Flag=\Flag^{(1)}$. As in Lemma~\ref{lem:flag-subcomplexes}, the $m$-flag subcomplexes of $K$ are exactly the $\Cl_m(X)$ for arbitrary subcomplexes $X\subseteq\sk_m K$.
\end{defin}

\begin{lem}[facet budget, graded]\label{lem:budget-m}
Let $L\subseteq K$ with $L$ $m$-flag ($K$ arbitrary), and let $\sigma\in F(K,L)=K\smallsetminus L$ with $\dim\sigma=d\ge m$. Put
\[
N_m(\sigma)\;=\;\{\rho\subseteq\sigma\;:\;1\le|\rho|\le m+1,\ \rho\notin L\},
\]
the small non-faces of $L$ inside $\sigma$. Then:
\begin{enumerate}[label=(\roman*)]
\item $N_m(\sigma)\neq\varnothing$;
\item $\partial_i\sigma=\bbot$ (equivalently, by \cref{con:D}, $\genface{i}{}\sigma=\sigma\smallsetminus\{v_i\}$ lies in $L$) if and only if $v_i\in\rho$ for every $\rho\in N_m(\sigma)$;
\item hence $\sigma$ has exactly $\bigl|\bigcap N_m(\sigma)\bigr|\le m+1$ collapsed facets, and at least $d-m$ genuine ones.
\end{enumerate}
The bound $m+1$ is attained: on $V=\{v_0,\dots,v_d\}$ take $K=2^{V}$ and $L=\{\tau\subseteq V:\rho_0\not\subseteq\tau\}$ for a fixed $\rho_0$ with $|\rho_0|=m+1$; then $L$ is $m$-flag with unique missing face $\rho_0$, and the collapsed facets of $\sigma=V$ are exactly those indexed by $\rho_0$.
\end{lem}
\begin{proof}
(i) $\sigma\notin L$ and $L$ is $m$-flag, so $\sigma$ contains a non-face of $L$ of cardinality $\le m+1$.

(ii) ($\Rightarrow$) If $\sigma\smallsetminus\{v_i\}\in L$ then every $\rho\subseteq\sigma$ with $v_i\notin\rho$ satisfies $\rho\subseteq \sigma\smallsetminus\{v_i\}$, hence $\rho\in L$ by face-closedness; so no such $\rho$ lies in $N_m(\sigma)$. ($\Leftarrow$) If every $\rho\in N_m(\sigma)$ contains $v_i$, then every subset of $\sigma\smallsetminus\{v_i\}$ of cardinality $\le m+1$ lies in $L$, so $\sigma\smallsetminus\{v_i\}\in L$ by $m$-flagness of $L$.

(iii) By (i) the family $N_m(\sigma)$ is non-empty and all its members have at most $m+1$ elements, so $\bigl|\bigcap N_m(\sigma)\bigr|\le m+1$; subtract from the $d+1$ facets.
\end{proof}

\begin{lem}[forcing, graded]\label{lem:forcing-m}
Let $K$ be $m$-flag, $L\subseteq K$ face-closed, and $\sigma\subseteq V(K)$ a vertex set at least $m+2$ of whose facets $\sigma\smallsetminus\{v\}$ lie in $F(K,L)$. Then $\sigma\in F(K,L)$.
\end{lem}
\begin{proof}
Let $\tau_j=\sigma\smallsetminus\{v_j\}\in F(K,L)$ for $m+2$ distinct $v_1,\dots,v_{m+2}\in\sigma$. Let $\rho\subseteq\sigma$ with $|\rho|\le m+1$. Since $m+2>|\rho|$, some $v_j\notin\rho$, whence $\rho\subseteq\tau_j\subseteq K$ and $\rho\in K$. So every subset of $\sigma$ of cardinality $\le m+1$ lies in $K$, and $\sigma\in K$ by $m$-flagness. Finally $\sigma\notin L$, since otherwise $\tau_1\in L$ by face-closedness.
\end{proof}

\begin{thm}[graded strictness]\label{thm:strict-m}
Let $K$ be a finite ordered $m$-flag complex and $L\subseteq K$ a subcomplex. The following are equivalent:
\begin{enumerate}[label=(\arabic*)]
\item every cover of $\preceq$ raises rank by at most $m$;
\item $L$ is $m$-flag;
\item every hollow simplex of $(K,L)$ has dimension at most $m$.
\end{enumerate}
The covers of jump $\ge2$ correspond to the minimal missing faces of $L$ of dimension $\ge2$ that belong to $K$, each giving a cover from its boundary component point of jump $\dim\tau$. The case $m=1$ is Theorem~\ref{thm:strict}: jump at most $1$ means no skips at all.
\end{thm}
\begin{proof}
By Lemma~\ref{lem:skip-hollow} the covers of jump $\ge2$ are exactly the hollow simplices, and a hollow simplex is precisely a missing face of $L$ lying in $K$; this gives $(1)\Leftrightarrow(3)$ and the bijection. $(2)\Rightarrow(3)$ is Definition~\ref{def:mflag}. For $(3)\Rightarrow(2)$, let $\tau$ be a missing face of $L$ with $|\tau|\ge m+2$. Every proper subset of $\tau$ lies in $L\subseteq K$, in particular every subset of cardinality $\le m+1$, so $\tau\in K$ because $K$ is $m$-flag; then $\tau$ is a hollow simplex of dimension $\ge m+1$.
\end{proof}

\begin{prop}[labelled, dimension-bounded]\label{prop:rigidity-dim} 
Let $K,K'$ be flag on a common vertex set $V$, and $L\subseteq K$, $L'\subseteq K'$ subcomplexes with $\dim L,\dim L'\le\ell$, with $\ell\ge0$. If $F(K,L)$ and $F(K',L')$ agree in dimensions $\le\ell+1$, they agree in all dimensions.
\end{prop}
\begin{proof}
Induct on $d\ge\ell+2$. For $\sigma\in F(K,L)$ of dimension $d$, every facet has dimension $d-1\ge\ell+1>\ell$, hence lies outside $L$: all $d+1\ge3$ facets are genuine, and by induction lie in $F(K',L')$. Lemma~\ref{lem:forcing} gives $\sigma\in F(K',L')$.
\end{proof}

\begin{lem}[intrinsic forcing]\label{lem:forcing-intrinsic}
Let $K$ be flag and $\tau_1,\tau_2,\tau_3\in K$ of dimension $d-1$, pairwise intersecting in $(d-2)$-simplices, with the three intersections $\tau_1\cap\tau_2$, $\tau_1\cap\tau_3$, $\tau_2\cap\tau_3$ pairwise distinct. Then $\sigma:=\tau_1\cup\tau_2$ has $d+1$ vertices, each $\tau_i$ is a facet of $\sigma$, and $\sigma\in K$.
\end{lem}
\begin{proof}
Write $r=\tau_1\cap\tau_2$, so $\tau_1=r\cup\{a\}$, $\tau_2=r\cup\{b\}$, $\sigma=r\cup\{a,b\}$ with $|\sigma|=d+1$. Since $|\tau_3\cap\tau_1|=d-1$ there is exactly one $x\in\tau_3\smallsetminus\tau_1$, and exactly one $y\in\tau_3\smallsetminus\tau_2$. If $x=y$ then $\tau_3\smallsetminus\{x\} \subseteq r$ and, by cardinality, $\tau_3=r\cup\{x\}$, whence $\tau_3\cap\tau_1=r=\tau_1\cap\tau_2$, excluded. So $x\neq y$, forcing $x\in\tau_2\smallsetminus\tau_1=\{b\}$ and $y\in\tau_1\smallsetminus\tau_2 =\{a\}$; hence $a,b\in\tau_3\subseteq\sigma$ and $\tau_3$ is a facet of $\sigma$. Now Lemma~\ref{lem:forcing} applies.
\end{proof}

\begin{rem}[the dichotomy]\label{rem:dichotomy-skel}
Collecting the results of this section:
\begin{center}
\begin{tabular}{@{}lll@{}}
\toprule
regime & determined by & sharpness\\
\midrule
collapse, $L$ flag & cells of $\dim\le3$ & Proposition~\ref{prop:2skel}\\
collapse, $L$ $m$-flag & cells of $\dim\le 2m+1$ & Proposition~\ref{prop:sharp-m}\\
collapse, $K$ flag, $\dim L\le\ell$ & cells of $\dim\le\ell+1$ & $\ell\le2$: Prop.~\ref{prop:skeleta},~\ref{prop:2skel}\\
gluing datum & no finite skeleton & Proposition~\ref{prop:glue-noskel}\\
\bottomrule
\end{tabular}
\end{center}
The third line has a threshold that grows with $\dim L$. Without a uniform bound on missing-face size or collapsed dimension there is no fixed reconstruction threshold for this class. The witnessing family also shows that the common-vertex-set hypothesis of Proposition~\ref{prop:rigidity-dim} cannot be dropped. For $d\ge1$ compare
\[
\bigl(\Delta^{d+2},\ \sk_d\Delta^{d+2}\bigr) \qquad\text{with}\qquad \bigl(K_2,\ \sk_dK_2\bigr),
\]
where $K_2$ is $d+3$ disjoint copies of $\Delta^{d+1}$ strung together by a path of edges joining them (no new clique appears, since the endpoints of each new edge have no common neighbour, so $K_2$ is flag and is exactly the $d+3$ simplices plus the path). In both quotients every surviving cell of dimension $d+1$ has all its facets in the collapsed $d$-skeleton, so both $(d+1)$-skeleta are $\bigvee_{d+3}S^{d+1}$ and are isomorphic as objects of $\FFace$; but the first carries one further cell, of dimension $d+2$, whose boundary is the alternating sum of the $d+3$ others, giving $H_{d+1}=\mathbb Z^{d+2}$ against $\mathbb Z^{d+3}$. The two pairs live on vertex sets of different sizes, so this is a counterexample to the label-free analogue of Proposition~\ref{prop:rigidity-dim} and not to the proposition itself. Note also that neither $L$ here is flag: $\sk_d$ of a simplex has missing faces of size $d+2$, so these are $(d+1)$-flag pairs, consistent with Theorem~\ref{thm:rigidity-m}. The hypothesis on $L$ makes the threshold uniform whatever $K$ is, and flagness of $L$ makes it $3$.

A collapse is specified by the subcomplex $L$, which, being flag, is determined by $1$-dimensional data, so only bounded information can hide above the visible range. A gluing datum is an equivalence relation on cells of \emph{every} dimension, and Definition~\ref{def:gluing}(ii) constrains it only downwards: identifications in the top dimension are free choices, invisible below. Thus the class of gluing presentations has no uniform fixed-skeleton reconstruction threshold, whereas Theorem~\ref{thm:rigidity-m} gives a positive compression statement for the specified collapse class (Corollary~\ref{cor:store}).
\end{rem}

\section{The convex model and the ball question}\label{app:convex}

\subsection{Proof of \texorpdfstring{\cref{lem:faces}}{the crushing lemma}}
\begin{proof}
Complete the vertex sets of the crushed faces to a partition $V(\Delta)=V_1\sqcup\dots\sqcup V_m$ by making every unused vertex a singleton block; singleton blocks create no nontrivial identification. Set $n_i=\lvert V_i\rvert$ and $N=\sum_i n_i=d+1$. Because the crushed faces are proper and disjoint, $m\ge2$.

For barycentric coordinates $\lambda_v\ge0$, $\sum_v\lambda_v=1$, put
\[
 t_i=\sum_{v\in V_i}\lambda_v,\qquad a_i=\min\{t_i,1-t_i\},\qquad y_{iv}=
 \begin{cases}
 a_i\bigl(\tfrac{\lambda_v}{t_i}-\tfrac1{n_i}\bigr),&t_i>0,\\[1mm]
 0,&t_i=0.
 \end{cases}
\]
The apparent division at $t_i=0$ is harmless, since $\lvert y_{iv}\rvert\le a_i\le t_i$; the same bound with $a_i\le1-t_i$ gives continuity at $t_i=1$. So $\Phi\colon\lambda\mapsto(t,y)$ is continuous.

Its image is exactly the set $P$ cut out by
\[
 t_i\ge0,\qquad \sum_i t_i=1,\qquad \sum_{v\in V_i}y_{iv}=0,\qquad y_{iv}\ge-\frac{t_i}{n_i},\qquad y_{iv}\ge-\frac{1-t_i}{n_i},
\]
the last two together saying $y_{iv}\ge-a_i/n_i$. When $0<t_i<1$ a point of $P$ has the unique preimage
\[
 \lambda_v=t_i\Bigl(\frac1{n_i}+\frac{y_{iv}}{a_i}\Bigr)
\]
on that block, with nonnegative coordinates summing to $t_i$. If $t_i=0$ the zero-sum and lower bounds force $y_i=0$ and all those $\lambda_v$ vanish. If $t_i=1$ all other block masses vanish and $y=0$, and the entire face $\Delta(V_i)$ is the fibre. These are exactly the prescribed nonsingleton fibres; in particular no internal coordinates of a block are merged at mixed mass $0<t_i<1$.

The defining inequalities are affine, so $P$ is convex, and it is bounded, since $t_i\in[0,1]$ and the lower bounds together with the zero-sum conditions bound each $y_{iv}$ above. Its affine dimension is $(m-1)+\sum_i(n_i-1)=N-1=d$, and the point $t_i=1/m$, $y_{iv}=0$ satisfies every inequality strictly there: the two lower bounds read $-1/(mn_i)<0$ and $-(m-1)/(mn_i)<0$, the second using $m\ge2$. So $P$ is a full-dimensional compact convex $d$-polytope, and $\Phi$ induces a continuous bijection from the compact quotient onto $P$, hence a homeomorphism.

Finally $\lambda_v>0$ for all $v$ if and only if every defining inequality is strict in the affine hull: a missing coordinate in a mixed block is exactly $y_{iv}=-a_i/n_i$, a missing block is $t_i=0$, and a pure block is $t_i=1$. So $q(\relint\Delta)$ is carried onto $\Int P$ and $q(\partial\Delta)$ onto $\partial P$; radial parametrization about an interior point gives the ball pair.
\end{proof}

\begin{rem}[why merging barycentric coordinates is not enough]\label{rem:why-deviations}
Keeping only $(t_1,\dots,t_m)$ would also identify distinct points with $0<t_i<1$, collapsing mixed faces and the interior along with the intended ones. The deviation variables $y_{iv}$ retain exactly those directions and vanish only when the whole mass of a block is concentrated on the face being crushed. The argument therefore computes the quotient space itself, which a dimension or homology count could not do.
\end{rem}

\begin{cor}[simplex quotients]\label{cor:simplex-ball}
Let $d\ge1$ and $A\subsetneq\Delta^d$. The prescribed quotient cells form a regular CW decomposition of a $d$-ball if and only if the connected components of $A$ are nonempty proper faces with pairwise disjoint vertex sets. These faces must be the components themselves, and a presentation of $A$ as some other union of faces does not suffice. Singleton components are allowed, and $A=\varnothing$ returns the simplex.
\end{cor}
\begin{proof}
Every full subcomplex of a simplex is a face. Apply Theorem~\ref{thm:fullness} and Lemma~\ref{lem:faces}; properness retains the open $d$-cell.
\end{proof}

\begin{cor}[all-cell acyclicity]\label{cor:all-cell-acyclicity}
Fix any field $\Bbbk$. For any finite simplicial pair $A\subseteq K$ the following are equivalent:
\begin{enumerate}[label=(\roman*)]
\item every component $A_\alpha$ is full in $K$;
\item for every surviving $\sigma\in K\smallsetminus A$ and every $\alpha$, the intersection $\lvert\sigma\rvert\cap\lvert A_\alpha\rvert$ is empty or $\Bbbk$-acyclic.
\end{enumerate}
The same holds with integral acyclicity. No flag hypothesis is needed.
\end{cor}
\begin{proof}
Under fullness every nonempty intersection is a face, hence contractible. Otherwise the minimal missing $\tau$ of Theorem~\ref{thm:fullness} has $\lvert\tau\rvert\cap\lvert A_\alpha\rvert=\lvert\partial\tau\rvert\cong S^{r-1}$, whose reduced homology in degree $r-1$ is nonzero over every field and over $\mathbb Z$; for $r=1$ the obstruction is the reduced $H_0$ of two points.
\end{proof}

\begin{rem}[the quantifier is what matters]\label{rem:quantifier}
For flag $A$ the edge criterion is also equivalent to connectedness of every nonempty $\lvert\sigma\rvert\cap\lvert A_\alpha\rvert$ \emph{for every surviving $\sigma$}: a missing edge inside one component has disconnected intersection, and fullness makes all such intersections faces. Testing only maximal simplices is not enough, as the chordless four-cycle of Example~\ref{ex:four-cycle} shows. Without flagness even the all-simplices connectedness test fails, by Example~\ref{ex:boundary-not-flag}. Corollary~\ref{cor:all-cell-acyclicity} repairs this, again with the condition at every surviving simplex.
\end{rem}

\subsection{Three properties, and how they are related}

For a proper $A\subsetneq\Delta^d$ distinguish
\[
\begin{aligned}
 \mathsf E&:\ \text{all inherited simplex characteristic maps are embeddings};\\
 \mathsf R&:\ \mathcal E(\Delta^d,A)\text{ is a regular CW decomposition};\\
 \mathsf B&:\ \bigl(\Delta^d/_{SC}A,\ q(\partial\Delta^d)\bigr)\cong(B^d,S^{d-1}).
\end{aligned}
\]

\begin{defin}[iterated faces]\label{def:iterated}
For $S\in\FFace$, $\sigma\in S_d$ and a non-empty $T\subseteq\{0,\dots,d\}$, let $\partial^{T}\sigma$ denote the iterated face obtained by applying $\partial_i$ for every $i\notin T$, in decreasing order of $i$. This is well defined by the semisimplicial identities, and $\dim\partial^{T}\sigma=\lvert T\rvert-1$ whenever the value is genuine; $T=\{0,\dots,d\}$ returns $\sigma$ itself. Thus $\partial^{T}\sigma$ is the face of $\sigma$ spanned by the positions in $T$, tracked through the decoration, and it is $\bbot_\alpha$ as soon as any stage of the descent is.
\end{defin}

\begin{prop}[embedding criterion for the inherited maps]\label{prop:regular}
Let $S\in\FFace$. The canonical characteristic maps of $\lvert S\rvert_{\bullet}$ supplied by Lemma~\ref{lem:cells} are embeddings if and only if for every $\sigma\in S_d$ with $d\ge1$:
\begin{enumerate}[label=(R\arabic*)]
\item $\partial^{T}\sigma\neq\bbot_\alpha$ for every $\alpha$ and every $T$ with $|T|\ge2$ (no positive-dimensional face is crushed);
\item $T\mapsto\partial^{T}\sigma$ is injective on non-empty proper subsets $T\subsetneq\{0,\dots,d\}$, distinct basepoints counting as distinct $0$-cells.
\end{enumerate}
For a collapse this says exactly that $\lvert\sigma\cap V(A_\alpha)\rvert\le1$ for every surviving $\sigma$ and every $\alpha$, i.e.\ that $q$ is injective on every closed simplex outside $A$.
\end{prop}
\begin{proof}
By Definition~\ref{def:real-pointed} and Lemma~\ref{lem:cells} the attaching map $\Phi^\sigma\colon\partial\Delta^{d}\to\lvert S\rvert^{(d-1)}_\bullet$ restricts, on the open face indexed by $T$, to the characteristic map of $\partial^{T}\sigma$ on the open disc when that value is genuine and to the constant map at $\bbot_\alpha$ otherwise. Characteristic maps are homeomorphisms from open discs onto open cells, so each restriction is injective precisely under (R1); the images, being open cells, are equal or disjoint, so they are pairwise disjoint precisely under (R2). Compactness of $\partial\Delta^d$ and Hausdorffness give an embedding. For the collapse form: two vertices of $\sigma$ in one $V(A_\alpha)$ violate (R2), while under the stated bound a surviving closed simplex meets each $\lvert A_\alpha\rvert$ in at most one point, since a nonempty intersection contains its vertices.
\end{proof}

\begin{prop}[strict hierarchy]\label{prop:hierarchy}
For every $d\ge1$ and $A\subsetneq\Delta^d$ one has $\mathsf E\Rightarrow\mathsf R\Rightarrow\mathsf B$. Neither converse holds in general, already in dimension two. The three properties are therefore distinct but not independent.
\end{prop}
\begin{proof}
Embedded inherited characteristic maps are in particular regular ones, giving the first implication. Since $A$ is proper the top open $d$-cell survives; its closure is the whole quotient and its complement in that closure is $q(\partial\Delta^d)$, so Lemma~\ref{lem:closed-cell} gives the second. Examples~\ref{ex:bigon} and~\ref{ex:arc} refute the converses. For a general source $K$ the second implication is unavailable: regularity does not make the whole quotient a ball, since there need be no surviving top cell whose closure is everything.
\end{proof}

\begin{ex}[one edge of a triangle: regular cells, non-embedded source maps]\label{ex:bigon}
Let $K=\Delta(abc)$ and $A=\Delta(ab)\sqcup\{c\}$. Both components are full, so the quotient is regular: two $0$-cells, two $1$-cells and one $2$-cell, a bigon disc. Yet $abc$ meets $A_1=\Delta(ab)$ in two vertices, so the inherited triangle map is not injective and $\mathsf E$ fails. Concretely, writing a point of the triangle as $(1-t)\bigl((1-s)a+sb\bigr)+tc$, the map
\[
 (s,t)\longmapsto\bigl(2(2s-1)\sqrt{t(1-t)},\ 2t-1\bigr)
\]
descends to a continuous bijection onto the closed unit disc, since at height $y=2t-1$ the disc has half-width $\sqrt{1-y^2}=2\sqrt{t(1-t)}$, exactly the range of the first coordinate; compactness makes it a homeomorphism, and the two surviving edges become the boundary semicircles.
\end{ex}

\begin{ex}[a collapsible boundary arc: a ball with irregular cells]\label{ex:arc}
In $\Delta(abc)$ crush $A=ab\cup bc$ to one point. The quotient is a $2$-ball, but the surviving edge $ac$ has both endpoints at the collapsed point, so its closure is $S^1$ rather than a closed interval and $\mathsf R$ fails. Here $A$ is flag, connected and collapsible, and still not full in $K$. Thus $\mathsf B\not\Rightarrow\mathsf R$, while Example~\ref{ex:bigon} gives $\mathsf R\not\Rightarrow\mathsf E$.
\end{ex}

All four possibilities that occur are realized on the one triangle:
\begin{center}
\begin{tabular}{@{}lccc@{}}
\toprule
$A\subseteq\Delta(abc)$ & $\mathsf E$ & $\mathsf R$ & $\mathsf B$\\
\midrule
$\{a\}\sqcup\{b\}\sqcup\{c\}$ & yes & yes & yes\\
$\Delta(ab)\sqcup\{c\}$ & no & yes & yes\\
$ab\cup bc$ & no & no & yes\\
$\partial\Delta^2$ & no & no & no\\
\bottomrule
\end{tabular}
\end{center}
The first row makes no nontrivial identification; the last collapses the whole boundary and yields $S^2$.

\begin{ex}[two opposite edges of a tetrahedron]\label{ex:opposite-edges}
Crush $\Delta(ab)$ and $\Delta(cd)$ in $\Delta(abcd)$ separately. The faces are disjoint and full, so the quotient is a regular ball with cells $(f_0,f_1,f_2,f_3)=(2,4,4,1)$; the four $2$-cells are bigons and the boundary has $\chi=2$. In the convex model of Lemma~\ref{lem:faces} there are two blocks of size two, and with $t=t_1=1-t_2$ and one deviation coordinate per block the cross-sections are squares of side proportional to $\min(t,1-t)$: the polytope $P$ is a square bipyramid. The vertices of $P$ are not cells of the quotient decomposition: $P$ is an auxiliary model, and only its interior/boundary pair is used.
\end{ex}

\begin{ex}[the chordless four-cycle]\label{ex:four-cycle}
In $\partial\Delta^3$ let $C_4=ab\cup bc\cup cd\cup da$, a flag circle separating the boundary sphere into two discs, so $\partial\Delta^3/C_4\cong S^2\vee S^2$. For the solid tetrahedron instead $\Delta^3/C_4\simeq\Sigma C_4\simeq S^2$, a homotopy statement and not a homeomorphism to the boundary quotient. The missing diagonals $ac$ and $bd$ are already surviving loop edges, so Corollary~\ref{cor:regular-edges} reports irregularity without any reference to the top cell: connectedness of the intersection with the maximal simplex alone would have missed it.
\end{ex}

\begin{ex}[why the edge test needs flagness]\label{ex:boundary-not-flag}
Take $A=\partial\Delta^2\subset K=\Delta^2$. Every edge is collapsed, so there is no surviving loop edge; yet the surviving $2$-cell has closure $S^2$. Theorem~\ref{thm:fullness} detects this through the minimal missing $2$-simplex. Here $A$ is not flag, and, consistently with Theorem~\ref{thm:strict}, this is also the archetype of a rank-skipping cover.
\end{ex}

\begin{rem}[what fullness does \emph{not} bound]\label{rem:not-mild}
Lemma~\ref{lem:budget} bounds the number of facets of a cell lying \emph{wholly} in $L$, and it is tempting to read this as a bound on how far the attaching map is from injective. The facet-budget bound does not bound the number of noninjective facet restrictions. Take $K=\Delta^{d}$ on $\{u,x,v,w_1,\dots,w_{d-2}\}$ with $L$ the flag path $u\!-\!x\!-\!v$: for $d\ge3$ no facet of the top cell lies in $L$, so the facet budget is untouched, yet every facet contains two of $u,x,v$ and its restriction identifies them, so all $d+1$ restrictions fail to be injective. Non-injectivity is governed by $\lvert\sigma\cap V(A_\alpha)\rvert$ (Proposition~\ref{prop:regular}), which flagness does not bound. What the encoding exploits is therefore not mildness of the failure but Corollary~\ref{cor:skeletal}: the attaching map, however far from injective, is still determined by the ordered facet list.
\end{rem}

\subsection{The weaker ball question, and why the obvious tests fail}\label{sec:ball-question}

Theorem~\ref{thm:fullness} settles regularity of the prescribed cells. Whether the quotient of a single ball is a ball is a weaker question with a different answer, and we record where the natural candidate criteria fail.

For $n\ge1$, a nonempty proper compact $C\subset S^n$ is \emph{cellular} if it is the intersection of a nested sequence of closed $n$-balls, $B_{k+1}\subset\Int B_k$. Brown's theorem gives $S^n/C\cong S^n$ exactly when $C$ is cellular~\citep{Brown,Uspenskij}, and the corresponding statement for $d\ge2$ and finitely many disjoint nonempty proper compact $C_i\subset\partial B^d$ crushed to distinct points is that $\bigl(B^d/\{C_i\},q(S^{d-1})\bigr)\cong(B^d,S^{d-1})$ if and only if every $C_i$ is cellular; the necessity uses the generalized Schoenflies theorem on a bicollared sphere around each exceptional fibre, and the sufficiency is the standard shrinking construction. Cellularity depends on the embedding and is not determined by the homotopy type.

For the following homotopy and homology statements, assume that the $C_i$ are finite boundary subcomplexes, so that the required inclusions are CW cofibrations. The preceding cellularity statement itself permits arbitrary compact sets. First, homology is not enough: attaching a cone on each $C_i$ and collapsing gives
\[
 B^d/\{C_i\}\;\simeq\;\bigvee_i\Sigma C_i,\qquad \widetilde H_k\bigl(B^d/\{C_i\}\bigr)\cong\bigoplus_i\widetilde H_{k-1}(C_i),
\]
so acyclicity of each $C_i$ is necessary but not sufficient. Second, acyclicity does not suffice: the presentation complex of $\langle x,y\mid x^3y^{-5},\,x^3(xy)^{-2}\rangle$ has $\partial_2=\left(\begin{smallmatrix}3&1\\-5&-2\end{smallmatrix}\right)$ of determinant $-1$ and $\partial_1=0$, hence is integrally acyclic, while $x\mapsto(1\,2\,3)$, $y\mapsto(1\,3\,4\,2\,5)$ gives $xy=(2\,5)(3\,4)$ and a nontrivial quotient of its fundamental group. Replace the presentation complex by a finite simplicial homotopy model before embedding it in a sufficiently large simplex boundary. If a finite polyhedron $C$ is cellular, some ball $B$ with $C\subset B$ lies inside a neighbourhood $U$ that retracts onto $C$; then $C$ is a retract of the contractible $B$, hence contractible. Thus this model is not cellular in any embedding, and crushing it as a boundary subcomplex does not yield a ball, even though $\Sigma C$ is simply connected and acyclic, hence contractible. Third, collapsibility is sufficient in the PL setting, since a collapsible PL subpolyhedron has ball regular neighbourhoods~\citep{RourkeSanderson}, but not necessary: Benedetti and Lutz exhibit a noncollapsible dunce hat in a PL $3$-sphere onto which a $3$-ball collapses~\citep{BenedettiLutz}, which is therefore cellular without admitting a collapsible triangulation.

None of this weakens Theorem~\ref{thm:fullness}, which is why we proved it by the explicit model of Lemma~\ref{lem:faces} rather than by ball recognition. The summary for a simplex source is:
\begin{center}
\begin{tabular}{@{}p{.40\textwidth}p{.52\textwidth}@{}}
\toprule
Property requested & Criterion\\
\midrule
$\mathsf E$: inherited maps are embeddings & every surviving simplex meets each component in at most one vertex (Prop.~\ref{prop:regular})\\[1mm]
$\mathsf R$: prescribed cells are regular & every component is full; for a simplex source, a proper face (Thm.~\ref{thm:fullness}, Cor.~\ref{cor:simplex-ball})\\[1mm]
$\mathsf B$: the quotient is a ball pair & the crushed sets are cellular in the boundary sphere\\[1mm]
only the reduced homology of a ball & each crushed subcomplex acyclic over the chosen coefficients\\
\bottomrule
\end{tabular}
\end{center}
Nothing here forces a change to the stored records: the implementation may keep the rank-dropping presentation, since reparametrizing a cell as a regular ball is a different operation from altering the stored source map.

\section{The gluing regime}\label{app:gluing}

We distinguish two presentation classes. The first uses the fixed orders of source simplices; the second retains affine vertex bijections explicitly. In either case the quotient identifies points of $|K|$ and does more than rename simplices.

\begin{defin}[order-preserving gluing datum]\label{def:gluing}
An order-preserving gluing datum on an ordered complex $K$ is an equivalence relation $\sim$ on its nonempty simplices such that
\begin{enumerate}[label=(\roman*),leftmargin=2em]
\item $\sigma\sim\sigma'$ implies $\dim\sigma=\dim\sigma'$;
\item $\sigma\sim\sigma'$ implies $\genface{i}{K}\sigma\sim\genface{i}{K}\sigma'$ for every local facet index $i$.
\end{enumerate}
The identified simplices are matched by the increasing bijection of their vertex lists and its affine extension. The space $|K|/{\sim}$ means the equivalence relation generated by these affine identifications.
\end{defin}

\begin{rem}[congruences and repeated gluings]\label{rem:not-iso}
This is a semisimplicial congruence, a standard coequalizer construction. It can be generated by several overlapping face identifications, rather than a single isomorphism between disjoint subcomplexes. In the torus example, one corner is identified with two different corners. The least congruence containing a union of order-preserving gluing relations is again a congruence, and this is what composing these operations means. After the first quotient the object need not be an abstract simplicial complex. The next quotient is supported by the semisimplicial presentation, without a new global order on quotient vertices.
\end{rem}

\begin{defin}[coherent gluing datum with permutations]\label{def:gluing-perm}
A permuted gluing datum is an equivalence relation on the simplices of $K$, preserving dimension, together with a bijection $\pi_{\sigma\sigma'}:V(\sigma)\to V(\sigma')$ for every related pair, such that
\[
 \pi_{\sigma\sigma}=\id,\qquad \pi_{\sigma'\sigma}=\pi_{\sigma\sigma'}^{-1},\qquad \pi_{\sigma'\sigma''}\pi_{\sigma\sigma'}=\pi_{\sigma\sigma''}.
\]
For every face $\tau\subseteq\sigma$ one requires $\tau\sim\pi_{\sigma\sigma'}(\tau)$ and the assigned face bijection to be the restriction of $\pi_{\sigma\sigma'}$. The affine maps generate the quotient of $|K|$. In particular, a loop of identifications of any simplex must induce the identity on its vertices; this is the \emph{trivial self-holonomy} requirement.
\end{defin}

The order-preserving datum is the special case with increasing bijections. The permutation extension is not an ordinary semisimplicial quotient with the unchanged cell inventory. It has the following separate CW realization statement.

\begin{prop}[CW realization of coherent permuted gluings]\label{prop:permuted-cw}
A datum of Definition~\ref{def:gluing-perm} gives a finite CW quotient with one open $d$-cell for each equivalence class of source $d$-simplices. Choose one ordered representative $\sigma$ in each class. Its characteristic map is the source simplex followed by the quotient; on a facet $\tau$ it is the characteristic map of the chosen representative $r(\tau)$ composed with the affine map induced by $\pi_{\tau,r(\tau)}$. These codimension-one affine maps determine all attachments. The cell-face order is strictly graded.
\end{prop}
\begin{proof}
An identification carries a face interior homeomorphically to an interior of the same dimension. Trivial self-holonomy ensures that no two distinct points in the interior of a chosen representative become equal. Restrictions to a lower face agree because the prescribed bijections are coherent under restriction and composition. One can therefore attach the representative simplices in dimension order, using those boundary maps. Each finite adjunction is a CW adjunction, and it identifies exactly the points required by the original relation. The resulting map from the compact source quotient to this Hausdorff CW space is a continuous bijection and hence a homeomorphism. Every proper source face factors through a sequence of facets; all dimensions survive under a gluing, so any relation spanning more than one dimension has an intermediate class. Thus no cover skips a dimension.
\end{proof}

\begin{prop}[A permutation quotient outside the unchanged semisimplicial inventory]\label{prop:permuted-not-ssset}
The coherent permuted class contains a quotient that has no ordinary semisimplicial presentation with the same open-cell inventory.
\end{prop}
\begin{proof}
In one triangle with vertices $0,1,2$, identify its cyclically directed sides by
\[
 (0,1)\longmapsto(1,2)\longmapsto(2,0)\longmapsto(0,1).
\]
Their threefold composite is the identity, and restrictions to vertices are coherent. The quotient has one vertex, one edge $a$, and one two-cell attached by $a^3$. Hence $\partial_2=3$ and $H_1=\mathbb Z/3$. With one ordinary semisimplicial simplex in each of degrees $0,1,2$ and none higher, all three faces of the two-simplex would instead be $a$. Its boundary is $a-a+a=a$, giving $H_1=0$. Thus the unchanged ordinary semisimplicial model is impossible.
\end{proof}

\begin{rem}[Scope and refinement]\label{rem:permuted-scope}
The realization theorem for order-preserving congruences (\cref{thm:realization}) is a semisimplicial coequalizer theorem, and \cref{prop:permuted-cw} is its separate affine-CW counterpart. It cannot be obtained by declaring a permuted relation to commute with the original ordered face maps. Symmetric $\Delta$-complexes provide a related established enlargement with symmetric-group actions~\citep{AllcockCoreyPayne}. The present thin coherence excludes nontrivial simplex stabilizers. Combining arbitrary lists of permutation generators can violate that coherence, so an implementation has to test self-holonomy.

Alternatively subdivide $K$ barycentrically. Its vertices are the nonempty source faces; order them by dimension, with any tie-breaking order. A simplex is a chain of faces of distinct dimensions. The induced gluing maps preserve those dimensions and are increasing on each chain. Hence the refined datum is order-preserving and has the same topological quotient. This changes the cell inventory and the cost. Full permutations are needed only to retain a particular unrefined presentation, and they produce no new homeomorphism types.
\end{rem}

\begin{prop}[Compatible orders for cyclic bipyramid gluings]\label{prop:lens-orders}
Let $p\ge4$, $1\le q<p$, and let $C_p$ have cyclically indexed vertices $c_0,\ldots,c_{p-1}$. Set
\[
 B_p=v*(C_p*S^0),\qquad S^0=\{n,s\}.
\]
This is a flag three-ball. Prescribe its boundary triangle maps by
\[
 g_i(c_i)=c_{i+q},\quad g_i(c_{i+1})=c_{i+q+1},\quad g_i(n)=s, \qquad i\pmod p.
\]
There is a total order of source vertices making every $g_i$ increasing if and only if $\gcd(p,q)>1$.
\end{prop}
\begin{proof}
Let $\epsilon_i=+1$ when $c_i<c_{i+1}$ and $-1$ otherwise. An increasing $g_i$ requires $\epsilon_i=\epsilon_{i+q}$. If $\gcd(p,q)=1$, translation by $q$ is transitive on indices, so all signs coincide. This would impose a strict cyclic chain, in one direction or the other, which no total order admits.

Conversely put $g=\gcd(p,q)>1$ and take a nonconstant sign pattern of period $g$, for example $\epsilon_i=+1$ precisely for $i\equiv0\pmod g$. Orient the cycle accordingly. The orientation is acyclic, since the only cycle in its underlying graph would be directed only if every sign coincided. Take a topological ordering of its vertices and put both $n,s$ after them. Corresponding cycle edges have the same relative order and both apices occupy the last triangle position. All maps are therefore increasing. The interior apex $v$ can be placed anywhere.
\end{proof}

\begin{rem}[A fixed-presentation obstruction]\label{rem:order-preserving-narrow}
The classical bipyramid construction gives $L(p,q)$ for a coprime pair, with the convention for rotations in~\citep{KorepanovMartyushev}. Proposition~\ref{prop:lens-orders} rules out an increasing order for this \emph{unrefined} presentation for every coprime pair, including every $p\ge7$. It does not exclude lens spaces as topological quotients of order-preserving data.

For example, $p=4,q=2$ gives the antipodal map of the octahedral boundary, and the quotient ball is $\mathbb{RP}^3=L(2,1)$. More generally, bisect every equatorial edge of the original coprime presentation. The new source is $B_{2p}$ with shift $2q$, a subdivision of the same gluing. Put the old equatorial vertices first, then the new midpoints, then $n,s$. Every boundary triangle has one vertex of each type, and every gluing preserves type. Thus the refined datum is increasing and its quotient remains $L(p,q)$. Although $(2p,2q)$ is not a reduced lens parameter pair, the quotient is still a lens space.

Exhaustive searches of orders on the $p+2$ boundary vertices give the following finite regression counts, which the proof does not use:
\begin{center}
\begin{tabular}{@{}lll@{}}\toprule
$p$ & admissible $q$ & number of orders, respectively\\\midrule
4 & 2 & 48\\
5 & none & 0\\
6 & 2,3,4 & 528,720,528\\\bottomrule
\end{tabular}
\end{center}
\end{rem}

\begin{ex}[Lens spaces on one source, with distinct attaching data]\label{ex:lens}
The flag ball $B_5$ has face vector $(8,22,25,10)$, hence $65$ nonempty simplices. The rotations $q=1,2$ give coherent permuted quotients with equal face counts
\[
 (f_0,f_1,f_2,f_3)=(3,13,20,10)
\]
and homology $(\mathbb Z,\mathbb Z/5,0,\mathbb Z)$. They are $L(5,1)$ and $L(5,2)$, which are not homotopy equivalent: $2$ is not a signed square modulo $5$~\citep{HAPLenses}. Equality of these counts and homology groups is not equality of cell structures or attachment records.

For $c_0<\cdots<c_4<n<s<v$, the signs of the five triangle bijections are
\[
 q=1:(+,+,+,-,-),\qquad q=2:(+,+,-,+,-).
\]
The target pairings $i\mapsto i+q$ differ too. A polygon rotation between two different triangles is not automatically an even permutation in their sorted local charts. A sign does not specify a general facet bijection because $S_k\to\{\pm1\}$ is not injective for $k\ge3$. The lens example serves to distinguish homology from homotopy type.
\end{ex}

\begin{ex}[Torus]\label{ex:torus}
Label the square $1=$ bottom-left, $2=$ bottom-right, $3=$ top-left, $4=$ top-right, and use diagonal $14$. Its triangles are $124,134$, and its only missing edge is $23$, so it is flag. Glue bottom to top by $1\mapsto3,2\mapsto4$ and left to right by $1\mapsto2,3\mapsto4$. These are increasing maps. The quotient has one vertex $Q$, edges $a=[12]=[34]$, $b=[13]=[24]$, $c=[14]$, and two triangles. This is the usual two-triangle torus $\Delta$-complex~\citep{hatcher2002algebraic}. Its signed trace is in Table~\ref{tab:torus}. An isomorphic source disc with a different gluing datum gives the Klein bottle (Remark~\ref{rem:orientation}); relabelling and transporting one fixed datum cannot change its topology.
\end{ex}

\begin{rem}[An order-preserving Klein-bottle gluing]\label{rem:orientation}
On the cyclically labelled square $1,2,3,4$ with diagonal $24$, glue $12$ to $34$ by $1\mapsto3,2\mapsto4$, and $14$ to $23$ by $1\mapsto2,4\mapsto3$. Both maps are increasing for $1<2<3<4$. Put $a=[12]=[34]$, $b=[14]=[23]$, $c=[24]$. Then
\[
 \partial(124)=a-b+c,\qquad \partial(234)=a+b-c, \qquad
 \partial_2=\begin{pmatrix}1&1\\-1&1\\1&-1\end{pmatrix}.
\]
Its Smith invariants are $1,2$, and $\partial_1=0$, so $H_1=\mathbb Z\oplus\mathbb Z/2$ and $H_2=0$. Every edge has two incident triangle-side occurrences, and the quotient vertex link is a circle of six corner intervals. Thus the quotient is a closed connected surface; the boundary word $aba b^{-1}$ identifies it as the Klein bottle. The fixed maps are increasing in exactly the four orders $1234,2143,3412,4321$. This is a different gluing datum on an isomorphic source disc, which relabelling a fixed datum could not produce. The example uses no order-reversing map, which would belong to the separate permutation regime.
\end{rem}

The next example exhibits a repeated genuine facet and then a composition of gluing with a CW subcomplex collapse.

\begin{ex}[M\"obius band and $\mathbb{RP}^2$; a doubled cover]\label{ex:mobius}
Let $K=\Delta^2$ on $\{1,2,3\}$ (flag) with the gluing datum generated by $(1,2)\sim(2,3)$ via the order-preserving bijection $1\mapsto2$, $2\mapsto3$. Condition (ii) of Definition~\ref{def:gluing} forces $2\sim3$ and $1\sim2$, so the classes are one vertex $P$, edges $a=\{(1,2),(2,3)\}$ and $c=\{(1,3)\}$, and one $2$-cell $F$ with
\[
\partial F=\bigl(a,\ c,\ a\bigr).
\]
The covers $a\lessdot F$ are therefore the two \emph{parallel} edges $e_{F,0}$ and $e_{F,2}$ of $Q(S)$, with signs $(-1)^0=+1$ and $(-1)^2=+1$, so
\[
[a{:}F]=2,\qquad \partial F=2a-c,\qquad \partial a=\partial c=0 ,
\]
giving $H_*=(\mathbb Z,\mathbb Z,0)$ and $\chi=1-2+1=0$. The space is the triangle with two adjacent sides identified in the same rotational sense, boundary word $a\,a\,c^{-1}$: the class $a$ carries two triangle sides and $c$ one, so this is a compact surface with one boundary circle and $\chi=0$, the \emph{M\"obius band}. Collapsing the surviving $1$-cell $c$, the boundary circle of this quotient (a CW subcomplex, to which the topological operation of Definition~\ref{def:sc} applies), leaves cells $(1,1,1)$ with $\partial F=2a$, $\chi=1$ and $H_*=(\mathbb Z,\mathbb Z/2,0)$: the minimal CW structure on $\mathbb{RP}^2$.
\end{ex}

\begin{rem}[incidence numbers separate the two regimes]\label{rem:incidence-regimes}
Under a collapse, distinct facets of a simplex are distinct simplices, so no two covers between \emph{genuine} cells are ever parallel and $[x{:}y]\in\{0,\pm1\}$ throughout: the chain complex looks regular even though the stored presentation is not (Proposition~\ref{prop:regular}). Parallel cover occurrences from component points arise at edges, whose two endpoints may map to the same point, and $C$ discards those anyway. Under a gluing, genuine parallel covers do occur and incidence coefficients can have absolute value greater than one (always bounded by the number of local facet occurrences for a fixed cell) (Example~\ref{ex:mobius}). This is a third respect in which the two regimes differ, besides strict gradedness (Remark~\ref{rem:dichotomy}) and skeletal determinacy (\cref{rem:dichotomy-skel}).
\end{rem}

\section{Complete experimental tables}\label{app:tables}

\subsection{Storage crossover: measurements}\label{sec:crossover-measured}
Concerning the storage crossover of \cref{sec:crossover}, we tested~\eqref{eq:crossover} on three families of flag complexes chosen to span a wide range of $s$, taking the collapsed subcomplex to be an induced subcomplex and recording the achieved simplex fraction rather than a vertex fraction, since an induced subcomplex on a fraction $f$ of the vertices contains roughly $f^{k+1}$ of the $k$-simplices.

\begin{center}
\begin{tabular}{@{}lrrrrrr@{}}
\toprule
family & $n$ & median $\lvert K\rvert$ & $s$ & predicted $\alpha^\ast$ & measured $\alpha^\ast$ & seeds\\
\midrule
Erd\H{o}s--R\'enyi & $1000$ & $5\,081$ & $1.89$ & $0.307$ & $0.305$ & $20$\\
Erd\H{o}s--R\'enyi & $5000$ & $25\,052$ & $1.87$ & $0.302$ & $0.301$ & $20$\\
Erd\H{o}s--R\'enyi & $20\,000$ & $100\,214$ & $1.86$ & $0.301$ & $0.300$ & $20$\\
Erd\H{o}s--R\'enyi & $50\,000$ & $250\,062$ & $1.86$ & $0.301$ & $0.300$ & $20$\\
Vietoris--Rips & $1000$ & $16\,650$ & $3.11$ & $0.513$ & $0.518$ & $20$\\
Vietoris--Rips & $5000$ & $84\,972$ & $3.12$ & $0.515$ & $0.521$ & $20$\\
Vietoris--Rips & $20\,000$ & $342\,979$ & $3.12$ & $0.515$ & $0.523$ & $20$\\
Vietoris--Rips & $50\,000$ & $854\,953$ & $3.12$ & $0.515$ & $0.524$ & $20$\\
uncapped clique & $5000$ & --- & $5.11$ & $0.673$ & $0.677$ & $100$\\
\bottomrule
\end{tabular}
\end{center}

The Vietoris--Rips rows use the $3$-skeleta of Vietoris--Rips complexes of uniform planar point clouds of mean degree $8$ (a dimension-capped clique complex need not itself be flag, so these rows test the counting statements, which hold for arbitrary pairs, rather than the flag-specific theorems). The Erd\H{o}s--R\'enyi rows use independent edges at the same mean degree, which suppresses higher cliques and leaves every source two-dimensional. The last row is uncapped clique complexes of geometric graphs on the flat torus, reaching dimensions $10$ to $14$. Predicted and measured crossovers differ by at most $0.0088$, with median error $0.0033$, across $s$ from $1.86$ to $5.11$ and $\alpha^\ast$ from $0.30$ to $0.68$ (\cref{fig:crossover}). In these experiments the crossover is governed by the dimension distribution alone: at fixed vertex count and mean degree, changing the clique structure alone moves $\alpha^\ast$ by a factor of two.

\begin{figure}[htbp]
\centering
\includegraphics[width=\textwidth]{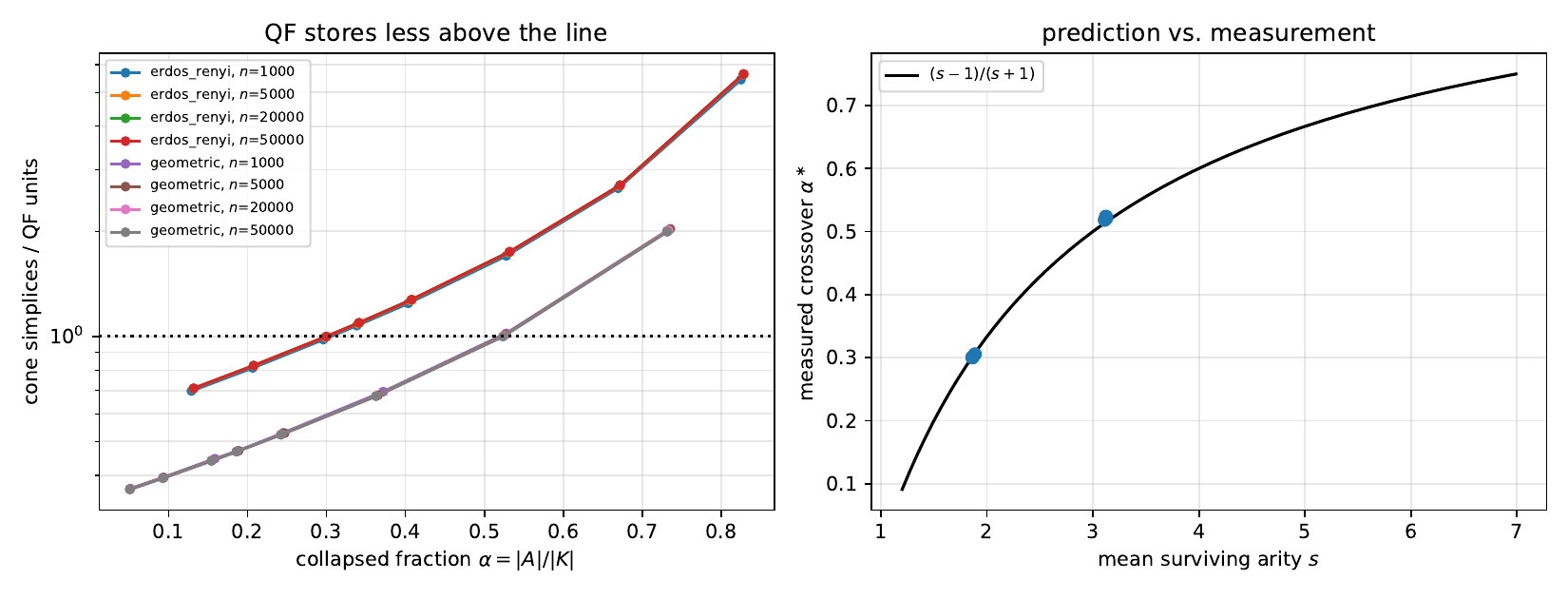}
\caption{Left: measured ratio of cone-model simplices to QF storage units against the collapsed fraction, for the Erd\H{o}s--R\'enyi and Vietoris--Rips families (the latter labelled `geometric' in the legend). Above the dotted line the quotient stores less. Right: the crossover predicted by~\eqref{eq:crossover} against the mean surviving arity, with the measured crossings superimposed. The Erd\H{o}s--R\'enyi crossing is measured at $\alpha^\ast\approx0.30$, in agreement with the prediction.}
\label{fig:crossover}
\end{figure}

\subsection{Static comparison: pilot and replication}

The capped pilot used $n\in\{1000,5000\}$, geometric and Erd\H{o}s--R\'enyi families, three seeds per condition, and five target simplex fractions. There were 60 pairs and 420 successful homology timing rows across seven routes. An exact-intersection experiment contributed 126 further rows on 18 side-pairs from three geometric union graphs. A full clique-complex example with 50,000 vertices, 1,170,307 simplices and dimension 11 contributed 21 homology rows and a separate 12-row topology-only series. All recorded Betti vectors agreed. These are exploratory measurements: there was one timing repetition per route/case, and only one source seed in the largest example.

\begin{table}[ht]
\centering
\caption{Geometric pilot, $n=5000$, dimension cap 3. Median cold compute milliseconds over three seeds. Full QF means the compact, topology-complete representation without the optional word trie.}
\label{tab:qf-gudhi-cold-050}
\begin{tabular}{rrrr}
\hline
Target $\alpha$ & Full QF & \textsc{Gudhi} relative & \textsc{Gudhi} cones\\
\hline
0.20 & 118.55 & 53.95 & 75.57\\
0.50 & 114.34 & 44.13 & 102.30\\
0.54 & 105.00 & 43.56 & 104.03\\
0.80 & 100.72 & 35.43 & 127.67\\
0.95 & 85.04 & 32.37 & 131.70\\
\hline
\end{tabular}
\end{table}

A crossover against component cones is observed between the sampled fractions near 0.54 and 0.80, but not against direct relative \textsc{Gudhi}. This does not identify a universal threshold. With both representations already materialized and the same \textsc{Gudhi} Matrix reducer, the median paired assembly-and-solve speedup from QFTree ranges approximately from 1.4 to 1.7 in these geometric $n=5000$ conditions. The cost of constructing the QF-tree is not included in that conditional comparison.

\paragraph{The full large example.} At $\alpha\simeq0.95$, cold compute took 4.542 seconds for compact full QF and 0.772 seconds for direct relative \textsc{Gudhi}. Feeding an already-materialized QF to \textsc{Gudhi} Matrix required 0.0111 seconds for assembly and solve, versus 0.0437 seconds from the already-materialized SimplexTree pair. The corresponding QFT, source-pair GST, and cone GST files occupied 3.604, 17.857, and 34.823 MiB, respectively. However, the QF construction/computation process peaked near 402 MiB, compared with about 201 MiB for relative \textsc{Gudhi}. The small retained representation therefore did not translate into a smaller cold-process peak. The compact full QF did not beat cone cold compute at any of the three tested fractions on this uncapped example, so the capped-pilot crossover does not transfer to another dimensional regime.

\paragraph{Interpretation.} The advantage supported by these data is a standalone quotient-space API that preserves source-labelled cells, repeated local facet occurrences, quotient maps, and further factorization without triangulating a larger cone model. Its native homology consumer is optional, and \textsc{Gudhi} Matrix can instead be used as a consumer. Sixty stages of nested quotient operations agreed byte-for-byte with the corresponding direct quotients. The immutable update algorithm used here rebuilds the entire table. For homology-only tasks a direct relative matrix needs less information, and repeated identical Betti queries can be cached.

The support-unit crossover serves as an explanatory proxy and does not measure RAM. Exact cell counts satisfy $N_Q=|K|-|A|+c$, $N_C=|K|+|A|+c$, so there is no cell-count crossover. These static experiments motivate evaluating the retained quotient on new operations instead of repeated requests for an unchanged vector of Betti numbers.

\subsubsection*{Independent replication}
\label{sec:user-run-replication}

A separate replication used the same geometric and Erd\H{o}s--R\'enyi families at 1000 and 5000 vertices, with three source seeds per condition. The repeated relative-chain and exact-intersection experiments reuse these sources, so they are not additional independent graph samples. The comparison comprises 135 distinct pairs, including 36 two-complex intersection cases. Timing repeats are summarized within each pair before comparing methods.

The recorded environment is macOS arm64 with Python 3.11.8 and Apple LLVM 21.0.0. Complexes have a maximum construction dimension of three, and the largest input has 90,992 simplices. Thus these observations do not replicate the previous uncapped 1,170,307-simplex stress case.

The relative-chain constructions agree entrywise on all 135 pairs. Their homology agrees with the independent \Gudhi{} Matrix and two-level CAM calculations. The intersection cases use exactly $A=K_1\cap K_2$.

\begin{table}[htbp]
\centering
\caption{$3$-skeleta of Vietoris--Rips complexes on $5000$ points: median cold compute time in ms over three seeds, one timing observation per route and seed (macOS arm64, Python 3.11, Apple LLVM 21). Full QF is the compact, topology-complete representation without the word trie.}
\label{tab:user-cold}
\begin{tabular}{r|rrrr}
\hline
Target $\alpha$ & Full QF & Relative \Gudhi & CAM & Component cones\\
\hline
0.20 & 84.6 & 56.7 & 65.9 & 79.1\\
0.50 & 79.6 & 51.4 & 66.1 & 97.3\\
0.54 & 69.0 & 50.4 & 63.8 & 109.1\\
0.80 & 57.1 & 45.4 & 72.2 & 110.4\\
0.95 & 55.1 & 41.4 & 78.9 & 126.8\\
\hline
\end{tabular}
\end{table}

A QF/cone time crossover is bracketed by the sampled actual median fractions $\alpha\simeq0.20006$ and $\alpha\simeq0.50025$. This is not a universal threshold. At target fraction 0.95, the median of paired cone/QF time ratios is 2.619, whereas direct relative \textsc{Gudhi} remains faster than full QF construction and solution. A median of paired ratios is not generally the ratio of the marginal medians in Table~\ref{tab:user-cold}.

For the same \textsc{Gudhi} Matrix solver applied to chains from ready presentations, the repeated geometric series on 5,000 vertices gives median relative-pair/QF ratios 0.979, 1.036, 1.043, 1.164, and 1.073 at the five sampled fractions. The gain is not uniform: the geometric 1,000-vertex conditions are slower along the recorded QF path. These measurements include the first use of the homology consumer, with no initialization cost subtracted.

Compared with the pilot above, the replication moves the QF/cone crossover bracket from between the sampled fractions near $0.54$ and $0.80$ to between $0.20$ and $0.50$, and lowers the ready-presentation ratio from approximately $1.4$--$1.7$ to the values above. Both series have three seeds per condition. We report them side by side instead of pooling them, and neither bracket should be expected to transfer to another platform.

For two factors by an exact common intersection, we sum the two side costs before comparing methods. On 5,000-vertex geometric graphs, the median full-QF costs are 38.229, 61.301, and 86.998 ms at shared-edge parameters $\eta=0.2,0.5,0.8$, versus 24.469, 40.916, and 64.369 ms for two relative \textsc{Gudhi} computations. For ready QFs and the identical \textsc{Gudhi} solver, the paired ratios are instead 1.325, 1.367, and 1.343. Here $\eta$ is an edge-sharing parameter, which differs from the simplex fraction $|A|/|K_i|$. These timings exclude the preparation of the common intersection itself.

Serialized size and process memory behave differently. At target $\alpha=0.95$ on the 5,000-vertex geometric inputs, median file sizes are 0.234 MiB for the QFT file, 1.278 MiB for the serialized \textsc{Gudhi} simplex tree of the pair with membership in $A$ marked (pair GST), and 2.492 MiB for that of the cone model (cone GST). Nevertheless, median compute-peak RSS is 55.938 MiB for compact QF versus 45.016 MiB for direct relative \textsc{Gudhi}. Across the archived runs, the maximum compute peak is 71.859 MiB and the maximum process-lifetime peak is 81.266 MiB. These figures exclude the generator and total system memory.

The source-free sphere example independently verifies the presentation functionality: four distinct edges share one index word, but remain distinct cells. Its cell counts are $(2,4,4)$ and Betti vector $(1,0,1)$. Collapsing the loaded object's entire one-skeleton produces cell counts $(1,0,4)$ and Betti vector $(1,0,4)$ with a validated quotient map. The four $2$-cells are then attached entirely to the basepoint, so the loaded object lies outside the strictly graded class of Theorem~\ref{thm:strict}. The encoding represents it regardless (Remark~\ref{rem:skeletal-filtration}). The example thus exercises operations on a retained quotient presentation beyond its homology.

The supported claim is therefore representation- and workload-specific: QFTree preserves a reusable quotient presentation, avoids cone inflation, and can reduce serialized size and later computation costs. These data do not establish universally faster one-shot relative homology or lower peak RAM than \textsc{Gudhi}. Entrywise equality of the relative matrices also rules out attributing the advantage to a necessarily smaller relative matrix.

\subsection{Closed-star budgets: measurements}\label{sec:local-measured}
For local collapses we measured the three budgets on $3$-skeleta of Vietoris--Rips complexes on $5000$ uniform points of the flat torus at mean degree $8$ (five seeds generated by \texttt{tools/star\_locality.py}, median $85\,031$ simplices, different from the instances of \cref{sec:crossover-measured}), with $A$ induced either on the points of a ball of radius $\rho$ or on a uniformly random vertex subset. Every local presentation was checked record by record against the full QF-tree, the assembled chain complexes column by column against the direct ones, and the compact presentation was checked to reproduce the Betti numbers without the source pair.

\begin{center}
\footnotesize
\setlength{\tabcolsep}{4pt}
\begin{tabular}{@{}lrrrrrrr@{}}
\toprule
$A$ induced on & param. & $\alpha=|A|/|K|$ & $|H|/|K|$ & ideal/cone & compact/cone & retained/cone & full QF/cone\\
\midrule
ball & $\rho=0.05$ & $0.005$ & $0.004$ & $0.998$ & $0.999$ & $1.007$ & $2.95$\\
ball & $\rho=0.10$ & $0.031$ & $0.009$ & $0.958$ & $0.961$ & $0.996$ & $2.81$\\
ball & $\rho=0.15$ & $0.062$ & $0.014$ & $0.908$ & $0.912$ & $0.980$ & $2.64$\\
ball & $\rho=0.20$ & $0.131$ & $0.020$ & $0.817$ & $0.821$ & $0.935$ & $2.29$\\
ball & $\rho=0.30$ & $0.278$ & $0.026$ & $0.605$ & $0.610$ & $0.844$ & $1.69$\\
ball & $\rho=0.40$ & $0.500$ & $0.034$ & $0.397$ & $0.403$ & $0.749$ & $0.99$\\
ball & $\rho=0.50$ & $0.769$ & $0.039$ & $0.174$ & $0.179$ & $0.631$ & $0.39$\\
\midrule
random vertices & $1\%$ & $0.001$ & $0.060$ & $1.11$ & $1.14$ & $1.17$ & $2.98$\\
random vertices & $5\%$ & $0.004$ & $0.256$ & $1.47$ & $1.58$ & $1.72$ & $2.97$\\
random vertices & $10\%$ & $0.009$ & $0.425$ & $1.77$ & $1.95$ & $2.20$ & $2.94$\\
random vertices & $20\%$ & $0.024$ & $0.679$ & $2.25$ & $2.46$ & $2.93$ & $2.87$\\
random vertices & $30\%$ & $0.050$ & $0.785$ & $2.38$ & $2.57$ & $3.17$ & $2.75$\\
random vertices & $50\%$ & $0.156$ & $0.813$ & $2.21$ & $2.31$ & $3.04$ & $2.27$\\
\bottomrule
\end{tabular}
\end{center}

Ratios are medians over seeds of paired ratios. For a ball, the median compact budget is below the cone budget at every sampled radius (from $0.999$ at $\alpha=0.005$ to $0.18$ at $\alpha=0.77$), though not in every individual case: two of the five seeds at $\rho=0.05$ lie above it. The median shell fraction is at most about $3.9\%$, while the largest individual shell fraction is about $6.41\%$. The compact budget carries a median overhead of about $2.69\%$ over the ideal one at $\rho=0.50$, reaching $3.13\%$ in an individual case. The two curves nonetheless nearly coincide in the figure, which normalizes by the cone budget: their absolute gap there is below $0.00654$ in every sampled ball case. The retained-source budget, which is what the default object holds, crosses the cone budget between $\rho=0.05$ and $\rho=0.10$. The full QF-tree crosses at $\alpha\approx0.5$, as~\eqref{eq:crossover} predicts for $s\approx3.1$. For a scattered vertex set the shell is already a large part of the complex at ten percent of the vertices (median shell fraction about $42.5\%$) and most of it at twenty percent (about $67.9\%$), and no closed-star budget is ever below the cone budget. \Cref{fig:star-locality} shows the ideal and compact budgets against the full QF-tree.

\begin{figure}[htbp]
\centering
\includegraphics[width=.9\textwidth]{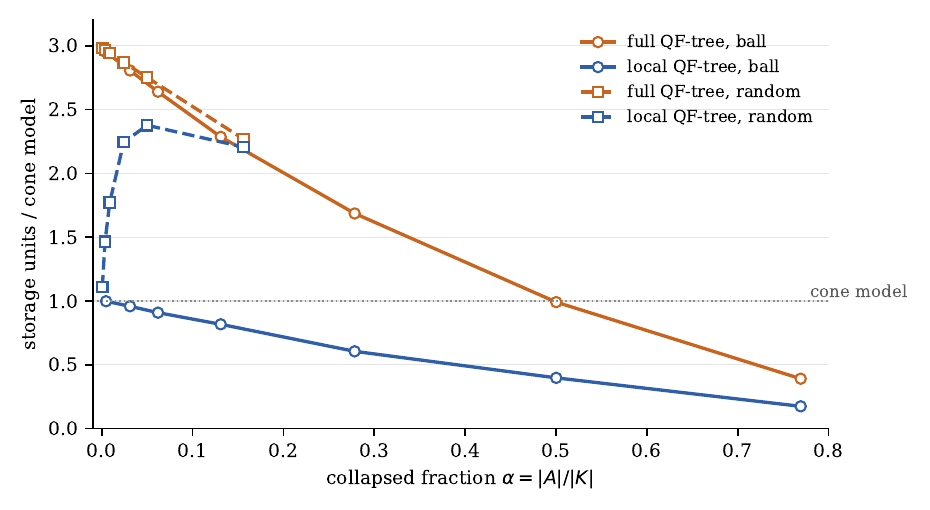}
\caption{Storage units relative to the cone model for the full QF-tree and for the closed-star presentation under the ideal and compact conventions, with the collapsed subcomplex induced on a ball (solid) or on a random vertex subset (dashed). Medians over five seeds. Below the dotted line the representation is the smaller one.}
\label{fig:star-locality}
\end{figure}

In this sampling regime the median compact budget of the closed-star presentation is below the cone budget at every sampled ball radius. For these medians the presentation removes the crossover of \cref{sec:crossover} for collapses that arise from a region of a point cloud, while it gives no saving for a collapsed subcomplex spread evenly over the complex. A record-sharing hybrid on a live simplex tree would reach the ideal budget. The library's compact object exceeds that budget by at most $3.1\%$ on these inputs (median $2.7\%$ at $\rho=0.50$, less at smaller radii), which is less than one percent of the cone budget. The retained-source object exceeds the ideal budget by a median factor that grows from $1.01$ at $\rho=0.05$ to $3.6$ at $\rho=0.50$.

\subsection{Design and endpoints of the operation experiments}\label{sec:qfe070design}
The operation experiments use periodic triangulations $K_s$ of the $2$-torus on an $s\times s$ grid of vertices, with $(f_0,f_1,f_2)=(s^2,3s^2,2s^2)$ and $|K_s|=6s^2$, where $s$ is the grid side (unrelated to the mean arity of \cref{sec:crossover}). Contracting a maximal spanning tree gives the initial quotient $Q_s$ with $(1,2s^2+1,2s^2)$ cells, which at $s=48$ means $13\,824$ source simplices and $9218$ quotient cells. A pilot uses $s\in\{4,12,24\}$ and twelve operations per sequence over all four tasks (quotients, edits, rings, composition). An extended design uses $s\in\{12,48\}$ and thirty-two operations for quotients and edits. Each condition has three seed indices, which vary vertex numbering and the chosen operations but not the topological type, and three timing repetitions. Reported times are medians over seeds of within-seed medians. Paired ratios are formed before aggregation, and error bars show the seed range. The measurements were made on macOS arm64 with Python 3.11.8 and single-thread OpenMP/OpenBLAS settings, with the same build for all matched comparisons within a series. The processor model and native compiler configuration were not recorded, which limits hardware-specific reproducibility. They evaluate the 0.7.0 computational kernels, which later releases retain. One extended execution has complete mathematical outputs and internal stage timings but no external wall-time or sampled-memory measurement, so it contributes only to the observed endpoints.

Correctness is checked at the level of each output: boundary matrices, cycle representatives, induced maps, kernels and cokernels for quotients; complete zigzag barcodes and final cell sets for edits; basis-independent invariants and the fan-model pullback for rings. All compared outputs agree.

Three quotient implementations are compared: the editable local QF, the immutable QF-tree rebuilt after each change, and a persistent \Gudhi{} simplex tree with an updated subcomplex mask, built once. All feed the same native homology-and-map consumer. For edits the comparators are local maintenance, reconstruction of the same cell table, and a \Gudhi{} simplex tree, all feeding \Gudhi's streaming zigzag, so that the comparison isolates representation maintenance from the zigzag algorithm. The timing endpoints are $T_{\mathrm{top}}$ (topology update only), $T_{\mathrm{ops}}$ (copy, topology and algebra per operation, summed), and $T_{\mathrm{cold}}$ (initialization plus $T_{\mathrm{ops}}$).

\paragraph{Routes in the static comparison.} The tasks are storing the quotient as a standalone presentation, computing relative Betti numbers of a pair, and constructing the cone model. The routes are: full QF-trees with and without their word index, feeding either the native $\mathbb F_2$ reducer or \Gudhi's persistence-matrix module (\Gudhi\ Matrix); relative chains assembled by traversing a \Gudhi\ simplex tree and reduced by \Gudhi\ Matrix (\emph{relative \Gudhi}); a two-level computation on $A\subseteq K$ with \Gudhi's persistent cohomology, based on the compressed annotation matrix (CAM) of~\citep{boissonnat2015cam}, using $\beta_d(K,A)=\beta_d(K)-r_d+\beta_{d-1}(A)-r_{d-1}$ with $r_d$ the rank of $H_d(A)\to H_d(K)$; and the component-cone simplex tree reduced by CAM. All use \Gudhi\ 3.13.0 and $\mathbb F_2$, include the top dimension, and avoid per-simplex Python loops. Cold compute time includes source import and validation, quotient or cone construction where applicable, boundary assembly and the solver, and excludes input generation, process startup and file reading.

\begin{figure}[tbp]
\centering
\includegraphics[width=.86\linewidth]{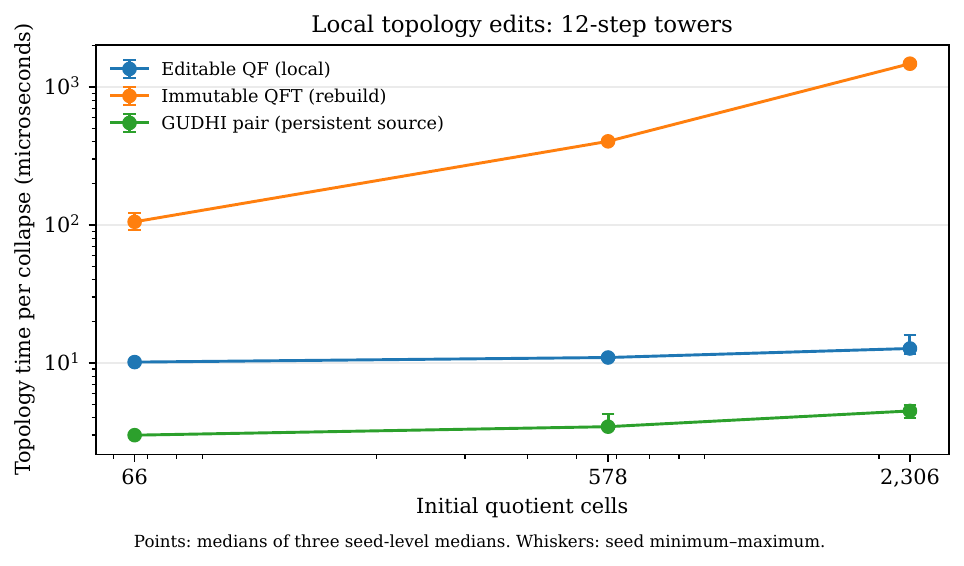}\\[1.2ex]
\includegraphics[width=.86\linewidth]{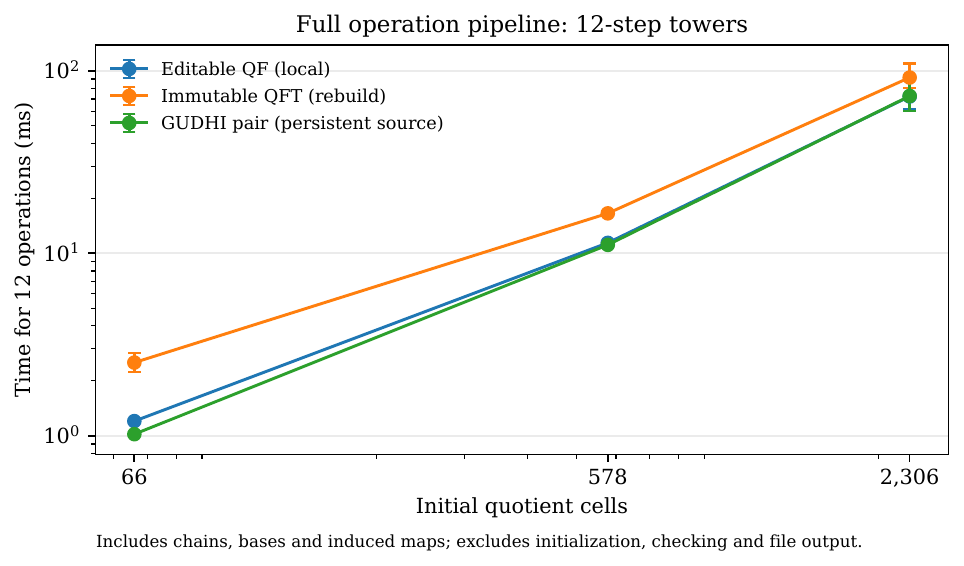}
\caption{Local quotient updates on the torus family. Above: topology time per collapse, which is flat in the quotient size for the editable representation and grows for the immutable rebuild. Below: the same sequences timed end to end, including chains, bases and induced maps and excluding initialization, checking and file output. The topology saving is large, while the complete-operation saving is small.}
\label{fig:qfe070local}
\end{figure}

\begin{figure}[tbp]
 \centering\includegraphics[width=.96\linewidth]{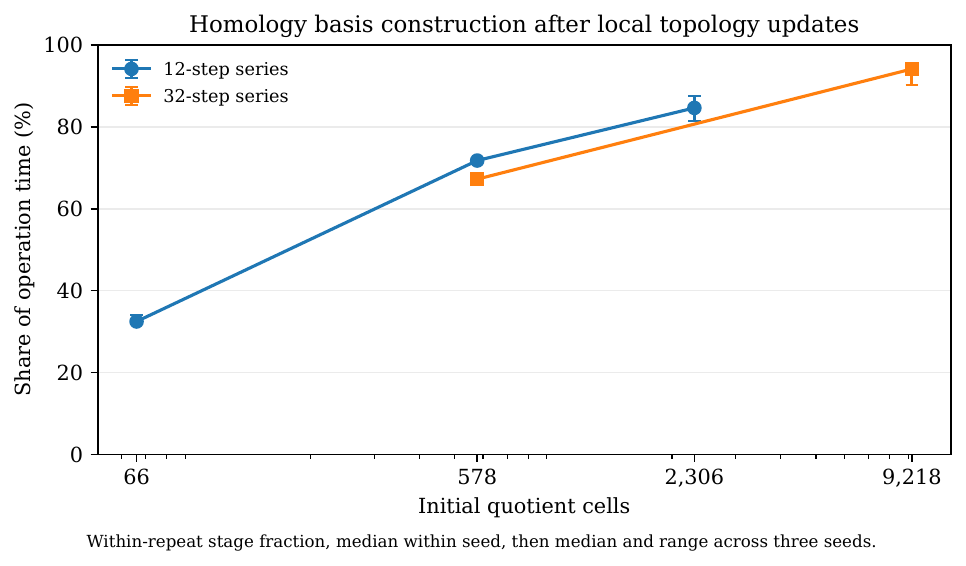}
 \caption{Fraction of local-QF operation time used by homology-basis construction. The two transcript lengths are shown separately, even at a shared source size.}
 \label{fig:qfe070basis}
\end{figure}

\begin{figure}[tbp]
\centering\includegraphics[width=.96\linewidth]{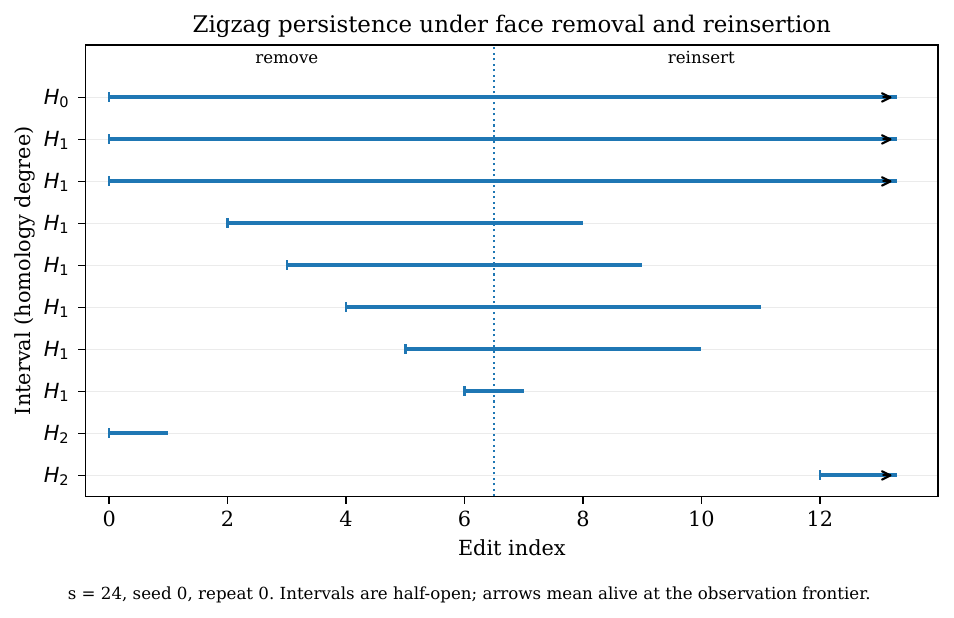}
\caption{Zigzag intervals for six face removals followed by six reinsertions on the $s=24$ torus. The two original degree-one classes persist; finite intervals identify classes created by missing faces. Rightward arrows indicate survival at the end of the observed sequence.}
\label{fig:qfe070barcode}
\end{figure}

\begin{figure}[tbp]
\centering\includegraphics[width=.96\linewidth]{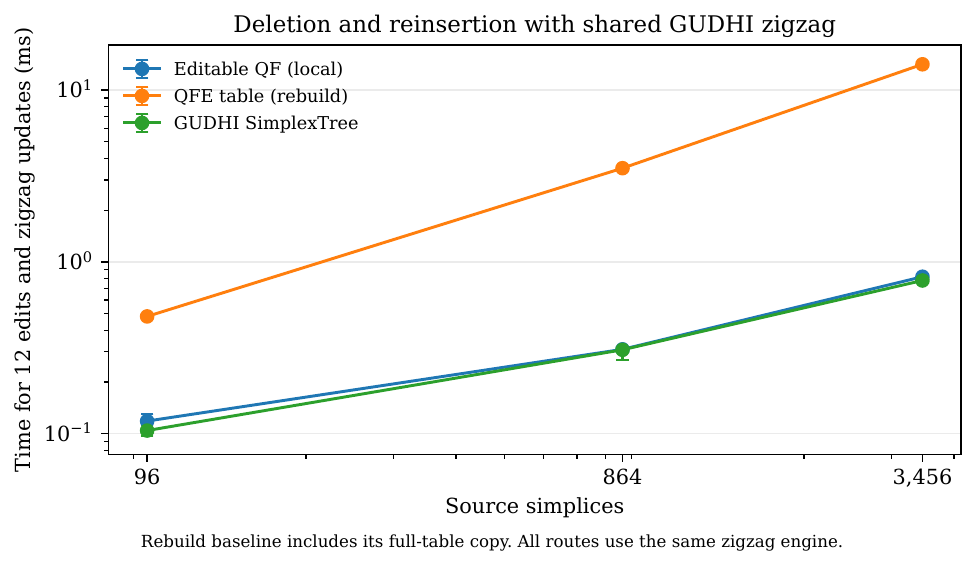}
\caption{Insertion/deletion cost with a common \Gudhi\ streaming zigzag engine. The reconstruction baseline includes its full-table copy. Initialization and independent validation are excluded from the plotted operation time.}
\label{fig:qfe070edittime}
\end{figure}

\begin{figure}[tbp]
\centering
\includegraphics[width=.86\linewidth]{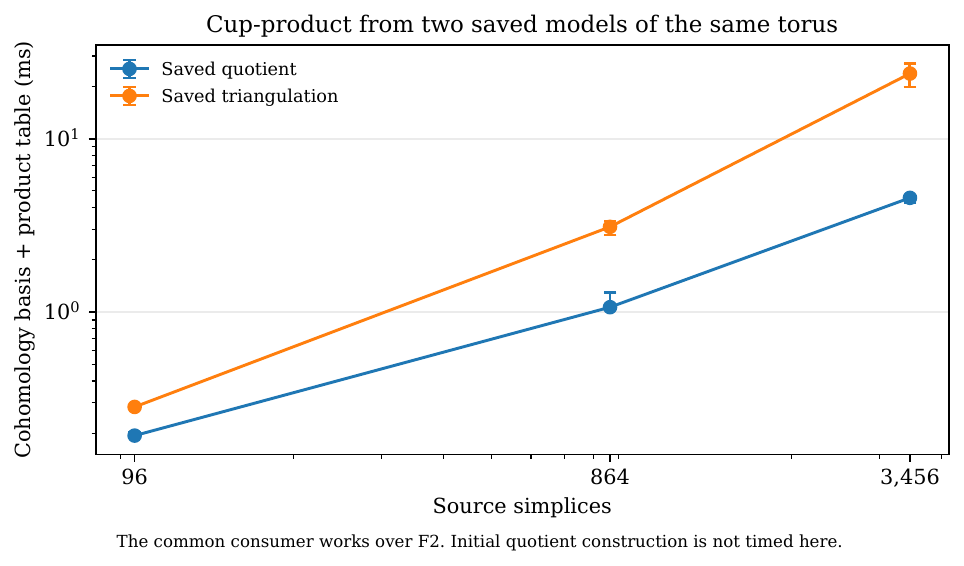}\\[1.2ex]
\includegraphics[width=.86\linewidth]{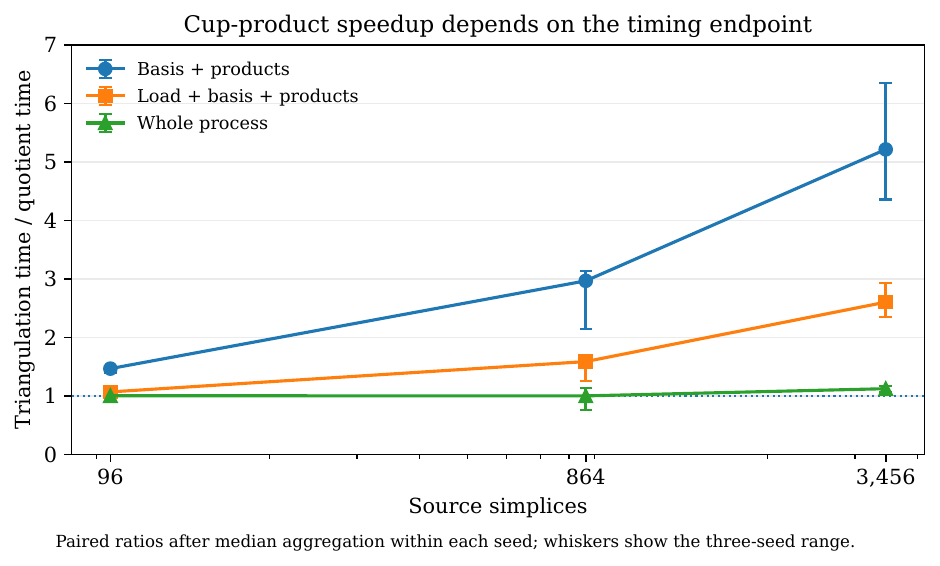}
\caption{Cohomology basis and cup products from two saved representations of the same torus. Above: absolute time. The original construction of the quotient is not part of this saved-object task. Below: the paired speedup at three timing endpoints, showing how much of it survives once loading, and then the whole process, are charged. Ratios are formed after median aggregation within each seed; whiskers show the three-seed range.}
\label{fig:qfe070cup}
\end{figure}

\begin{figure}[tbp]
\centering\includegraphics[width=\linewidth]{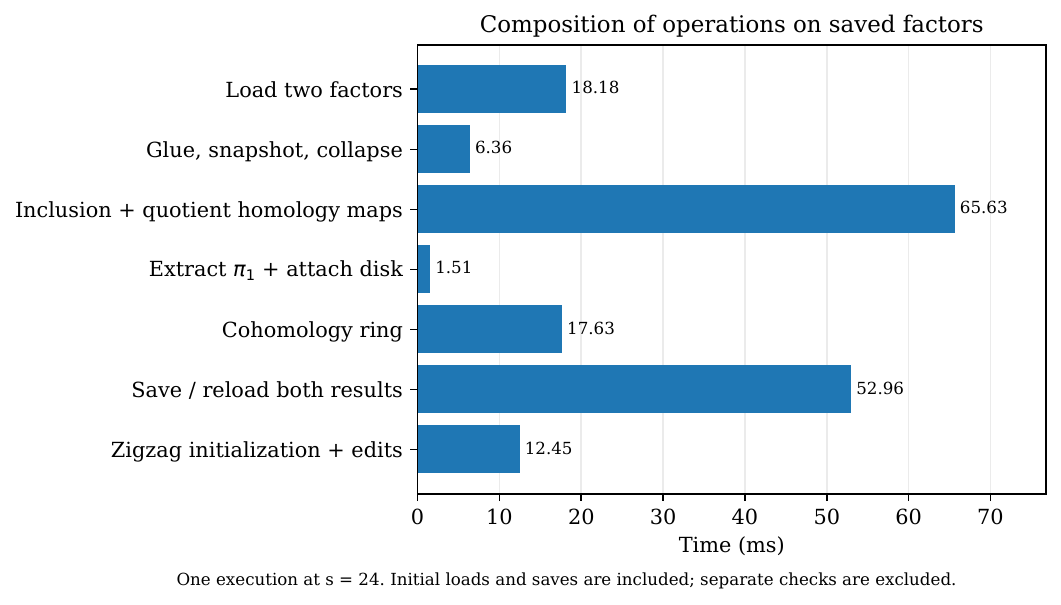}
\caption{Stage costs for one source-independent composition at $s=24$. The measured operations include loading and saving but exclude separate correctness checks.}
\label{fig:qfe070workflow}
\end{figure}

\subsection{Operation experiments: complete timing and memory tables}

Tables~\ref{tab:qfe070collapsepilot} and~\ref{tab:qfe070collapseextended} list every measured size, shape, and collapse route. The branch-copy cost in the editable route is charged explicitly. Zeros in the copy column of other routes mean no separately timed copy stage, not that those implementations allocate no memory: immutable reconstruction and mask copying occur in their own measured operations. Tables~\ref{tab:qfe070edits}, \ref{tab:qfe070cup}, and \ref{tab:qfe070workflow} give the remaining suites. Units and starting conditions are part of each caption.

\begin{table}[htbp]
\centering\small
\caption{Complete pilot collapse measurements: hierarchical medians.}
\label{tab:qfe070collapsepilot}
\begin{tabular}{rllrrrrr}
\toprule
$s$ & Shape & Route & $T_{\rm top}$ & $T_{\rm copy}$ & $T_{\rm ops}$ & $T_{\rm cold}$ & RSS\\
\midrule
4 & T & E & 0.121 & 0.000 & 1.202 & 2.830 & 40.86 \\
4 & T & Q & 1.262 & 0.000 & 2.517 & 4.122 & 41.94 \\
4 & T & G & 0.036 & 0.000 & 1.020 & 2.475 & 41.00 \\
4 & B & E & 0.109 & 0.143 & 1.401 & 2.954 & 40.31 \\
4 & B & Q & 1.147 & 0.000 & 2.361 & 3.868 & 40.98 \\
4 & B & G & 0.037 & 0.000 & 1.041 & 2.474 & 40.25 \\
12 & T & E & 0.131 & 0.000 & 11.395 & 14.666 & 40.73 \\
12 & T & Q & 4.836 & 0.000 & 16.533 & 19.738 & 41.16 \\
12 & T & G & 0.041 & 0.000 & 11.108 & 13.874 & 40.72 \\
12 & B & E & 0.124 & 0.909 & 13.328 & 16.631 & 41.59 \\
12 & B & Q & 4.874 & 0.000 & 17.573 & 20.792 & 41.34 \\
12 & B & G & 0.048 & 0.000 & 12.164 & 14.942 & 40.73 \\
24 & T & E & 0.152 & 0.000 & 72.254 & 83.770 & 43.02 \\
24 & T & Q & 17.675 & 0.000 & 91.746 & 102.424 & 42.77 \\
24 & T & G & 0.054 & 0.000 & 72.663 & 82.322 & 42.16 \\
24 & B & E & 0.142 & 3.667 & 81.130 & 92.748 & 46.39 \\
24 & B & Q & 17.720 & 0.000 & 96.228 & 106.945 & 43.58 \\
24 & B & G & 0.054 & 0.000 & 76.630 & 86.352 & 42.81 \\
\bottomrule
\end{tabular}
\par\smallskip\begin{minipage}{\linewidth}\footnotesize Times are milliseconds for the complete transcript; RSS is process peak MiB. E: editable local QF; Q: immutable QFT rebuild; G: persistent \Gudhi{} pair. T: tower; B: branching. There are 12 operations per transcript. Each entry is the median over three seed-specific medians of three timing repeats. Paired ratios are given in Table~\ref{tab:qfe070ratios}.\end{minipage}
\end{table}

\begin{table}[htbp]
\centering\small
\caption{Complete extended collapse measurements: hierarchical medians.}
\label{tab:qfe070collapseextended}
\begin{tabular}{rllrrrrr}
\toprule
$s$ & Shape & Route & $T_{\rm top}$ & $T_{\rm copy}$ & $T_{\rm ops}$ & $T_{\rm cold}$ & RSS\\
\midrule
12 & T & E & 0.211 & 0.000 & 16.804 & 18.941 & 41.38 \\
12 & T & Q & 12.519 & 0.000 & 39.858 & 42.966 & 41.19 \\
12 & T & G & 0.076 & 0.000 & 16.716 & 18.515 & 41.20 \\
12 & B & E & 0.283 & 1.848 & 23.649 & 25.878 & 43.48 \\
12 & B & Q & 12.712 & 0.000 & 44.127 & 47.144 & 41.73 \\
12 & B & G & 0.083 & 0.000 & 20.521 & 22.354 & 41.45 \\
48 & T & E & 0.444 & 0.000 & 1239.008 & 1308.649 & 53.70 \\
48 & T & Q & 191.293 & 0.000 & 1763.450 & 1839.094 & 51.16 \\
48 & T & G & 0.189 & 0.000 & 1221.002 & 1288.074 & 49.08 \\
48 & B & E & 0.406 & 27.359 & 1745.879 & 1815.742 & 90.17 \\
48 & B & Q & 192.658 & 0.000 & 2220.561 & 2294.599 & 70.67 \\
48 & B & G & 0.251 & 0.000 & 2770.530 & 2876.843 & 59.69 \\
\bottomrule
\end{tabular}
\par\smallskip\begin{minipage}{\linewidth}\footnotesize Times are milliseconds for the complete transcript; RSS is process peak MiB. E: editable local QF; Q: immutable QFT rebuild; G: persistent \Gudhi{} pair. T: tower; B: branching. There are 32 operations per transcript. Each entry is the median over three seed-specific medians of three timing repeats. Paired ratios are given in Table~\ref{tab:qfe070ratios}.\end{minipage}
\end{table}

\begin{table}[htbp]
\centering\small
\caption{Paired collapse speedups: baseline time divided by local-QF time.}
\label{tab:qfe070ratios}
\begin{tabular}{lrllll}
\toprule
Series & $s$ & Shape & Baseline & Topology ratio & Operation ratio\\
\midrule
P & 24 & T & Q & 115.633 [92.198, 128.168] & 1.270 [1.205, 1.309] \\
P & 24 & T & G & 0.388 [0.251, 0.391] & 0.997 [0.986, 1.006] \\
P & 24 & B & Q & 125.096 [124.711, 131.380] & 1.186 [1.151, 1.221] \\
P & 24 & B & G & 0.393 [0.355, 0.401] & 0.945 [0.939, 0.958] \\
E & 48 & T & Q & 369.351 [306.589, 431.257] & 1.172 [1.070, 1.770] \\
E & 48 & T & G & 0.365 [0.349, 0.510] & 0.985 [0.622, 1.010] \\
E & 48 & B & Q & 474.623 [245.943, 530.861] & 1.709 [0.667, 1.785] \\
E & 48 & B & G & 0.618 [0.550, 0.659] & 1.587 [0.925, 1.591] \\
\bottomrule
\end{tabular}
\par\smallskip\begin{minipage}{\linewidth}\footnotesize Entries are median [minimum, maximum] over three seed-specific ratios, after taking the median of the timing repeats within each seed. P: 12-step pilot; E: 32-step extension. Values above one favor local QF.\end{minipage}
\end{table}

Figure~\ref{fig:qfe070largeratios} displays the seed ranges in the larger series. At $s=12$ with branches, the GUDHI/local operation ratio ranges from 0.562 to 1.363; at $s=48$ with branches it ranges from 0.925 to 1.591. Both ranges cross one. Three timing repeats and three seed indices do not resolve these changes into an implementation-independent ranking. The locality counts are more direct in this family: all local collapse steps visit four records and touch six occurrences, regardless of the initial quotient size.

\begin{figure}[tbp]
 \centering\includegraphics[width=.98\linewidth]{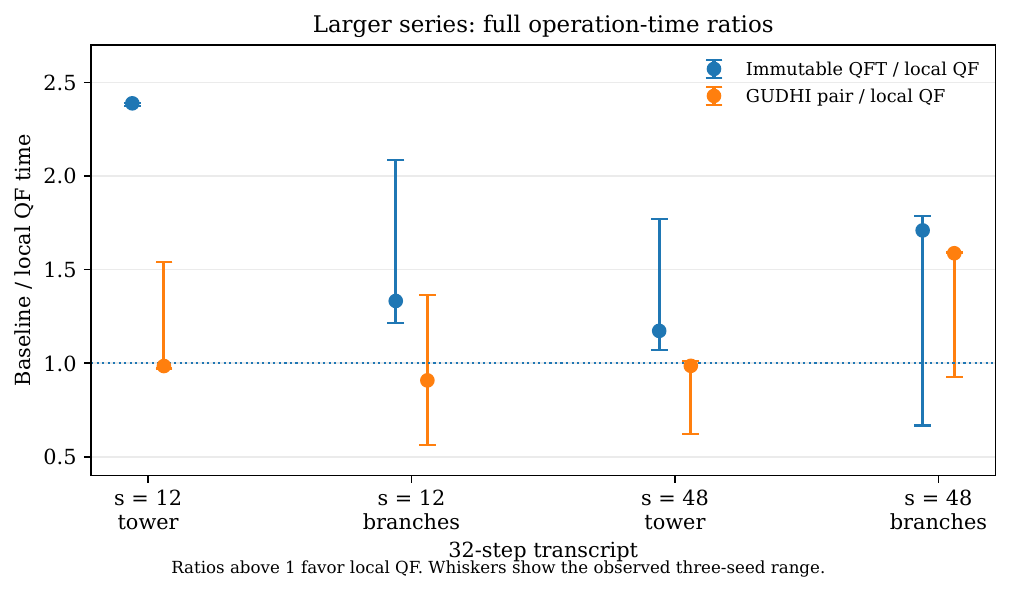}
 \caption{Paired complete-operation ratios for the separately configured 32-step series. Points are medians over seeds after within-seed timing medians; whiskers are the observed seed range.}
 \label{fig:qfe070largeratios}
\end{figure}

Peak resident memory is reported separately from serialized size. Neither measures only the data structure, and neither includes the whole generator/parent-process workload. At $s=48$, local-QF branching uses a median 90.17 MiB of process RSS, versus 70.67 MiB for immutable QFT and 59.69 MiB for the pair; the corresponding tower medians are 53.70, 51.16, and 49.08 MiB. Figure~\ref{fig:qfe070memory} shows no general reduction of peak process memory from local topology maintenance. The maximum observed process memory is 94.61 MiB.

\begin{figure}[tbp]
 \centering\includegraphics[width=.96\linewidth]{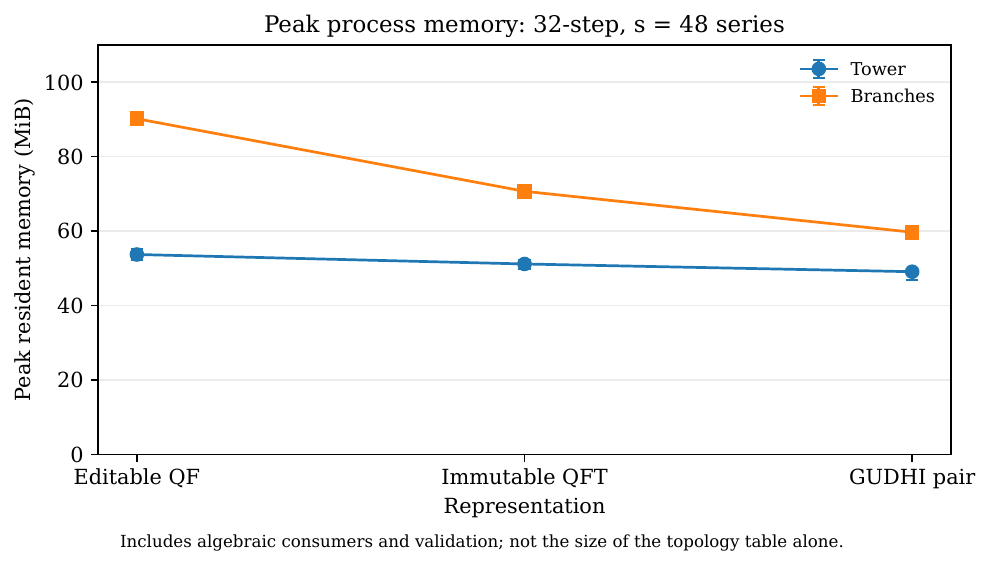}
 \caption{Process high-water memory at the largest measured size. Live branch snapshots and algebraic objects contribute, so this measures neither the native table alone nor the serialized file size.}
 \label{fig:qfe070memory}
\end{figure}

\begin{table}[htbp]
\centering\small
\caption{Insertion/deletion pipelines with a common GUDHI zigzag engine.}
\label{tab:qfe070edits}
\begin{tabular}{lrlrrrrr}
\toprule
Series & $s$ & Route & $T_{\rm top}$ & $T_{\rm copy}$ & $T_{\rm zz}$ & $T_{\rm ops}$ & $T_{\rm cold}$\\
\midrule
P & 4 & E & 0.041 & 0.000 & 0.077 & 0.118 & 1.911 \\
P & 4 & R & 0.032 & 0.371 & 0.082 & 0.481 & 2.283 \\
P & 4 & S & 0.027 & 0.000 & 0.078 & 0.104 & 1.886 \\
P & 12 & E & 0.044 & 0.000 & 0.255 & 0.309 & 6.768 \\
P & 12 & R & 0.036 & 3.213 & 0.259 & 3.509 & 10.076 \\
P & 12 & S & 0.028 & 0.000 & 0.278 & 0.308 & 6.643 \\
P & 24 & E & 0.056 & 0.000 & 0.749 & 0.819 & 25.624 \\
P & 24 & R & 0.043 & 13.324 & 0.757 & 14.124 & 39.129 \\
P & 24 & S & 0.030 & 0.000 & 0.748 & 0.779 & 25.431 \\
E & 12 & E & 0.085 & 0.000 & 0.367 & 0.455 & 6.827 \\
E & 12 & R & 0.060 & 5.737 & 0.407 & 6.204 & 10.425 \\
E & 12 & S & 0.039 & 0.000 & 0.346 & 0.386 & 4.318 \\
E & 48 & E & 0.161 & 0.000 & 5.110 & 5.288 & 130.804 \\
E & 48 & R & 0.125 & 92.728 & 4.289 & 96.540 & 174.850 \\
E & 48 & S & 0.072 & 0.000 & 4.138 & 4.195 & 91.540 \\
\bottomrule
\end{tabular}
\par\smallskip\begin{minipage}{\linewidth}\footnotesize Times are milliseconds. P: 12-edit pilot; E: 32-edit extension. E (route): local editable QF; R: rebuilding the same QFE table before each edit; S: \Gudhi{} simplex tree. The edit is applied to the source triangulation, with $6s^2$ cells, not to the $4s^2+2$-cell quotient. The rebuild cost appears in the copy column and is included in the operation total.\end{minipage}
\end{table}

\begin{table}[htbp]
\centering\small
\caption{Cohomology rings from saved representations of the same torus.}
\label{tab:qfe070cup}
\begin{tabular}{rlrrrrrr}
\toprule
$s$ & Model & Cells & Load & Basis & Products & $T_{\rm ops}$ & $T_{\rm cold}$\\
\midrule
4 & Quotient & 66 & 2.269 & 0.141 & 0.056 & 0.194 & 2.468 \\
4 & Source & 96 & 2.357 & 0.230 & 0.053 & 0.283 & 2.637 \\
12 & Quotient & 578 & 4.024 & 0.917 & 0.157 & 1.066 & 5.069 \\
12 & Source & 864 & 4.975 & 2.942 & 0.160 & 3.100 & 8.043 \\
24 & Quotient & 2306 & 9.687 & 4.136 & 0.395 & 4.558 & 14.252 \\
24 & Source & 3456 & 13.601 & 23.246 & 0.516 & 23.762 & 37.367 \\
\bottomrule
\end{tabular}
\par\smallskip\begin{minipage}{\linewidth}\footnotesize Times are milliseconds. Both representations use the same native consumer. Quotient construction during data preparation is excluded; the load column is measured. The table of products covers positive-degree basis pairs; degree-zero units are regression-tested separately.\end{minipage}
\end{table}

\begin{table}[htbp]
\centering\small
\caption{Source-independent composition of point gluing, quotienting and disc attachment. Times are in milliseconds and memory is peak process MiB.}
\label{tab:qfe070workflow}
\begin{tabular}{rrrrr}
\toprule
$s$ & Final cells & Subsequent operations & Including initial loads & Memory\\
\midrule
4 & 130 & 7.715 & 11.058 & 41.09\\
12 & 1154 & 37.681 & 44.461 & 45.09\\
24 & 4610 & 168.005 & 185.971 & 60.94\\
\bottomrule
\end{tabular}
\par\smallskip\begin{minipage}{\linewidth}\footnotesize The computation includes topology, induced maps, ring calculations and measured saves and reloads. Independent correctness checks are outside the timing endpoint but included in process memory. Entries are medians of the seed-specific medians, not additive stage measurements.
\end{minipage}
\end{table}

The cup-product speedup depends strongly on the endpoint (the lower panel of \cref{fig:qfe070cup}). The $s=24$ ring-consumer ratio of 5.214 becomes 2.603 after loading and 1.123 at the whole-process boundary. The latter contains initialization, verification and output work, which a long-lived application could handle differently. The three endpoints measure different things and are not interchangeable.

Figure~\ref{fig:qfe070basis} records the dominant remaining consumer cost. Within each execution, the homology-basis fraction is formed before aggregating repeats and seeds. It is not obtained by dividing two separately aggregated stage medians. At $s=48$ the three local tower fractions are approximately 95.6\%, 90.3\%, and 94.1\%. Improving topology maintenance alone cannot remove that measured basis-construction cost.

\clearpage

\end{document}